\documentclass{article}
\usepackage[utf8]{inputenc}
\usepackage{graphicx} 
\usepackage{amsmath,mathtools,amsthm,fancyhdr,amssymb,bbm, enumerate,float, mathrsfs}
\usepackage{tcolorbox}
\usepackage{appendix}
\tcbuselibrary{breakable}
\usepackage{listings}
\usepackage{color, comment}
\usepackage{subcaption}
\usepackage{accents}
\usepackage{booktabs}
\usepackage{tabularx}
\usepackage{array}
\usepackage{hyperref}
\hypersetup{
colorlinks,
linkcolor=blue,
anchorcolor=blue,
citecolor=blue
}

\mathtoolsset{showonlyrefs,showmanualtags} 

\newcommand{\reall}{\mathbb{R}}

\newcommand{\x}{\mathbf{x}}
\newcommand{\y}{\mathbf{y}}

\newcommand{\one}{\mathbbm{1}}
\newcommand{\scl}[2]{\left\langle #1\,,\, #2\right\rangle} 
\newcommand{\om}{\omega}
\newcommand{\vr}{\varrho}
\newcommand{\vp}{\varphi}
\newcommand{\dr}[1]{\left(#1\right)}

\newcommand{\dn}[1]{\left\|#1\right\|}
\newcommand{\da}[1]{\left|#1\right|}
\newcommand{\db}[1]{\left\{#1\right\}}
\newcommand{\Rmnum}[1]{\uppercase\expandafter{\romannumeral#1}}
\newcommand{\eps}{\varepsilon}

\newcommand{\br}{\bar{r}}
\newcommand{\bz}{\bar{z}}
\newcommand{\tw}{\widetilde w}
\newcommand{\bw}{\overline w}
\newcommand{\bu}{\overline{u}}
\newcommand{\tr}{\tilde r}
\newcommand{\tz}{\tilde z}
\newcommand{\hH}{\widehat H}
\renewcommand{\H}{\mathbb{H}}
\newcommand{\dbtilde}[1]{\accentset{\approx}{#1}}

\newcommand{\E}{\mathbf{E}}
\newcommand{\Span}{\mathrm{Span\,}}

\newtheorem{proposition}{Proposition}
\newtheorem{lemma}{Lemma}
\newtheorem{definition}{Definition}
\newtheorem{theorem}{Theorem}
\newtheorem{corollary}{Corollary}

\numberwithin{equation}{section}

\newcommand{\Rb}{\mathbb{R}}

\newcommand{\Dc}{\mathcal{D}}
\newcommand{\Ec}{\mathcal{E}}
\newcommand{\Fc}{\mathcal{F}}
\newcommand{\Gc}{\mathcal{G}}

\newcommand{\Oc}{\mathcal{O}}

\newcommand{\pa}{\partial}

\newcommand{\inn}{\textup{in}}
\newcommand{\out}{\textup{ex}}
\renewcommand{\mod}{\textup{Mod}}
\newcommand{\Ls}{\mathscr{L}}
\newcommand{\Ms}{\mathscr{M}}
\newcommand{\Hs}{\mathscr{H}}
\newcommand{\As}{\mathscr{A}}

\newcommand{\Lst}{\tilde{\mathscr{L}}}
\newcommand{\BS}{\mathrm{BS}}

\title{Axisymmetric type \Rmnum{2} blowup solutions to the three-dimensional Keller-Segel system}

 \author{Thomas Y. Hou\footnote{Applied and Computational Mathematics, Caltech, Pasadena, CA. Email: \href{hou@cms.caltech.edu}{hou@cms.caltech.edu}},\; 
 Van Tien Nguyen\footnote{Department of Mathematics, National Taiwan University, Taiwan. Email: \href{vtnguyen@ntu.edu.tw}{vtnguyen@ntu.edu.tw}},\; 
 Peicong Song\footnote{Applied and Computational Mathematics, Caltech, Pasadena, CA. Email: \href{psong2@caltech.edu}{psong2@caltech.edu} }}

\date{\vspace{-5ex}}

\begin{document}
\maketitle
\begin{abstract}
    We construct axisymmetric solutions to the three-dimensional parabolic-elliptic Keller-Segel system that blow up in finite time.
    In particular, the singularity is of type \Rmnum{2}, which locally admits a leading-order profile of the rescaled stationary solution of the two-dimensional system. Additionally, mass concentration occurs along a one-dimensional ring in the plane. In the analysis, we rely on an approximate solution of the eigenproblem associated with the linearized operator around the stationary solution as well as the modulation dynamics to control the perturbation function and derive the accurate blowup rate.
\end{abstract}

\section{Introduction}
\subsection{Problem Setup}
We consider the three-dimensional Keller-Segel system 
\begin{equation}\label{Introduction:3dks}
       \begin{cases}
        \pa_t u(\x,t) = \nabla\cdot(\nabla u(\x,t) - u(\x,t)\nabla\Phi_u(\x,t))\quad (\x,t)\in\reall^3\times\reall,\\
        -\Delta\Phi_u(\x,t) = u(\x,t),
    \end{cases} \tag{3dKS}
\end{equation}
where the $3$D Poisson field is given by $\Phi_u = \frac{1}{4\pi|\x|}*u$. More generally, the Keller-Segel system can be defined in $\reall^d$, which we will discuss shortly.
The Keller-Segel system models \textit{chemotaxis}, the directed movement of biological organisms in response to chemical gradients. Notable examples include the behavior of slime mold (\textit{Dictyostelium discoideum}) and \textit{Escherichia coli} bacteria.
It was first established by Patlak \cite{Patlak1953} and Keller \& Segel \cite{KellerSegel1970}. We refer to \cite{Horstmann2004} and \cite{Chavanis2008} for a survey of this model as well as related mathematical problems.   
Since the Keller-Segel system (in general dimension $d$) takes a divergence form, its strong solutions preserve the total mass:
\[
     \int_{\reall^d}u(\x,t)\,d\x = \int_{\reall^d}u(\x,0)\,d\x:=M \quad \forall\;t>0.
\]
In addition, the solutions admit an important scaling symmetry (besides the translation symmetries in space and time): if $u(\x,t)$ is a solution, then so is
\[
    u_\lambda(\x,t):= \lambda^2 u(\lambda \x,\lambda^2t),
\]
for any $\lambda >0$. We say that the solution $u$ blows up in finite time $T$, if
\[
    \limsup_{t\to T}\|u(t)\|_{L^\infty(\reall^d)} = +\infty.
\]
A blowup at time $T$ is said to be of type \Rmnum{1} if there exists a constant $C>0$, such that
\[
   \limsup_{t\to T}\;(T-t)\|u(t)\|_{L^\infty(\reall^d)}\leq C.
\]
Otherwise, the blowup is called type \Rmnum{2}.
The aim of this work is to construct a type \Rmnum{2} finite-time blowup solution to the $3$D system \eqref{Introduction:3dks}, whose mass concentrates along a ring on the plane, i.e. $\{(x_1,x_2,0):x_1^2+x_2^2 = R^2\}$ for some $R>0$. 

\subsection{Previous Results}
There are abundant results on both the well-posedness and singularity formation of the Keller-Segel system in dimension $d$. In particular, the blowup mechanisms can vary significantly in different dimensions. The case $d=2$ is called $L^1$-critical, as the scaling transformation $u\mapsto u_\lambda$ preserves the $L^1$-norm of $u$. On the other hand, the case $d\geq 3$ is called $L^1$-supercritical, and the scaling transformation preserves the $L^{d/2}$-norm. \\
\textbf{The case} $\mathbf{d=2}$. In the study of the $2$D Keller-Segel system, the stationary state solution plays a fundamental role:
\begin{equation}\label{intro: steady state solution}
    U(x_1,x_2):= \frac{8}{(1+x_1^2+x_2^2)^2},
\end{equation}
whose Poisson field is $\Psi_U = -2\log(x_1^2+x_2^2)$. It turns out that $\int U = 8\pi$ is the critical mass threshold that distinguishes between the global existence and finite-time blowup. For $M<8\pi$, there is global existence of solutions that diffuse to zero; for example, see \cite{Blanchet2006,Biler2006}. For $M=8\pi$, there exist infinite-time blowup solutions as well as global regularity results \cite{Blanchet2008,Davila2024,Blanchet2012}. For $M>8\pi$, there are various concrete examples of finite-time blowup solutions, which admit the form 
\begin{equation}\label{previous 2D blowup result and law}
    u(\x,t) = \frac{1}{\lambda^2(t)}(U+\tilde u(t))\dr{\frac{\x-\x^*(t)}{\lambda(t)}},\quad \lambda(t) \ll \sqrt{T-t},
\end{equation}
with $\tilde u\to 0$ and $\x^*(t)\to\x^*$ as $t\to T$ in a certain topology. Formal asymptotics and rigorous proofs can be found in \cite{Herrero1996,Velazquez2002, Raphael2014,Ann.PDE22,CPAM2022, buseghin2023existencefinitetimeblowup}. In particular, in \cite{CPAM2022} the authors are able to determine the precise stable blowup rate:
\begin{equation}\label{more precise blowup law}
    \lambda(t) = 2e^{-\frac{2+\mathbf{E}}{2}}\sqrt{T-t}\;e^{-\sqrt{\frac{|\log(T-t)|}{2}}}(1+o_{t\to T}(1)),
\end{equation}
where $\mathbf{E}$ is the Euler constant. Other (unstable) blowup rates were also derived in \cite{CPAM2022}. Our work is closely related to this line of research. Indeed, the $3$D axisymmetric Keller-Segel system resembles a $2$D one near the center of the blowup ring, which allows us to locally recover the same blowup mechanism (in particular, the same blowup rate). See the next subsection for a detailed discussion. There are other blowup scenarios, for example, the unstable ones in \cite{CPAM2022}, as well as the multiple collapsing blowup \cite{Seki2013,collot2024}. It is worth mentioning that for the $2$D Keller-Segel system there is no type \Rmnum{1} blowup (for example, see \textit{Theorem 10} in \cite{YukiNaito2008}).
\\
\textbf{The case} $\mathbf{d\geq 3}$. Similar to the $2$D case, there is a threshold on $\|u(0)\|_{L^{d/2}}$ that distinguishes between global existence in time and finite-time blowup. For small initial data, global existence results can be found in \cite{Corrias2004,Calvez2012}.
Different from the $d=2$ case, for $d\geq 3$ there exist type \Rmnum{1} blowups \cite{Herrero1998,Senba2005, nguyen2025construction,Souplet2019, nguyen2026infinitely}. In particular, the exact self-similar solution has been shown to be fully stable when $d=3$ \cite{Glogić2024,collot2024stabilitytypeiselfsimilar,li2025nonradialstabilityselfsimilarblowup}. There is also a type \Rmnum{2} radial collapsing sphere blowup, which was first formally constructed in \cite{Herrero1997collapsingring3d} and then proved rigorously in \cite{COLLOT2023collapsingRing}. In this scenario, the blowup profile is a traveling wave solution of the viscous Burgers' equation. It is worth noting that blowup can occur on the entire sphere in the Keller-Segel type system with strong diffusion degeneracies \cite{Winkler2025}. Finally, we refer the readers to \cite{Tello07062007,KANG201657,WINKLER2011261,Winkler2018,Fuest2021,liu2025finitetimeblowupkellersegel} for related discussions on the more general Keller-Segel system with logistic damping.

\subsection{Statement of the Result}
In \eqref{Introduction:3dks}, we consider the axisymmetric setting and adopt the cylindrical coordinate
\begin{equation*}
    u = u(r,z, t), \quad r = \sqrt{x_1^2 + x_2^2}, \quad z = x_3,
\end{equation*}
where the functions are defined on the half-space
\begin{equation}
    \H^+:= \db{(r,z):r\geq 0,\;z\in\reall}.
\end{equation}
For any function $f(r,z):\H^+\to \reall$, we define the norm
\[
    \|f\|_\Ec:= \dr{\int_{\H^+\cap\{(r,z):\,r^2+z^2\leq 4\}}\dr{|f(r,z)|^2+|\nabla_{r,z} f(r,z)|^2}\,drdz}^{\frac{1}{2}}+\sup_{\H^+\cap\{(r,z):\,r^2+z^2\geq 1\}}
    |f(r,z)(1+r^2+z^2)^\frac{3}{4}|.
\]
The solutions we construct lie in the following function space:
\[
    \Ec:= \{u:\H^+\to \reall\mid \|u\|_\Ec<+\infty\}.
\]
\begin{theorem}[Axisymmetric type \Rmnum{2} blowup for the $3$D Keller-Segel system]\label{main theorem}
 For any $T>0$, there exist initial data $u_0\in\Ec$ and a radius $R_0>0$, such that the following holds for the associated solution to \eqref{Introduction:3dks}. It blows up at finite time $T$ according to the dynamics
 \[
     u(r,z,t) = \frac{1}{\lambda^2{(t)}}\dr{U+\tilde u(t)}\dr{\frac{r-R(t)}{\lambda(t)},\frac{z}{\lambda(t)}}\text{ with }u(r,z,0) = u_0\dr{r-R_0,z},
 \]
 such that:
 \begin{itemize}
     \item Law for the blowup scale:
     \begin{equation}\label{main theorem: blowup law}
         \lambda(t) = \sqrt{T-t}e^{-\sqrt{\frac{|\log(T-t)|}{2}}+\Oc(1)}\text{ as } t\to T;
     \end{equation}
     \item Convergence to the stationary state profile:
     \[
         \|\tilde u(t)\|_\Ec \to 0\text{ as } t\to T;
     \]
     \item Convergence of the blowup radius: there exists $R_* = R_*(T,u_0,R_0)>0$, such that $R(t)\to R_*$ as $t\to T$.
 \end{itemize}
\end{theorem}
\begin{figure}[H]
    \centering
    \begin{subfigure}{0.48\textwidth}
        \centering
        \includegraphics[width=\textwidth]{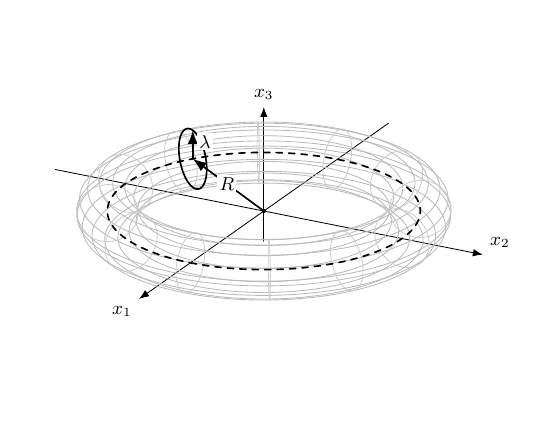}
        \caption{Axisymmetric ring geometry.}
    \end{subfigure}
    \hfill
    \begin{subfigure}{0.48\textwidth}
        \centering
        \includegraphics[width=\textwidth]{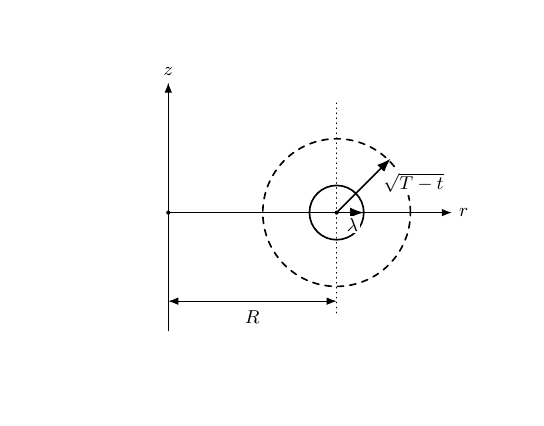}
        \caption{Cross-section in the $(r,z)$-plane.}
    \end{subfigure}
    \caption{Schematic illustration of the ring-shaped blowup geometry.
Left: the singular set is an axisymmetric ring of radius $R$ lying in the plane $x_3=0$,
with transverse thickness measured by the smaller scale $\lambda(t)$.
Right: in the $(r,z)$ cross-section near the ring center $(R,0)$, the smaller solid circle
represents the soliton scale $\lambda(t)$, while the larger dashed circle represents
the parabolic scale $\sqrt{T-t}$.} \label{fig:ring-geometry}
\end{figure}
As illustrated in Figure 1, our constructed solution concentrates on a fixed ring of radius $R > 0$ with the length scale $\lambda(t) \to 0$ as $t \to T$. Formally, the leading-order dynamics resemble the two-dimensional case, in the sense that near the ring ($|r-R|\lesssim \lambda$), we have 
\[
\frac{1}{r}\partial_r\sim \frac{1}{R\lambda}
\ll
\partial_r^2+\partial_z^2\sim \frac{1}{\lambda^2},
\]
hence,
\[
\Delta_{\x} u \approx (\partial^2_r + \partial_z^2)u, \quad \nabla_{\x} \cdot (u \nabla_{\x} \Phi_u) \approx \partial_r(u \partial_r \Psi_u) + \partial_z(u \partial_z \Psi_u),\quad\text{where}\; -(\pa_r^2+\pa_z^2)\Psi_u =u.
\]
Therefore, the local blowup profile is governed by $U$, the stationary solution of the two-dimensional Keller--Segel system. This intuition will be made rigorous in the proof.

We remark that the constructed solutions can have either finite or infinite total mass, depending on the choice of initial data. Indeed, let $\mathbf{m}(f) = 2\pi\int_{-\infty}^\infty \int_0^\infty f(r,z)r dr dz$ be the total mass of $f$ and $f_{\lambda, R}(r,z) := \frac{1}{\lambda^2}f\Big( \frac{r - R}{\lambda}, \frac{z}{\lambda}\Big)$, we then have $\mathbf{m}(u) = \mathbf{m}(U_{\lambda, R}) + \mathbf{m}(\tilde u_{\lambda, R})$, where $\mathbf{m}(U_{\lambda, R}) \sim 16\pi^2 R_*$. Since the control of the perturbation $\tilde{u}_{\lambda, R}$ requires neither the finiteness of $\mathbf{m}(\tilde u_{\lambda, R})$ nor the smallness of $\|\tilde u_{\lambda,R} \|_{L^1}$, this allows the initial perturbation $\tilde u$ to have either finite mass or infinite mass supported away from the ring. This demonstrates the robustness of our method, which does not require the $L^1$-smallness assumption of the perturbation.

\noindent\textbf{Comments on the result.}\\
\textit{(\romannumeral1) A new blowup scenario for the Keller-Segel system.} To the best of our knowledge, this is the first blowup result of this kind for the Keller-Segel system, whose leading-order geometry is nonradial with a $1$-dimensional singular set. The solution we construct here converges in distribution to a Dirac measure supported on a $1$-dimensional circle on the $x_3=0$ plane. It is worth noting that similar blowup phenomena occur in other systems, for example, harmonic map flow into $\mathbb{S}^2$ \cite{JuanDávila2019_HarmonicMapFlow} and supercritical heat equation \cite{delpino2020newtypeiifinite}. In our case, the partial mass technique does not work and the linearized operator is essentially nonlocal. Indeed, to close the bootstrap argument, we must derive sharp controls on perturbations across regions to prevent interference. It is interesting to see that while the leading-order dynamics near the blowup ring resembles the $2$D Keller-Segel system, the dynamics away from the ring remain genuinely three-dimensional and should be dealt with separately. This work provides a method of lifting a lower-dimensional blowup to a higher-dimensional space, which may be applied to other systems. After our work, we realized that the authors of \cite{delPino_JFA26} established a similar axisymmetric blowup result, in which they used a highly sophisticated parabolic inner-outer gluing procedure to construct solutions that concentrate on multiple separate rings. Moreover, they were able to provide a refined blowup law that is consistent with the $2$D blowup scenario \eqref{more precise blowup law}.
\\
\textit{(\romannumeral2) Simplification of the spectral analysis.}
The spectral information of the linearized operator $\Ls_\nu^\zeta$ (defined in \eqref{intro: linearized operator parabolic}) in the radial sector played an essential role in the analysis of \cite{Ann.PDE22,CPAM2022}. However, as shown in \cite{Ann.PDE22}, the precise construction of the eigenfunctions can be a heavy task. Since the eigenfunctions are only used to construct an approximate solution, it suffices to solve the eigenproblems only approximately, which greatly simplifies our analysis. Indeed, through a simple asymptotic matching procedure, we obtain the first two approximate eigenfunctions of $\Ls^\zeta_\nu$ with sufficiently small generated errors, which are sufficient for our analysis. See Proposition \ref{proposition 1} for details. It is worth noting that a similar technique has been applied in a recent work \cite{collot2024} to construct a finite-time singularity formed by the collision of two collapsing solitons for the $2$D Keller-Segel system.\\ 
\textit{(\romannumeral3) A robust approach.} In our analysis, we completely avoid using the partial mass setting and control both the radial and nonradial parts of the perturbation at the same time. We remark that the analysis in \cite{CPAM2022} crucially used the partial mass setting for which the nonlocal operator $\Ls_\nu^\zeta$ was transformed into a local one which is self-adjoint in a weighted $L^2$ space. In our case, by enforcing suitable local orthogonality conditions, we can obtain equivalence of norms as well as coercivity of the linearized operator for the whole perturbation function. See Sect. \ref{sec: properties of the linearized oper} for the discussion. Since the strategy provided here is simple and not restricted to the radial sector, we expect it can be applied to other problems. 
\\
\textit{(\romannumeral4) Adapted inner product and coercivity of the linearized operator.} The coercivity of the linearized operator plays a crucial role in the control of the perturbation around the blowup ring. In this region, we deal with a two-scale problem -- the larger parabolic scale and the smaller soliton scale. The linearized operator $\Ls^\zeta_\nu$ has different limits in these two scales (partly due to the presence of the scaling term), each of which has its own coercivity structure. Therefore, in order to obtain coercivity in both scales (i.e. the ``global" coercivity) we design a mixed inner product (see \eqref{def of adapted inner scl, parabolic}) that is compatible with both structures according to the idea of asymptotic matching. Moreover, it preserves norm equivalence for functions with the local orthogonality conditions, which is important for the energy estimates. See Proposition \ref{Main coercivity prop} for details. 
\\
\textit{(\romannumeral5) Stability.} Based on our analysis, we expect that the constructed ring-like blowup solution is stable in the axially symmetric regime, which is similar to the $2$D blowup mechanism reported in \cite{CPAM2022}. On the other hand, we believe that this ring-like concentration mechanism is inherently structurally unstable against non-axially symmetric perturbations, similar to liquid filaments or vortex rings in fluid dynamics such as Rayleigh-Plateau-type instabilities or azimuthal wrinkling. \\
\textit{(\romannumeral6) On numerical simulations.} Since the ring-like solution is expected to be stable within the strictly axially symmetric regime, numerical simulation should be feasible, and the particle method developed in \cite{bertozzi2012aggregation} (see also \cite{bertozzi2016regularity, campospinto2018convergence}) for the inviscid case could be used to simulate this ring-type solution. It is worth mentioning that in the absence of diffusion, solutions to the inviscid system can collapse into a singular measure supported on nontrivial shapes, for example, a three-branched star, a skeleton of codimension 1, or a toroidal patch with a square section, … (see \cite{bertozzi2012aggregation}). We expect that no such nontrivial shape-concentrating solution exists in the Keller-Segel system because of the smoothing effect of diffusion, so proving this exclusion would be an interesting problem to be addressed separately. Finally, it is worth mentioning the recent work \cite{hu2026faststochasticinteractingparticlefield}, where the authors developed a novel numerical method (stochastic interacting particle-field method with particle-in-cell acceleration, or SIPF-PIC) for the efficient simulation of the $3$D parabolic-parabolic Keller-Segel systems. In particular, they were able to accurately resolve the ring-shaped blowup pattern, which aligns with our theoretical scenario.
\\
\textit{(\romannumeral7) Connection with the Nonlinear Schr\"odinger equations (NLS).} Finally, we remark on the connection between the Keller-Segel system and the Nonlinear Schr\"odinger equations:
\begin{equation}\label{NLS}
    i\pa_t\psi(\x,t)+\Delta \psi(\x,t)+\psi(\x,t)|\psi(\x,t)|^{p-1}=0,\quad \x\in\reall^d.\tag{NLS}
\end{equation}
Here we summarize some blowup phenomena of NLS that share similarities with the Keller-Segel system. The case $p-1 = \frac{4}{d}$ in \eqref{NLS} is called ($L^2$-)critical and blowup occurs once the mass (i.e., $\|\psi\|_{L^2}$) is above a certain threshold. A stable blowup mechanism in this case enjoying the so-called ``loglog" law can be found in \cite{Perelman2001,MerleNLSAnnals2005,MerleNLSGAFA2003,MerleNLSJAMS06}. In the supercritical cases $p-1>\frac{4
}{d}$, there exist standing-ring blowup solutions, spherical in this context (see \cite{RaphaelNLS2006,RaphaelNLS2009}), as well as collapsing ring blowup solutions (see \cite{MerleRapaelSzeftelNLS14}). It is worth noting that all the blowup solutions mentioned above converge in a certain sense to some $1$-dimensional ground state solutions of \eqref{NLS}. We also recommend \cite{NLS2015} for a comprehensive review on NLS.
\\[4pt]
\noindent\textbf{Notation and conventions.} Unless otherwise specified, differential operators such as $\Delta$, $\nabla$ and $\nabla\cdot$ are understood as $2$D ones, and $\int$ denotes the integration on the half-space $\H^+$ equipped with the $2$D Lebesgue measure. For any function $f(r,z):\H^+\to \reall$, we identify it with a $3$D axisymmetric function via $\tilde f(x_1,x_2,x_3):= f(\sqrt{x_1^2+x_2^2},x_3)$, and define its $3$D Poisson field
\[
    \Phi_f := \frac{1}{4\pi|(x_1,x_2,x_3)|}*\tilde f,
\]
where $|(x_1,\dots,x_n)|:=\sqrt{x_1^2+\cdots+x_n^2}$ denotes the standard Euclidean norm on $\reall^n$. On the other hand, we can extend $f$ to some $\bar f:\reall^2\to \reall$ ( see, for example \eqref{Appendix: def of modified 2D Poisson field} for details), and define its $2$D Poisson field as
\[
   \Psi_f := -\frac{1}{2\pi}\log(|(x_1,x_2)|)*\bar{f}.
\]
With a slight abuse of notation, we also use $\Phi$ and $\Psi$ to denote the standard Poisson fields for functions on $\reall^3$ and $\reall^2$ (with suitable decay), respectively. We define the difference of the $2$D and $3$D Poisson fields as
\[
    \Theta_f := \Phi_f - \Psi_f. 
\]
Now for $\nu>0$, we denote 
\[
    U_\nu(x_1,x_2):= \frac{1}{\nu^2}U(x_1/\nu,x_2/\nu),
\]
where $U$ is the stationary solution defined in \eqref{intro: steady state solution}, and the $2$D differential operator
\[
    \Lambda f(x_1,x_2) := -\left.\frac{d}{d\nu}\right|_{\nu=1} \frac{1}{\nu^2}f(x_1/\nu,x_2/\nu)= 2f + x_1\pa_{x_1} f+x_2\pa_{x_2} f = \nabla\cdot((x_1,x_2)f).
\]
Define $\chi\in C^\infty_c(\reall^2)$ to be a radially symmetric positive cutoff function with:
\[
    \chi(\x) = 
    \begin{cases}
        1\text{ for }|\x|\leq 1,\\
        0,\text{ for } |\x|\geq 2.
    \end{cases}
\]
With a slight abuse of notation, we will denote $\chi = \chi(|\x|)$. Given two fixed constants $0<\zeta_*\ll 1\ll \zeta^*< +\infty$ (which will be specified later in the analysis) and a small parameter $\nu>0$, we denote ($\zeta:=|\x|$)
\begin{align}\label{notations: special cutoff functions}
    &\chi_*(\zeta) := \chi(\zeta/\zeta_*),& &\chi^*(\zeta):=\chi(\zeta/\zeta^*),\nonumber\\
    &\chi_\nu(\zeta):=\chi(\zeta/|\log\nu|),& &\bar\chi_\nu(\zeta):=\chi(\zeta\nu/|\log\nu|).
\end{align}
We define the norms:
\begin{equation}\label{notations: weighted L2 norms}
    \|f\|^2_\inn:=\int_{\reall^2}\frac{\nu^2f^2\chi^2_\nu\vr_\nu}{U_\nu},\quad \|f\|^2_{L^2(U_\nu)}:=\int_{\reall^2}\frac{\nu^2f^2}{U_\nu},
\end{equation}
where $\vr_\nu$ is the exponential weight function defined in \eqref{def of weight functions}.
We denote the standard $L^2(\reall^2)$ inner product as $\scl{f}{g}:=\int_{\reall^2}fg$. For any two positive quantities, $A_1\lesssim A_2$ means that there exists some universal (independent of any parameters in this problem) constant $C>0$, such that $A_1\leq CA_2$. Similarly, $A_1\approx A_2$ means that there exists a universal constant $C>0$, such that $\frac{1}{C}A_1\leq A_2\leq CA_1$. Universal constants will be denoted generically as $C$ or $\delta$, the specific values of which may change from line to line. We use brackets to specify the dependence of constants on other quantities. For example, $C(A_1,A_2)$ will denote (generically) a constant depending only on $A_1$ and $A_2$, the specific value of which may also vary in different places. The expression $A_1 = \Oc(A_2)$ means that there exists a universal constant $C$, such that $|A_1|\leq C A_2$. We use $\sim$ to denote ``asymptotically equivalent" under a certain limiting process (which will always be specified), i.e., $A_1\sim A_2$ means $\lim A_1/A_2 \approx 1$.

\subsection{Strategy of the Proof}
Now we briefly describe the main steps of the proof of Theorem \ref{main theorem}. \\
\textit{Step 1: Renormalization and linearization of the problem.}
Using the scaling invariance of the equation, we first make the following change of variables according to the parabolic scaling and the standing ring scenario:  
\begin{align*}
    u(r,z,t) = \frac{1}{\mu^2}w\dr{\frac{r-R(t)}{\mu},\frac{z}{\mu},\tau},\quad \br := \frac{r-R(t)}{\mu},\quad\bz := \frac{z}{\mu},\quad \tau = \int_{0}^{t}\frac{1}{\mu^2(\tilde t)}\;d\tilde t,
\end{align*}
where $\mu(t):=\sqrt{T-t}$ is the parabolic scale with blowup time $T>0$. Define $\zeta:=\sqrt{\br^2+\bz^2}$. Then, the system for $(w,\Phi_w)$ is 
\begin{align}\label{Parabolic System}
\begin{split}
    &\pa_{\tau} w = \nabla\cdot(\nabla w - w\nabla\Phi_w)+\frac{1}{\br+R/\mu}(\pa_{\br} w - w\pa_{\br}\Phi_w)-\beta\Lambda w +\frac{R_\tau}{\mu}\pa_{\br} w\quad \beta:=-\frac{\mu_\tau}{\mu}=\frac{1}{2}, \\
    &-\dr{\pa^2_{\br} + \pa^2_{\bz} +\frac{1}{\br+R/\mu}\pa_{\br}}\Phi_w = w.
\end{split}
\end{align}
Since the blowup solutions we construct are of type \Rmnum{2}, there exists a smaller scale $\nu(t)\to 0$ as $t\to T$, beyond the parabolic scale. This scale, called the soliton scale (or the blowup scale), will serve as a crucial asymptotically small parameter in our analysis. 
Specifically, we consider the soliton change of variables:
\begin{align*}
    w(\br,\bz,\tau) = \frac{1}{\nu^2}v\dr{\frac{\br}{\nu},\frac{\bz}{\nu},s},\quad \rho:=\frac{\br}{\nu},\quad\xi:=\frac{\bz}{\nu},\quad \frac{d s}{d\tau} = \frac{1}{\nu^2}.
\end{align*}
Define $\gamma:=\sqrt{\rho^2+\xi^2}$. Note that $\Phi_v(\rho,\xi) = \Phi_w(\br,\bz)$. Then, the system for $(v,\Phi_v)$ is
\begin{align}\label{Soliton System}
\begin{split}
    &\pa_{s} v = \nabla\cdot(\nabla v - v\nabla\Phi_v)+\frac{1}{\rho+R/(\mu\nu)}(\pa_{\rho} v - v\pa_{\rho}\Phi_v)-\dr{\eta+\nu^2\beta}\Lambda v +\nu\frac{R_\tau}{\mu}\pa_{\rho} v\quad \eta:=-\frac{\nu_s}{\nu},\\
    &-\dr{\pa^2_{\xi} + \pa^2_{\rho} +\frac{1}{\rho+R/(\mu\nu)}\pa_{\rho}}\Phi_v = v.
\end{split}
\end{align}
The $3$D system \eqref{Parabolic System} can be approximated by a $2$D one in the parabolic scale (in a sense that will soon become clear), i.e., in the region around the blowup ring (and away from the axis of symmetry). Linearization of the system \eqref{Parabolic System} around the stationary solution $U_\nu$ gives the leading-order linearized operator:
\begin{align}\label{intro: linearized operator parabolic}
    \Ls^\zeta_\nu f  = \Ls^\zeta_{0,\nu}f-\frac{1}{2}\Lambda f = \nabla\cdot(U_\nu\nabla(\Ms^\zeta_\nu f))-\frac{1}{2}\Lambda f,\quad\text{with}\quad \Ms^\zeta_\nu f:= \frac{f}{U_\nu} - \Psi_f.
\end{align}
Equivalently, the linearized operator around $U$ for the system \eqref{Soliton System} is 
\begin{align}\label{intro: linearized operator soliton}
    \Ls f  = \Ls_0f-\frac{1}{2}\nu^2\Lambda f = \nabla\cdot(U\nabla(\Ms f))-\frac{1}{2}\nu^2\Lambda f,\quad\text{with}\quad \Ms f:= \frac{f}{U} - \Psi_f.
\end{align}
The core of the analysis lies within the parabolic scale (which includes the smaller soliton scale). We remark that while most of the time our analysis is done using the parabolic variables, it is equivalent and sometimes more convenient to work with soliton variables.\\
\textit{Step 2: Construction of an approximate solution.} 
The blowup solutions we construct converge locally to the steady state profile $U$ around the blowup ring. More specifically, we will show that $w$ in \eqref{Parabolic System} decomposes as $U_\nu$ plus some controllable perturbation. However, due to the instability of the linearized operator in certain directions, orthogonality conditions need to be imposed on the perturbation. This is done by the introduction of modulation parameters, whose dynamics yield the blowup rate. \\
To begin, we construct the first two approximate eigenfunctions of the linearized operator $\Ls^\zeta_\nu$, corresponding to the positive and ``almost zero" eigenvalues, respectively:
\begin{align}
    \Ls^\zeta_\nu \vp_{i,\nu} = \dr{1-i+\frac{1}{2\log(\nu)}}\vp_{i,\nu}+ R_i,\quad i=0,1,
\end{align}
where $R_i$ are small error terms. Then, exploiting the cancellation of $\vp_{1,\nu} - \vp_{0,\nu}$, we decompose the solution as
\[
    w = U_\nu+a(\vp_{1,\nu} - \vp_{0,\nu})+\eps,
\]
where $a=a(\tau)$ is a modulation parameter.
Thus, we can write the evolution for $\eps$:
\begin{equation}\label{intro: evolution for epsilon}
    \pa_\tau\eps = \Ls^\zeta_\nu\eps+L(\eps)+NL(\eps)+E,
\end{equation}
where both the extra linear term $L(\eps)$ and the nonlinear term $NL(\eps)$ are small in a suitable sense. It is essential that $\eps$ satisfies the local orthogonality conditions (we recall the definition of $\chi_*$ in \eqref{notations: special cutoff functions}):
\begin{equation}\label{intro: orthog conditions}
     \int_{\reall^2}\eps\chi_*(\zeta)\;d\br d \bz = \int_{\reall^2}\eps\Lambda U_\nu\chi_*(\zeta)\;d\br d \bz = \int_{\reall^2}\eps\nabla U_\nu\chi_*(\zeta)\;d\br d \bz =0.
\end{equation}
These conditions are preserved dynamically by our choice of the modulation parameters, aided by the enforcement of even symmetry of the solution in the $\bz$-direction. Through a study of the linearized operator, we can define an adapted inner product $\scl{\cdot}{\cdot}_{\nu,*}$ (see \eqref{def of adapted inner scl, parabolic}), such that its corresponding norm is equivalent to a weighted $L^2$ norm for any function satisfying the orthogonality conditions \eqref{intro: orthog conditions}. Moreover, for such functions, $\Ls^\zeta_\nu$ (up to a slight modification) will be coercive under the adapted inner product, in the sense of \eqref{main coercivity corollary ineq}. This coercivity is crucial in the energy estimates of $\eps$. 
\\
\textit{Step 3: Modulation dynamics and energy estimates.} The generated error in \eqref{intro: evolution for epsilon}, consisting mainly of the modulation parameters, admits a further decomposition:
\[
    E = \mod_0\vp_{0,\nu}+\mod_1\vp_{1,\nu}+\frac{R_\tau}{\mu}\pa_{\br}U_\nu+\tilde E,
\] 
where $\mod_i$ and $\tilde E$ are terms defined in Proposition \ref{prop: decomp of the gene error}. Note that we have three modulation parameters at hand, namely $\nu(\tau),a(\tau)$ and $\frac{R_\tau(\tau)}{\mu}$, which correspond precisely to the three orthogonality conditions in \eqref{intro: orthog conditions}. Projecting \eqref{intro: evolution for epsilon} onto these three directions together with the bootstrap assumptions on $\eps$ yields the modulation equations:
\[
\begin{cases}
    |\mod_0|=\da{a_\tau - 2a\beta\dr{1+\frac{1}{2\log(\nu)}}-16\nu^2\dr{\frac{\nu_\tau}{\nu}-\beta}} =\Oc\dr{\frac{\nu^2}{|\log\nu|}},\\[4pt]
    |\mod_1|=\da{-a_\tau +\frac{a(\tau)\beta}{\log(\nu)}} =\Oc\dr{\frac{\nu^2}{|\log\nu|^2}},\\[4pt]
    \da{\frac{R_\tau}{\mu}}=\Oc\dr{\frac{\nu}{|\log\nu|}},
\end{cases}
\]
the solution of which yields the blowup law given in \eqref{main theorem: blowup law}. The energy estimates for $\eps$ are done separately in the inner zone (i.e., $\br^2+\bz^2\lesssim 1$) and the outer zone (i.e., $\br^2+\bz^2\gtrsim 1$). In the inner zone, a weighted $H^1$-norm of $\eps$ is controlled, thanks to the coercivity of $\Ls^\zeta_\nu$. In the outer zone, we return to the original $3$D structure and control a weighted $L^\infty$-norm of $\eps$ via the dissipative structure of the system. The two sets of estimates are connected in the intermediate region via an $H^2$-control of $\eps$ resulting from the parabolic regularity of the system. Finally, combining the modulation equations and the energy estimates, we are able to close the bootstrap argument for both the modulation parameters and the energy norms of $\eps$, and the global control of $\eps$ in the renormalized time variable $\tau$ implies the finite-time blowup of the solution of the original Keller-Segel system.

For the reader's convenience, we summarize the main scales, modulation parameters, and auxiliary constants used throughout the proof in Table~\ref{tab:main-parameters}:

\begin{table}[H]
\centering
\footnotesize
\renewcommand{\arraystretch}{1.12}
\begin{tabularx}{\textwidth}{@{}p{0.22\textwidth}p{0.28\textwidth}X@{}}
\toprule
\textbf{Category} & \textbf{Symbols} & \textbf{Type / definition} \\
\midrule

Physical variables &
$t$, $T$, $r$, $z$ &
$t$ is time, $T$ is the blowup time, and $(r,z)$ are cylindrical variables. \\[0.2cm]

Ring parameters &
$R_0$, $R(t)$, $R_*$ &
Initial, time-dependent, and limiting radii of the concentration ring. \\[0.2cm]

Profiles &
$U$, $U_\nu$ &
$U(x)=8(1+|x|^2)^{-2}$ is the two-dimensional stationary profile, and $U_\nu(x)=\nu^{-2}U(x/\nu)$. \\[0.2cm]

Main scales &
$\lambda(t)$, $\mu(t)$, $\nu(\tau)$ &
$\lambda(t)$ is the physical blowup scale, $\mu(t)=\sqrt{T-t}$ is the parabolic scale, and $\nu(\tau)$ is the relative soliton scale, with $\lambda(t)=\mu(t)\nu(\tau)$. \\[0.2cm]

Renormalized times &
$\tau$, $s$ &
$\tau$ is the parabolic time and $s$ is the soliton time, related by $ds/d\tau=\nu^{-2}$. \\[0.2cm]

Parabolic variables &
$\bar r$, $\bar z$, $\zeta$ &
$\bar r=(r-R(t))/\mu(t)$, $\bar z=z/\mu(t)$, and $\zeta=(\bar r^2+\bar z^2)^{1/2}$. \\[0.2cm]

Soliton variables &
$\rho$, $\xi$, $\gamma$ &
$\rho=\bar r/\nu$, $\xi=\bar z/\nu$, and $\gamma=(\rho^2+\xi^2)^{1/2}$. \\[0.2cm]

Modulation parameters &
$a(\tau)$, $\beta$, $\eta$, $R_\tau/\mu$ &
$a$ is the scalar amplitude modulation, $\beta=-\mu_\tau/\mu=1/2$, $\eta=-\nu_s/\nu$, and $R_\tau/\mu$ is the ring-radius modulation. \\[0.2cm]

Perturbation variables &
$\eps$, $\tilde u$, $P$ &
$\eps$ is the perturbation in parabolic variables, $\tilde u$ is the perturbation in the original statement, and $P$ denotes the approximate correction. \\[0.2cm]

Approximate eigenfunctions &
$\vp_{0,\nu}$, $\vp_{1,\nu}$ &
The first two approximate eigenfunctions of the linearized operator around $U_\nu$. \\[0.2cm]

Operators and potentials &
$\Phi$, $\Psi$, $\Theta$, $\Lambda$, $\Ls^\zeta_\nu$, $\Ms^\zeta_\nu$ &
$3$D and $2$D Poisson fields, their difference, scaling operator, and linearized operators used in the inner/parabolic analysis. \\[0.2cm]

Cutoffs &
$\chi$, $\chi_*$, $\chi^*$, $\chi_\nu$, $\bar\chi_\nu$ &
Smooth localization functions used to separate the soliton, parabolic, and matching regions. \\[0.2cm]

Fixed constants &
$\zeta_*$, $\zeta^*$, $K_i$, $M_0$, $C$, $\delta$ &
Auxiliary localization, bootstrap, largeness, and universal constants. \\

\bottomrule
\end{tabularx}
\caption{Summary of the main constants, parameters, and variables used in the construction.}
\label{tab:main-parameters}
\end{table}

\indent This work is organized as follows. In Sect. \ref{sec: properties of the linearized oper}, we construct the first two approximate eigenfunctions of the linearized operator, and then explore its coercivity properties. Sect. \ref{sec: construction of blowup} is the heart of the analysis, including the setup of the bootstrap assumptions, derivation of the modulation equations, and energy estimates. Finally, we close the bootstrap argument in Sect. \ref{sec: existence of blowup} and conclude the proof of the main theorem.  

\section{Properties of the linearized operator}\label{sec: properties of the linearized oper}
This section is devoted to the study of the linearized operator 
\begin{equation} \label{def:Lszeta}
      \Ls^\zeta_\nu f  = \nabla \cdot \big( U_\nu \nabla \Ms^\zeta_\nu f \big) -\beta\Lambda f= \Ls_{0,\nu}^\zeta f - \beta \Lambda f.
\end{equation}
First, we construct the first two approximate eigenfunctions of $\Ls^\zeta_\nu$ in Proposition \ref{proposition 1}, which will be important building blocks of the approximate solution of the Keller-Segel system. We also describe their asymptotic behaviors and generated errors, which will be helpful in the energy estimates.
Next, we study the operator $\Ms^\zeta_\nu$,
the appearance of which is natural from a linearization of the free energy functional associated to the two-dimensional Keller-Segel equation: 
\begin{equation}
    \Fc(f) = \int_{\Rb^2} f \Big( \log f - \frac 12 \Psi_f\Big)\; d\x.
\end{equation}
Its definiteness and norm equivalence properties will motivate our definition of the adapted inner product, with which we are able to prove the crucial coercivity result for $\Ls^\zeta_\nu$ (Proposition \ref{Main coercivity prop}). 
We will adopt the soliton variables $(\rho,\xi)$ instead of the parabolic ones $(\br,\bz)$ when it is more convenient, though these two settings are equivalent in terms of analysis.

\subsection{Two approximate eigenfunctions}
In \cite{Ann.PDE22}, the authors used the partial mass setting to derive a complete description of the spectrum of $\Ls^\zeta_\nu$ in the radial setting. Here we derive only rough information of the spectrum via a simple asymptotic matching procedure, which is sufficient for our purpose of constructing blowup solutions. 

\begin{proposition}[Two approximate eigenfunctions]\label{proposition 1} Consider $\beta > 0$ and $0 < \nu \ll 1$ to be fixed. There are two smooth radial functions $\varphi_{0, \nu}$ and $\varphi_{1, \nu}$, with supports in $\{\zeta:\;\zeta\leq2|\log\nu|\}$, that solve
\begin{equation}
\Ls_\nu^\zeta \varphi_{i, \nu} = 2\beta \Big(1 - i + \frac{1}{2\log(\nu)} \Big)\varphi_{i, \nu} + R_i,\quad i=0,1.
\end{equation}
(i) \textup{(Approximate eigenfunctions)} 
\begin{equation}
    \varphi_{i, \nu}(\zeta) = -\frac{1}{16\nu^4} \varphi_{i}^\inn(\zeta/\nu) \chi_m  + \varphi_{i}^\out(\zeta) (1 - \chi_m)\chi_\nu = -\frac{1}{16 \nu^4} \Lambda U(\zeta/\nu) \chi_\nu  + \tilde \varphi_{i}(\zeta), 
\end{equation}
where $\varphi_i^\inn$ and $\varphi_i^\out$ are defined by \eqref{inner eigenfunction} and \eqref{outer eigenfunction}, and the cutoff function $\chi_m(\zeta) = \chi(\zeta/\zeta_m)$ is defined at the beginning of the proof. In particular, we have the pointwise estimates for $k = 0,1,2$:
\begin{align}\label{prop1: pointwise est 1}
\begin{split}
   &|\pa_\zeta^k \tilde \varphi_0 (\zeta)| + | \pa_\zeta^k \nu \pa_\nu \tilde \varphi_0(\zeta)| \lesssim\dr{\frac{\nu^2\zeta^{2-k}\log^2(2+\zeta/\nu)}{(\nu+\zeta)^6}+\frac{\zeta^{2-k} }{ |\log\nu|(\nu + \zeta)^{4}}}(1+\log(\zeta)\one_{\{\zeta>1\}}),\\  
   &|\pa_\zeta^k \tilde \varphi_1 (\zeta)| + | \pa_\zeta^k \nu \pa_\nu \tilde \varphi_1(\zeta)| \lesssim \frac{\zeta^{2-k} }{ (\nu + \zeta)^{4}}(1+\log(\zeta)\one_{\{\zeta>1\}}),
\end{split}
\end{align}
and the improved estimate near the origin,
\begin{align}\label{prop1: improved est near the origin}
\begin{split}
&|\pa_\zeta^k \nu \pa_\nu (\varphi_{1, \nu} - \varphi_{0, \nu})|  \lesssim  \dr{\frac{\nu^2\zeta^{2-k}\log^2(2+\zeta/\nu)}{(\nu+\zeta)^6}+\frac{1}{|\log(\nu)|}\cdot\frac{\zeta^{2-k}}{(\nu+\zeta)^2(1+\zeta)^2}}(1+\log(\zeta)\one_{\{\zeta>1\}}),\quad k = 0,1,2, \\
&|\pa_\zeta^k (\varphi_{1, \nu} - \varphi_{0, \nu})|  \lesssim  \frac{\zeta^{2-k} }{ (\nu + \zeta)^{4}}(1+\log(\zeta)\one_{\{\zeta>1\}}),\quad k = 0,1,2.
\end{split}
\end{align}

\noindent (ii) \textup{(Pointwise estimates of $R_i$)} 
\begin{equation}\label{prop1: pointwise est for R_i}
    \left|\pa^k_{\zeta}R_i(\zeta)\right|\lesssim
        \frac{\zeta^{2-k}}{(\nu+\zeta)^2(1+\zeta)^2}\frac{|\log(\nu+\zeta)|}{|\log(\nu)|}+\frac{\nu^2\zeta^{2-k}\log^2(2+\zeta/\nu)}{(\nu+\zeta)^4}\quad k=0,1,2,
    \quad\;
    \left| \int_{\zeta < 1} R_i(\zeta) \zeta d\zeta \right| \lesssim \frac{1}{|\log \nu|}.
\end{equation}
\end{proposition}

\begin{proof}[Proof of Proposition 1]
We proceed as follows: First, we construct the inner approximate eigenfunctions by iterated inversions of the linearized operator. Second, we solve the outer approximate eigenproblems whose eigenfunctions are well known. Third, we match the inner eigenfunctions with the outer ones by specifying the $\Oc(|\log\nu|^{-1})$ part of the approximate eigenvalues. Finally, the pointwise estimates follow directly from the explicit construction of the (global) eigenfunctions.\\[3pt]
\underline{The construction of $\varphi^\inn_i$:}
Fix a small constant $\zeta_m>0$ (the subscript ``$m$" stands for ``matching"), and $\zeta=\zeta_m$ will be our matching point. Denote $\chi_m(\zeta):=\chi(\zeta/\zeta_m)$. Now, consider the inner region, i.e., $\zeta\in (0,\zeta_m)$ or $\gamma\in (0,\zeta_m/\nu)$. The inner eigenproblem, in the soliton variables, is equivalent to
\begin{equation}\label{inner eigenproblem}    \Ls_0\varphi^\inn_{i}:=\nabla\cdot\dr{U\nabla\dr{\frac{\varphi^{\inn}_{i}}{U}-\Psi^{\inn}_{i}}} = \nu^2\beta\Lambda \varphi^\inn_{i}+2\nu^2\beta(1-i+\tilde\alpha_{i,\nu})\varphi^\inn_{i},
\end{equation}
where
\[
    -(\pa^2_\gamma+\frac{1}{\gamma}\pa_\gamma)\Psi^\inn_{i} = \varphi^\inn_{i},
\]
and $\tilde\alpha_{i,\nu}$ is a next-order part of the approximate eigenvalues to be solved. We look for an approximate solution which takes the form:
\begin{equation}\label{inner eigenfunction}
    \varphi^\inn_{i} = \Lambda U + \nu^2 c_{i,2} V_2+\nu^2 \tilde c_{i,2}\tilde V_2 + \nu^4d_{i,4}V_{4,\#}+\nu^4\tilde d_{i,4}\tilde V_{4,\#} + \nu^4c_{i,4}V_4+\nu^4\tilde c_{i,4}\tilde V_4, 
\end{equation}
where
\begin{equation}\label{proposition 1: construction of building block functions}
\begin{aligned}
    &\Ls_0 V_2 = \Lambda U, &&\Ls_0 \tilde V_2 = \Lambda^2 U, &&\Ls_0 V_{4,\#} = \Lambda V_2,\\
    &\Ls_0\tilde V_{4,\#} = \Lambda \tilde V_2, &&\Ls_0 V_4  =V_2, &&\Ls_0 \tilde V_4 = \tilde V_2,
\end{aligned}
\end{equation}
and $c_{i,2},d_{i,4},c_{i,4},\tilde c_{i,2},\tilde d_{i,4},\tilde c_{i,4}$ are constants (which may depend on $\tilde\alpha_{i,\nu}$) that will be chosen to improve the approximation and matching errors. The building block functions above can be solved explicitly (as the corresponding second order ODE admits explicit solutions). Their asymptotic behaviors, as $\gamma \to \infty$ are: 
\begin{equation}\label{Prop1: asyp at infty}
\begin{aligned}
&V_2 = \frac{4}{\gamma^2}+\Oc\dr{\frac{\log^2(\gamma)}{\gamma^4}}, &&\tilde V_2 = -\frac{8}{\gamma^2}+\Oc\dr{\frac{\log^2(\gamma)}{\gamma^4}},&& V_{4,\#} = 1+\Oc\dr{\frac{\log^2(\gamma)}{\gamma^2}},\\
&\tilde V_{4,\#} = -2+\Oc\dr{\frac{\log^2(\gamma)}{\gamma^2}}, &&V_4 = \log(\gamma)-\frac{5}{4}+\Oc\dr{\frac{\log^2(\gamma)}{\gamma^2}},&& \tilde V_4 = -2\log(\gamma)+\frac{7}{2}+\Oc\dr{\frac{\log^2(\gamma)}{\gamma^2}}.\\
\end{aligned}
\end{equation}
and their asymptotic behavior as $\gamma\to 0^+$ are (for derivative orders $k=0,1,2$):

\begin{equation}\label{Prop1: asymp at 0}
\begin{aligned}
    &V_2^{(k)}\sim \gamma^{2-k},&&\tilde V_2^{(k)}\sim \gamma^{2-k}, &&V_{4,\#}^{(k)} \sim \gamma^{4-k},\\
    &\tilde V_{4,\#}^{(k)} \sim \gamma^{4-k},&& V_4^{(k)}\sim \gamma^{2-k},& &\tilde V_4^{(k)}\sim \gamma^{2-k}.
\end{aligned}
\end{equation}
We remark that although these building block functions are not linearly independent at the leading order, some of them are in fact necessary in order to obtain cancellations in the generated error. This again (as we have already seen in the formal asymptotic matching in the previous section) emphasizes the idea of ``iterative inversions of $\Ls_0$", which is a natural way of improving the generated error in the soliton scale. The inner generated error in the soliton variable is defined as
\begin{align}\label{prop1: full inner error}
\begin{split}
    R^\gamma_i &:= \Ls_0 \varphi^\inn_{i} - \beta\nu^2\Lambda \varphi^\inn_i - 2\beta\nu^2(1-i+\tilde\alpha_{i,\nu})\varphi^\inn_i\\
    &= \nu^2\dr{c_{i,2}\Lambda U + \tilde c_{i,2}\Lambda^2U - \beta\Lambda^2U-2\beta(1-i+\tilde\alpha_{i,\nu})\Lambda U}\\
     &\quad+ \nu^4\Big(d_{i,4}\Lambda V_2+\tilde d_{i,4}\Lambda\tilde V_2+c_{i,4}V_2+\tilde c_{i,4}\tilde V_2 - \beta c_{i,2}\Lambda V_2-\beta\tilde c_{i,4}\Lambda\tilde V_2\\
     &\quad -2\beta c_{i,2}(1-i+\tilde\alpha_{i,\nu})V_2-2\beta \tilde c_{i,2}(1-i+\tilde\alpha_{i,\nu})\tilde V_2\Big)\\
     &\quad -\beta\nu^6\dr{d_{i,4}\Lambda V_{4,\#}+\tilde d_{i,4}\Lambda\tilde V_{4,\#} - c_{i,4}\Lambda V_4- \tilde c_{i,4}\Lambda \tilde V_4}\\
     &\quad -2\beta\nu^6(1-i+\tilde\alpha_{i,\nu})\dr{d_{i,4} V_{4,\#}+\tilde d_{i,4}\tilde V_{4,\#} - c_{i,4} V_4- \tilde c_{i,4}\tilde V_4}.
\end{split}
\end{align}
To cancel out the $\Oc(\nu^2)$ terms in $R^\gamma_i$, we choose
\begin{equation}\label{prop1: determine constants c_2}
    c_{i,2} = 2\beta(1-i+\tilde\alpha_{i,\nu}),\quad\tilde c_{i,2} = \beta.
\end{equation}
Similarly, to cancel out the $\Oc(\nu^4)$ terms, we choose
\begin{equation}
\begin{aligned}
    d_{i,4} &= 2\beta^2(1-i+\tilde\alpha_{i,\nu}),
    &\tilde d_{i,4} &= \beta^2,\\
    c_{i,4} &= 4\beta^2(1-i+\tilde\alpha_{i,\nu})^2,& \tilde c_{i,4} &= 2\beta^2(1-i+\tilde\alpha_{i,\nu}), 
\end{aligned}
\end{equation}
and the error becomes
\begin{align}\label{prop1: inner error}
    R^\gamma_i =& -\beta^3\nu^6\dr{\Lambda \tilde V_{4,\#}+2(1-i+\tilde\alpha_{i,\nu})\Lambda V_{4,\#}+2(1-i+\tilde\alpha_{i,\nu})\Lambda \tilde V_4+4(1-i+\tilde\alpha_{i,\nu})^2\Lambda V_4}\\
    &-2\beta^3\nu^6(1-i+\tilde\alpha_{i,\nu})\dr{\tilde V_{4,\#}+2(1-i+\tilde\alpha_{i,\nu}) V_{4,\#}+2(1-i+\tilde\alpha_{i,\nu}) \tilde V_4+4(1-i+\tilde\alpha_{i,\nu})^2 V_4}.\nonumber
\end{align}
\underline{The construction of $\varphi^\out_i$:} Then, we work in the outer region, i.e., $\zeta\in (\zeta_m,+\infty)$. In this region, we have $U_\nu(\zeta) \lesssim \nu^2$ and $\pa_\zeta \Psi_{U_\nu}(\zeta) \sim -\frac{4}{\zeta}$, hence, the operator $\Ls_\nu^\zeta$ behaves like the Hermite operator in dimension $6$:
\begin{equation*}
    \Hs = \pa_\zeta^2 + \frac{5}{\zeta} \pa_\zeta - \beta \Lambda,
\end{equation*}
and we formally have $\Ls^\zeta_{\nu} = \Hs + \Oc(\nu^2)$ when $\zeta\geq\zeta_m$. Thus, we consider an approximation of the outer eigenfunction of the form 
\begin{equation}\label{outer eigenfunction}
    \varphi_i^\out(\zeta)= \Omega_i(\zeta) + \tilde \varphi_i^\out(\zeta), 
\end{equation}
for some lower order term $\tilde \varphi_i^\out(\zeta)\sim \Oc(\tilde\alpha_{i,\nu})$ and the leading term $\Omega_i$ solves  
\begin{equation}\label{proposition 1: outer eigen problem}
    \dr{\Hs - 2\beta(1 - i)} \Omega_i(\zeta) = 0, \quad \textup{with} \quad \Omega_i(\zeta) \sim \frac{1}{\zeta^4} \;\; \textup{as} \;\; \zeta \to 0. 
\end{equation}
The solutions without exponential growth at infinity are
\begin{align*}
    \Omega_0(\zeta) =\frac{1}{\zeta^4},\quad
    \Omega_1(\zeta) = \frac{1}{\zeta^4}+\frac{\beta}{2\zeta^2}.
\end{align*}
We remark that the eigenproblem for the operator $\Hs$ actually determines our eigenvalues to the leading order, i.e., $\Hs f = \lambda f$ has solutions in the class of functions with suitable decay when $\lambda = 2\beta(1-i)\;(i=0,1)$. Next, we consider the next order which solves:
\begin{equation*}
    \dr{\Hs - 2\beta(1-i)}\tilde\varphi^\out_i = 2\beta\tilde\alpha_{i,\nu}\Omega_i.
\end{equation*}                      
The solutions (without exponential growth and homogeneous modes) are
\begin{align}\label{prop1: outer sol exact form}
\begin{split}
    &(2\beta\tilde\alpha_{0,\nu})^{-1}\tilde\varphi^\out_0(\zeta) =  -\frac{\log(\zeta)}{\beta\zeta^4}-\frac{1}{\beta^2\zeta^6}-\frac{\beta\zeta^2-2}{\beta^2\zeta^4}e^{\frac{\beta\zeta^2}{2}}\int_{\zeta}^{+\infty}\frac{1}{r^3}e^{-\frac{\beta r^2}{2}}\;dr,\\
    &(2\beta\tilde\alpha_{1,\nu})^{-1}\tilde\varphi^\out_1(\zeta) = -\frac{1}{\zeta^4}e^{\frac{\beta\zeta^2}{2}}\int_{\zeta}^{+\infty}\frac{(\beta r^2+2)^2}{2\beta^2 r^3}e^{-\frac{\beta r^2}{2}}\;dr +\dr{-\frac{\log(\zeta)}{\beta}+\frac{1}{\beta^2\zeta^2}}\frac{\beta\zeta^2+2}{2\zeta^4}.
\end{split}
\end{align}
Their asymptotic behaviors, as $\zeta\to 0^+$, are
\begin{align}\label{prop1: asymptotic of outer solution}
\begin{split}
    &(2\beta\tilde\alpha_{0,\nu})^{-1}\tilde\varphi^\out_0(\zeta) =-\frac{1}{4\zeta^2}+\frac{\beta}{32}\dr{1-2\E-2\log(\frac{\beta}{2})-4\log(\zeta)}+\Oc(\zeta^2\log(\zeta)),\\
    &(2\beta\tilde\alpha_{1,\nu})^{-1}\tilde\varphi^\out_1(\zeta) =-\frac{1}{4\zeta^2}+\frac{\beta}{32}\dr{-3+2\E+2\log(\frac{\beta}{2})+4\log(\zeta)}+\Oc(\zeta^2\log(\zeta)),
\end{split}
\end{align}
where $\E$ is the Euler constant. We remark that certain cancellation occurs in \eqref{prop1: outer sol exact form} as $\zeta\to0^+$, so that the leading-order behavior is $1/\zeta^2$ rather than $\log(\zeta)/\zeta^4$ or $1/\zeta^6$, which may seem likely given the expression of \eqref{prop1: outer sol exact form}.\\[3pt]
\underline{Matching $\varphi^\inn_i$ and $\varphi^\out_i$:}
Now in the matching region $\nu\ll \zeta \ll 1$, we match the normalized inner solution $-\frac{1}{16\nu^4}\varphi^\inn_i(\zeta/\nu)$ with the outer solution $\varphi^\out_i(\zeta)$. By the asymptotic \eqref{Prop1: asyp at infty} and the choice of constants above we have
\begin{align}\label{prop1: inner eigen asymptotics}
\begin{split}
    -\frac{1}{16\nu^4}\varphi^\inn_i(\zeta/\nu) ={} \frac{1}{\zeta^4} +\frac{\beta(i-\tilde\alpha_{i,\nu})}{2\zeta^2}-\frac{\beta^2}{16}\Big(2(1-2i)+(10i-1)\tilde\alpha_{i,\nu}+(4-8i)\tilde\alpha_{i,\nu}\log(\zeta)\\
    -(4-8i)\tilde\alpha_{i,\nu}\log(\nu)
    +4\tilde\alpha^2_{i,\nu}(\log(\zeta)-\log(\nu)-\frac{5}{4})\Big)+\Oc(\nu^2),\quad \forall \;\zeta\gtrsim 1,
\end{split}
\end{align}
and by the asymptotic \eqref{prop1: asymptotic of outer solution} we have
\begin{align}\label{prop1: outer eigen asymptotics}
\begin{split}
    \varphi^\out_i(\zeta) = \frac{1}{\zeta^4}+\frac{\beta (i - \tilde\alpha_{i,\nu})}{2\zeta^2}+\frac{\beta^2\tilde\alpha_{i,\nu}(1-2i)}{16}\dr{1+2i-2\E-2\log(\beta)+2\log(2)-4\log(\zeta)}\\
    +\Oc(\tilde\alpha_{i,\nu}\zeta^2\log(\zeta)),\quad\forall \;\zeta\ll 1.
\end{split}
\end{align}
Note that the first two terms of \eqref{prop1: inner eigen asymptotics} and \eqref{prop1: outer eigen asymptotics} already match. Now we are ready to determine $\tilde\alpha_{i,\nu}$. First, due to the $\tilde\alpha_{i,\nu}\log(\nu)$ term in \eqref{prop1: inner eigen asymptotics}, we must have $\tilde\alpha_{i,\nu} = \Oc(1/|\log\nu|)$ in order to minimize the matching error of \eqref{prop1: inner eigen asymptotics} and \eqref{prop1: outer eigen asymptotics}. Furthermore, since the third term in \eqref{prop1: outer eigen asymptotics} is of size $\Oc(1/|\log\nu|)$, in order to improve the matching error by $|\log\nu|^{-1}$ the $\Oc(1)$ parts in the third term of \eqref{prop1: inner eigen asymptotics} must be canceled:
\[
     2(1-2i)-(4-8i)\tilde\alpha_{i,\nu}\log(\nu) = \Oc(1/|\log\nu|),
\]
and one simple choice is
\[
    \tilde\alpha_{i,\nu} = \frac{1}{2\log(\nu)}\quad i=0,1.
\]
With this choice of $\tilde\alpha_{i,\nu}$ we obtain the matching error:
\begin{equation}\label{prop1: eigenfunc matching}
    -\frac{1}{16\nu^4}\varphi^\inn_i(\zeta/\nu) - \varphi^\out_i(\zeta) = \Oc\dr{\frac{1}{\log(\nu)}},\quad \forall\; \frac{1}{4}\zeta_m\leq \zeta\leq 4\zeta_m.
\end{equation}
Finally, using $\varphi^\inn_{i}$ and $\varphi^\out_{i}$, we construct the global approximate eigenfunctions as
\begin{align*}
    \varphi_{i,\nu}(\zeta)&:= -\frac{1}{16\nu^4}\varphi^\inn_{i}(\zeta/\nu)\chi_m(\zeta)+(1-\chi_m(\zeta))\chi_\nu(\zeta)\varphi^\out_i(\zeta)\\
    &:= -\frac{1}{16\nu^4}\Lambda U(\zeta/\nu)\chi_\nu(\zeta)+\tilde\varphi_{i}(\zeta).
\end{align*}
\underline{Pointwise estimates:}
As for the pointwise estimate, we first note from \eqref{prop1: outer sol exact form} that
\[
    |\pa_\zeta^k\varphi^\out_{0}(\zeta)|\lesssim \frac{1}{\zeta^{6-k}}+\frac{\log(\zeta)}{\zeta^{4+k}},\quad |\pa_\zeta^k\varphi^\out_1(\zeta)|\lesssim \frac{\log(\zeta)}{\zeta^{2+k}}+\frac{1}{\zeta^{6-k}},\quad\forall\; \zeta\gg 1.
\]
Then, \eqref{prop1: pointwise est 1} follows from this far field estimate and the asymptotics  \eqref{Prop1: asymp at 0}\eqref{prop1: inner eigen asymptotics}. Note that the pointwise value of $\tilde\varphi_{0,\nu}$ is $\Oc(|\log\nu|^{-1})$ smaller than that of $\tilde\vp_{1,\nu}$ in the region $\zeta\approx 1$, because when $i = 0$ there is a gain of $1/|\log\nu|$ in the $\Oc(1/\zeta^2)$ term in \eqref{prop1: inner eigen asymptotics}. This is important in the upcoming derivation of modulation estimates. To derive the improved estimate near the origin, we note that for $V(\gamma) = V_2(\gamma)\text{ or }\tilde V_2(\gamma)$:
\[
    \left|\nu\pa_\nu\dr{\frac{1}{\nu^2}V(\zeta/\nu)}\right| = \left|\frac{1}{\nu^2}(\Lambda_\gamma V)(\zeta/\nu)\right| \lesssim \frac{\nu^2\zeta^2\log^2(1+\zeta/\nu)}{(\nu+\zeta)^6},
\]
Moreover, for $V(\gamma) = V_4(\gamma)\text{ or }\tilde V_4(\gamma)$ (the leading order of which is log growth at infinity), note that
\[
    |\pa_\zeta^k\nu\pa_\nu(\tilde\alpha_{1,\nu}V(\zeta/\nu))|\lesssim \frac{1}{|\log(\nu)|}\frac{\zeta^{2-k}}{(\nu+\zeta)^2}+\frac{\nu^2\zeta^2\log^2(1+\zeta/\nu)}{(\nu+\zeta)^4},\quad k = 0,1,2.
\]
Combining these facts gives us \eqref{prop1: improved est near the origin}. 
Now we estimate the generated error:
\begin{align*}
    &R_i :=  \Ls^\zeta_\nu \varphi_{i,\nu}(\zeta) - 2\beta(1-i+\tilde\alpha_{i,\nu})\varphi_{i,\nu}(\zeta)\\
     &= -\frac{1}{16\nu^4}\chi_m(\zeta)\dr{\Ls^\zeta_\nu-2\beta(1-i+\tilde\alpha_{i,\nu})}\varphi^\inn_{i}(\zeta/\nu)+(1-\chi_m(\zeta))\chi_\nu(\zeta)\dr{\pa^2_\zeta+\frac{1}{\zeta}\pa_\zeta-\pa_\zeta\Psi_{U_\nu}(\zeta)\cdot\pa_\zeta}\varphi^\out_{i}(\zeta)\\
     &\quad -\pa_\zeta U_\nu(\zeta)\cdot\pa_\zeta\Psi^\out_{i,m}(\zeta)+2(1-\chi_m(\zeta))\chi_\nu(\zeta)U_\nu(\zeta)\varphi^\out_{i}(\zeta)-(1-\chi_m(\zeta))\chi_\nu(\zeta)\beta\Lambda\varphi^\out_i(\zeta)\\
     &\quad -\frac{1}{8\nu^5}\pa_\gamma\varphi_i^\inn(\zeta/\nu)\chi'_m(\zeta)-\frac{1}{16\nu^4}\varphi^\inn_i(\zeta/\nu)\chi''_m(\zeta)-\frac{1}{16\nu^4\zeta}\varphi^\inn_i(\zeta/\nu)\chi'_m(\zeta)+\frac{1}{16\nu^4}\varphi^\inn_i(\zeta/\nu)\chi'_m(\zeta)\pa_\zeta\Psi_{U_\nu}(\zeta)\\
    &\quad -\dr{\pa_\zeta\Psi^\inn_{i,m}(\zeta) - \chi_m(\zeta)\pa_\zeta\Psi_{i,*}^\inn(\zeta)}\pa_\zeta U_\nu(\zeta)+\beta\zeta\chi'_m(\zeta)\frac{1}{16\nu^4}\varphi^\inn_i(\zeta/\nu)\\
    &\quad -2\chi'_m\chi_\nu\pa_\zeta\varphi^\out_i - \chi''_m\chi_\nu\varphi^\out_i-\frac{1}{\zeta}\chi'_m\chi_\nu\varphi^\out_i+\chi'_m\chi_\nu\varphi^\out_i(\zeta)\pa_\zeta\Psi_{U_\nu}(\zeta)+\beta\zeta\chi'_m\chi_\nu\varphi^\out_i(\zeta)\\
    &\quad +2(1-\chi_m)\chi'_{|\log\nu|}\pa_\zeta\varphi^\out_i + (1-\chi_m)\chi''_{|\log\nu|}\varphi^\out_i+\frac{1}{\zeta}(1-\chi_m)\chi'_{|\log\nu|}\varphi^\out_i\\
    &\quad-(1-\chi_m)\chi'_{|\log\nu|}\varphi^\out_i(\zeta)\pa_\zeta\Psi_{U_\nu}(\zeta)-\beta\zeta(1-\chi_m)\chi'_{|\log\nu|}\varphi^\out_i(\zeta)\\
    &:= R^\inn_i+R^\out_i+R^{bd}_i,
\end{align*}
where we denote
\begin{align*}
    -(\pa^2_\zeta+\frac{1}{\zeta}\pa_\zeta)\Psi^\inn_{i,m}(\zeta) &= -\frac{1}{16\nu^4}\varphi^\inn_i(\zeta/\nu)\chi_m(\zeta),&  -(\pa^2_\zeta+\frac{1}{\zeta}\pa_\zeta)\Psi^\inn_{i,*}(\zeta) &= -\frac{1}{16\nu^4}\varphi^\inn_i(\zeta/\nu) ,\\
    -(\pa^2_\zeta+\frac{1}{\zeta}\pa_\zeta)\Psi^\out_{i,m}(\zeta) &=\varphi^\out_i(\zeta)(1-\chi_m(\zeta)).&
\end{align*}
We assume that the supports of $\chi'_m$ and $\chi'_{|\log\nu|}$ are disjoint so that the terms containing $\chi'_m\chi'_{|\log\nu|}$ are all zero which we do not display in the expression of $R_i$. Note that for the outer solutions, we treat $(1-\chi_m)\chi_\nu\varphi^\out_{i}$ as a whole, which is different from the case of the inner solutions. This is because the singularity of $\varphi^\out_i$ at the origin only allows the existence of the Poisson field of $(1-\chi_m)\varphi^\out_i$, not of $\varphi^\out_i$. As for the inner error, note that
\begin{align*}
    R^\inn_i(\zeta) = -\frac{16}{\nu^6}R^\gamma_i(\zeta/\nu)\chi_m(\zeta).  
\end{align*}
By \eqref{prop1: inner error}, we have, for $\zeta\in[0,\zeta_m]$,
\[
    \left|\pa^k_{\zeta}\dr{\frac{1}{\nu^6}R^\gamma_i(\zeta/\nu)}\right|\lesssim
    \frac{\nu^2\zeta^{2-k}\log^2(2+\zeta/\nu)}{(\nu+\zeta)^4}+\frac{\zeta^{2-k}|\log(\nu+\zeta)|}{(\nu+\zeta)^2|\log\nu|}.
\] 
As for the outer error, we calculate that
\begin{align*}
    R^\out_i &= (1-\chi_m)\chi_\nu\dr{\pa^2_\zeta+\frac{1}{\zeta}\pa_\zeta-\pa_\zeta\Psi_{U_\nu}\cdot\pa_\zeta}\varphi^\out_{i}-\pa_\zeta U_\nu\cdot\pa_\zeta\Psi^\out_{i,m}\\
    &\quad+2(1-\chi_m)\chi_\nu U_\nu\varphi^\out_{i}-(1-\chi_m)\chi_\nu\beta\Lambda\varphi^\out_i\\
    &= (1-\chi_m)\chi_\nu(\Hs - 2\beta(1-i+\tilde\alpha_{i,\nu}))\varphi^\out_i(\zeta)+(1-\chi_m)\chi_\nu\dr{\frac{4\zeta}{\zeta^2+\nu^2}-\frac{4}{\zeta}}\pa_\zeta\varphi^\out_{i}\\
    &\quad+\frac{32\nu^2\zeta}{(\zeta^2+\nu^2)^3}\pa_\zeta\Psi^\out_{i,m}+\frac{16\nu^2}{(\zeta^2+\nu^2)^2}(1-\chi_m)\chi_\nu\varphi^\out_{i}(\zeta)\\
    &=\tilde\alpha_{i,\nu}(1-\chi_m)\chi_\nu\tilde\varphi^\out_{i}(\zeta)-\frac{4\nu^2}{\zeta(\zeta^2+\nu^2)}(1-\chi_m)\chi_\nu\pa_\zeta\varphi^\out_{i}(\zeta)\\
    &\quad+\frac{32\nu^2\zeta}{(\zeta^2+\nu^2)^3}\pa_\zeta\Psi^\out_{i,m}+\frac{16\nu^2}{(\zeta^2+\nu^2)^2}(1-\chi_m)\chi_\nu\varphi^\out_{i}(\zeta).
\end{align*}
 Thus, by the asymptotic behavior at $\zeta\to\infty$ of the outer solutions as well as their Poisson fields, we have
 \[
    |\pa^k_\zeta R^\out_i(\zeta)|\lesssim \frac{1}{\log^2(\nu)}\cdot\frac{\log(1+\zeta)}{\zeta^2},\quad \forall \zeta\geq \zeta_m,\;\forall\; k = 0,1,2.    
 \]
 Now we come to the boundary error $R^{bd}_i$. First, note that $\text{supp}\,\chi'_m\subset [\zeta_m,2\zeta_m]$, by the matching condition \eqref{prop1: eigenfunc matching}, the terms involving the derivatives of $\chi_m$ cancel in pairs. For example,
\[
    \left|\pa^k_\zeta\dr{-\frac{1}{8\nu^4}\pa_\zeta\varphi^\inn_i(\zeta/\nu)\chi'_m(\zeta)-2\pa_\zeta\varphi^\out_i(\zeta)\chi'_m(\zeta)}\right| \lesssim \frac{1}{|\log\nu|},  
\]
and the rest are similar. Second, we note that $\text{supp}\,\chi'_{|\log\nu|}\subset [ |\log\nu|,2|\log\nu| ]$, and we have the decay property
\[
    |\varphi^\out_i(\zeta)| \lesssim \frac{\log(\zeta)}{(1+\zeta)^2}.
\]
Then, any term involving the derivative of $\chi_\nu$ is of size at most $\Oc(1/|\log\nu|)$.
It then remains to estimate the terms involving Poisson fields. Note that for a 2D radial symmetric Poisson problem with Neumann boundary condition (we assume $S$ to have certain regularity which is the case in our problem):
\[
    -(\pa^2_\zeta +\frac{1}{\zeta}\pa_\zeta)\Psi(\zeta) = S(\zeta), 
\]
the Poisson field satisfies
\[
    \pa_\zeta\Psi(\zeta) = \frac{1}{\zeta} \int_{0}^{\zeta}rS(r)\;dr. 
\]
Using this, we obtain that
\[
    \pa_\zeta\Psi^\inn_{i,m}(\zeta) - \chi_m(\zeta)\pa_\zeta\Psi_{i,*}^\inn(\zeta)  \equiv 0,\quad \forall\;\zeta\in [0,\zeta_m],
\]
and
\[
    \pa_\zeta\Psi^\inn_{i,m}(\zeta) - \chi_m(\zeta)\pa_\zeta\Psi_{i,*}^\inn(\zeta) \equiv \text{const.},\quad \forall\;\zeta\in [2\zeta_m,+\infty).
\]
Moreover, for $\zeta>\zeta_m$, $\pa_\zeta U_\nu\lesssim \nu^2/\zeta^5$, 
which completes the pointwise estimate for the boundary error.
Finally, the estimate of the partial mass follows directly from the pointwise estimate of the generated error.  
\end{proof}

\subsection{Coercivity of the Linearized Operator}
The main goal of this section is to establish the coercivity of the linearized operator $\Ls^\zeta_\nu$ (after a slight modification) under a certain adapted inner products.

\subsubsection{Properties of the Operator $\Ms^\zeta_\nu$}
To begin with, we collect several properties of $\Ms^\zeta_\nu$, in particular, boundedness and definiteness.
\begin{lemma}\label{Lemma: coercivity of gradient M}
    Let $f$ be a function on $\reall^2$ with $\int_{\reall^2}(1+|\y|^2)f^2+\int_{\reall^2}\frac{|\nabla f|^2}{U}<+\infty$. Then, there exist universal constants $c,C>0$, such that
    \begin{equation}
        \int_{\reall^2}U|\nabla\Ms f|^2 \geq c\int_{\reall^2}\frac{|\nabla f|^2}{U} -  C\dr{|\scl{f}{\Lambda U}|^2+|\scl{f}{\nabla U}|^2}.
    \end{equation}
\end{lemma}
In the parabolic variables, the inequality equivalently becomes
\[
\int U_\nu|\nabla\Ms^\zeta_\nu f|^2\geq c\int\frac{|\nabla f|^2}{U_\nu} -C\dr{\nu^2|\scl{f}{\Lambda U_\nu}|^2+\nu^4|\scl{f}{\nabla U_\nu}|^2}.
\]

\begin{proof}
    The proof follows the same tactic as in \cite{Raphael2014}. 
    First, by Hardy's inequality and the estimates of Poisson fields \eqref{appendix: pointwise est of Poisson field radial}\eqref{appendix: pointwise est of Poisson field nonradial}, we have the a priori bounds
    \begin{align}\label{Lemma coercivity of M: pre-bounds}
    \begin{split}
        &\int_{\reall^2} (1+|\y|^2)g^2 \lesssim \int_{\reall^2}(1+|\y|^2)^2|\nabla g|^2 = 8\int_{\reall^2}\frac{|\nabla g|^2}{U},\\
        & \int_{\reall^2} \frac{|\nabla \Psi_g|^2}{1+|\y|^2} \lesssim \int_{\reall^2} (1+|\y|^2)g^2 \lesssim \int_{\reall^2}(1+|\y|^4)|\nabla g|^2,\\
        &\int_{\reall^2} \frac{|\Psi_g|^2}{1+|\y|^4} \lesssim \int_{\reall^2} (1+|\y|^2)g^2 \lesssim \int_{\reall^2}(1+|\y|^4)|\nabla g|^2.
    \end{split}
    \end{align}
    Moreover, note that 
    \begin{align*}
        \int_{\reall^2}U|\nabla \Ms f|^2 = \int_{\reall^2}U\da{\nabla\dr{\frac{f}{U}}-\nabla \Psi_f}^2 \geq \frac{1}{2}\int_{\reall^2}U\da{\nabla\dr{\frac{f}{U}}}^2 - \int_{\reall^2}U|\nabla\Psi_f|^2, 
    \end{align*}
    and through integration by parts,
    \begin{align*}
        \int_{\reall^2}U\da{\nabla\dr{\frac{f}{U}}}^2 &= \int_{\reall^2}U\da{\frac{\nabla f}{U}-\frac{f\nabla\Psi_U}{U}}^2\\
        &=\int_{\reall^2}\frac{|\nabla f|^2}{U}+\frac{|\nabla\Psi_U|^2f^2}{U} - \int_{\reall^2} \frac{2f\nabla f\cdot\nabla\Psi_U}{U}\\
        & = \int_{\reall^2}\frac{|\nabla f|^2}{U}+\frac{|\nabla\Psi_U|^2f^2}{U} + \int_{\reall^2} f^2\nabla\cdot\dr{\frac{\nabla\Psi_U}{U}}\\
        & = \int_{\reall^2}\frac{|\nabla f|^2}{U} - f^2.
    \end{align*}
    Then, by Poisson field estimates, we obtain the sub-coercivity estimate:
    \begin{equation}\label{Lemma coercivity of M: subcoercivity}
         \int_{\reall^2}U|\nabla \Ms f|^2\gtrsim \int_{\reall^2}\frac{|\nabla f|^2}{U} - C\int_{\reall}(1+|\y|)f^2.
    \end{equation}
    Assume, by contradiction, that there exists a sequence of functions $\{f_n\}$ that satisfies 
    \begin{equation}
        \int_{\reall^2}(1+|\y|^2)f^2_n<+\infty,\quad\int_{\reall^2}\frac{|\nabla f_n|^2}{U}=1,\quad \scl{f_n}{\Lambda U} = \scl{f_n}{\nabla U} = 0,\quad\text{and } \lim_{n\to+\infty}\int_{\reall^2}U|\nabla\Ms f_n|^2 = 0.
    \end{equation}
    Then, by \eqref{Lemma coercivity of M: pre-bounds} and $\int |\Delta\Psi_{f_n}|^2 = \int f_n^2<+\infty$, we know from Sobolev embedding that there exist some $f$ and $\Psi$, such that (up to a subsequence)
    \begin{align*}
        &f_n\rightharpoonup f,\; \text{ in } H^1(\reall^2),\quad\text{and }\quad f_n\to f,\; \text{ in } L^2_{\text{loc}}(\reall^2),\\
        &\Psi_{f_n}\to \Psi,\; \text{ in } H^1_{\text{loc}}(\reall^2).
    \end{align*}
    In particular, the convergence above holds in the sense of distribution (i.e., in $\mathcal{D'}(\reall^2)$), and $-\Delta \Psi = f$ in $\Dc'(\reall^2)$. Since $\int_{\reall^2}U|\nabla\Ms f_n|^2 \to 0$, $\nabla \Ms f_n \to 0$ in $\mathcal{D'}(\reall^2)$ so that $\nabla \dr{\frac{f}{U}-\Psi} =  0$ in $\mathcal{D'}(\reall^2)$. In summary, we have
    \[
       \begin{cases}
           -\Delta \Psi = f, \\[3pt]
           \nabla \dr{\frac{f}{U}-\Psi} =  0,
       \end{cases}
       \text{ in } \Dc'(\reall^2).
    \]
    From standard lower semi-continuity estimates, we have
    \begin{align*}
        \int_{\reall^2} \frac{|\nabla f|^2}{U}\leq 1,\quad
        \int_{\reall^2}\frac{|\nabla\Psi|^2}{1+|\y|^2}+\frac{\Psi^2}{1+|\y|^4}\lesssim 1.
    \end{align*}
    Then, by elliptic regularity which is bootstrapped by the relation $\nabla \dr{\frac{f}{U}-\Psi} =  0$, we know that $(f,\Psi)\in C^\infty(\reall^2)$, and in particular $\Psi = \Psi_f$. By \textit{Lemma 2.1} in \cite{Raphael2014}, we obtain
    \[
        f\in \text{Span}\{\Lambda U, \pa_{y_1} U,\pa_{y_2}U\}.
    \]
    Since the orthogonality conditions pass to $f$, i.e., $\scl{f}{\Lambda U} = \scl{f}{\nabla U} = 0$, we deduce that $f \equiv 0$. On the other hand, since by assumption $\int f^2_n(1+|\y|^2)$ are uniformly bounded, by local strong convergence we have
    \[
        \lim_{n\to+\infty}\int_{\reall^2}f^2_n(1+|\y|) = \int_{\reall^2}f^2(1+|\y|).
    \]
    Then, by sub-coercivity \eqref{Lemma coercivity of M: subcoercivity}, 
    \begin{equation}
        \int_{\reall^2}f^2(1+|\y|) \gtrsim \lim_{n\to+\infty} \frac{1}{C} \dr{\int_{\reall^2}\frac{|\nabla f_n|^2}{U}-\int_{\reall^2}U|\nabla\Ms f_n|^2} = \frac{1}{C},
    \end{equation}
    which contradicts the fact that $f\equiv 0$. In summary, there exists some $c>0$ such that
    \begin{equation}\label{Lemma coercivity of M: coercivity for f}
        \int_{\reall^2}U|\nabla\Ms f|^2\geq c\int_{\reall^2}\frac{|\nabla f|^2}{U},
    \end{equation}
    for any $f$ with $\int\frac{|\nabla f|^2}{U}<+\infty$ satisfying the orthogonality conditions.
    Finally, for general $f$, define
    \[
        F := f - \frac{\scl{f}{\Lambda U}}{\scl{\Lambda U}{\Lambda U}}\Lambda U - \sum_{i=1,2}\frac{\scl{f}{\pa_{y_i} U}}{\scl{\pa_{y_i} U}{\pa_{y_i} U}}\pa_{y_i}U.   
    \]
    Applying \eqref{Lemma coercivity of M: coercivity for f} to $F$ completes the proof.
\end{proof}

The following lemma implies norm equivalence on some finite codimensional function space, which will motivates our design of the adapted inner product.   
\begin{lemma}\label{lemma: properties of the adapted inn product}
    The quadratic form $(f,g)\mapsto \int_{\reall^2}f\Ms^\zeta_\nu g$ is symmetric. Moreover, for any $f$ such that $\int_{\reall^2} f^2/U_\nu<+\infty$ and $\int_{\reall^2}|\nabla f|^2/U_\nu<+\infty$, we have the estimates
    \begin{equation}\label{Lemma_est of M: L^2 est}
        \int_{\reall^2}U_\nu |\Ms^\zeta_\nu f|^2\lesssim \int_{\reall^2}\frac{f^2}{U_\nu},
    \end{equation}
       \begin{equation}\label{Lemma_est of M: H^1 est}
        \int_{\reall^2}U_\nu |\nabla\Ms^\zeta_\nu f|^2\lesssim \int_{\reall^2}\frac{|\nabla f|^2}{U_\nu},
    \end{equation}
    and
    \begin{align}\label{Lemma_est of M: L^2 coercivity}
        \int_{\reall^2}f\Ms^\zeta_\nu f \geq \frac{1}{C}\int_{\reall^2}\frac{f^2}{U_\nu} - C\dr{\nu^4|\scl{f}{\Lambda U_\nu}|^2+\nu^6|\scl{f}{\nabla U_\nu}|^2+|\scl{f}{1}|^2},
    \end{align}
    for some universal $C>0$. In addition, if $\int_{\reall^2} f = 0$, we have the definiteness
    \begin{equation}\label{Lemma_properties of the adapted inner prod: definiteness}
        \int_{\reall^2}f\Ms^\zeta_\nu f\geq 0.
    \end{equation}
\end{lemma}
\begin{proof}
    The symmetry of the quadratic form follows from the symmetry: \[
    \scl{f}{\Psi_g} = \scl{\Psi_f}{g} = -\frac{1}{2\pi}\int\int \log|\x-\y|f(\x)g(\y)\;d\x d\y. 
    \]
    By \eqref{appendix: pointwise est of Poisson field radial and nonradial, parabolic} (taking $\alpha=1$) and Hardy's inequality,
    \begin{equation*}
        \int U_\nu|\nabla\Psi_f|^2 \lesssim \int\frac{\nu^2(1+\log(\zeta/\nu))}{(\nu+\zeta)^6}\int f^2(\nu+\zeta)^2\lesssim \frac{1}{\nu^2}\int|\nabla f|^2(\nu+\zeta)^4\lesssim\int\frac{|\nabla f|^2}{U_\nu}.
    \end{equation*}
    This completes the proof of \eqref{Lemma_est of M: H^1 est}. As for the rest, see \textit{Lemma 2.1} and \textit{Proposition 2.3} in \cite{Raphael2014}.
\end{proof}
\subsubsection{Adapted Inner Product and Coercivity}
Define the weight functions
\begin{align}\label{def of weight functions}
   &\varrho_\nu(\zeta) := e^{-\frac{\beta \zeta^2}{2}},\quad \varrho(\gamma): =   e^{-\frac{\beta \nu^2\gamma^2}{2}},
\end{align}
Observe the following two decompositions of the linearized operator (written in soliton variables):
\begin{align}\label{two decomp of L in the soliton scale}
    \Ls f = \nabla\cdot(U\nabla\Ms f) - b\Lambda f = \frac{1}{\om}\nabla\cdot(\om\nabla f) + 2(U-b)f - \nabla U\cdot \nabla \Psi_f, 
\end{align}
where we denote $b:= \beta\nu^2$ and $\om:=\frac{1}{U}\vr$. In the near field, i.e., $\gamma\ll \frac{1}{\nu}$, according to the first decomposition in \eqref{two decomp of L in the soliton scale} the scaling term $b\Lambda f$ becomes negligible, and the coercivity of $\nabla\Ms$ (Lemma \ref{Lemma: coercivity of gradient M}) leads to the coercivity of $\Ls$ with some appropriate inner product in this domain. In the far field, i.e., $\gamma\gg \frac{1}{\nu}$, the terms $\nabla U\cdot\nabla \Psi_f$ and $Uf$ become negligible due to the fast decay of $\nabla U$ and $U$, according to the second decomposition in \eqref{two decomp of L in the soliton scale}. Therefore, $\Ls$ will be coercive with the weighted $L^2$-inner product (with $\om$ as the weight function) in this domain. In order to obtain coercivity in the whole domain, we define the mixed inner product which adapts to both coercivity structures:
\[
    \scl{f}{g}_*: = \int_{\reall^2} fg\bar\chi^2_\nu\om - \int_{\reall^2} \sqrt{\vr}\bar\chi_\nu f \Psi_{\sqrt{\vr}\bar\chi_\nu g} = \int_{\reall^2}\sqrt{\vr}\bar\chi_\nu f\Ms(\sqrt{\vr}\bar\chi_\nu g) ,
\]
or equivalently, in the parabolic variables
\begin{equation}\label{def of adapted inner scl, parabolic}
    \scl{\eps}{\vartheta}_{\nu,*} := \int_{\reall^2}\sqrt{\vr_\nu}\chi_\nu\eps \Ms^\zeta_\nu(\sqrt{\vr_\nu}\chi_\nu\vartheta),
\end{equation}
where we recall $\bar\chi_\nu(\gamma): = \chi(\gamma\nu/|\log\nu|)$ and $\chi_\nu(\zeta):=\chi(\zeta/|\log\nu|)$. One remark: thanks to Lemma \ref{lemma: properties of the adapted inn product}, we know that for any $f$ satisfying the orthogonality condition
\[
    \int_{\reall^2}f\Lambda U\chi(2\gamma\nu/c_0)\sqrt{\vr} =\int_{\reall^2}f\nabla U\chi(2\gamma\nu/c_0)\sqrt{\vr}=\int_{\reall^2}f\chi(\gamma\nu/\zeta^*)= 0,
\]
where $c_0>0$ is some fixed constants and $b$ is small enough, there holds the equivalence of norms:
\[
    \frac{1}{C}\int f^2\bar\chi^2_\nu\om \leq \scl{f}{f}_*\leq C\int f^2\bar\chi^2_\nu\om.
\]
In addition, we need to modify the linearized operator a little bit to adapt to the inner product:
\[
    \Lst f:= \Delta f - \nabla U\cdot \nabla \Psi_{\sqrt{\vr}\bar\chi_\nu f} - \nabla f\cdot\nabla \Psi_U + 2Uf - b\Lambda f,
\]
or equivalently,
\[
    \Lst^\zeta_\nu \eps:= \Delta\eps - \nabla U_\nu\cdot\nabla\tilde\Psi_\eps -\nabla\eps\cdot\nabla\Psi_{U_\nu} + 2U_\nu\eps-\beta\Lambda \eps,
\]
where $\tilde\Psi_\eps:=\Psi_{\chi_\nu\sqrt{\vr_\nu}\eps}$.
Applying the aforementioned ideas, we are able to prove the following proposition. 
\begin{proposition}[Coercivity estimate]\label{Main coercivity prop}
    There exist constants $\delta,\zeta_*,C,b_*>0$, such that for any $0<b<b_*$, and any $\eps$ satisfying $\int \frac{\eps^2+|\nabla\eps|^2}{U}<+\infty$ and 
    \[
        \int_{\reall^2}\eps\Lambda U\chi(\gamma\nu/\zeta_*) =\int_{\reall^2}\eps\nabla U\chi(\gamma\nu/\zeta_*)= 0,
    \]
    we have
    \begin{equation}\label{prop2: coercivity}
        \scl{\Lst \eps}{\eps}_*\leq -\delta\dr{\int_{\reall^2}\frac{|\nabla \eps|^2\bar\chi^2_\nu}{U}\vr+b\int_{\reall^2}\frac{ \eps^2\bar\chi^2_\nu}{U}\vr} +C\nu^{100}\|\eps\|^2_{L^\infty(\gamma\geq |\log\nu|/\nu)}
    \end{equation}
\end{proposition}
\begin{proof}
    Define $\bar \chi_0(\gamma):= \chi(\gamma\nu/\zeta_0)$, and decompose
    \[
        \eps = \bar \chi_0 \eps+(1-\bar \chi_0)\eps:=\eps_1 + \eps_2,
    \]
    where $0<\zeta_0=\zeta_0(\eps)\ll 1$ is a parameter to be determined. Though the specific value of $\zeta_0$ depends on $\eps$, we will see that there exists a universal constant $\zeta_*>0$ such that $\zeta_*<\zeta_0$ for any $\eps$.  
    Thus,
    \[
        \scl{\Lst \eps}{\eps}_* = \scl{\Lst \eps_1}{\eps_1}+\scl{\Lst \eps_2}{\eps_2}+\scl{\Lst\eps_1}{\eps_2}+\scl{\Lst \eps_2}{\eps_1}.
    \]
    For brevity, in the following we denote
    \[
        \|f\|_{L^2_\om} := \dr{\int\frac{f^2\vr}{U}}^\frac{1}{2}.
    \]
    \noindent\underline{Coercivity of $\scl{\Lst \eps_1}{\eps_1}_*$:} Since $\bar\chi_\nu\eps_1 = \eps_1$ and
    \[
        \sqrt{\vr}\bar\chi_\nu\Lst \eps_1 = \Ls(\sqrt{\vr}\eps_1) + (1-\sqrt{\vr}\bar\chi_\nu)\nabla\Psi_{\sqrt{\vr}\eps_1}\cdot\nabla U+[\bar\chi_\nu\sqrt{\vr},\,\Ls+\nabla U\cdot\nabla\Psi_{\cdot}]\eps_1,
    \]
    by integration by parts, we have
    \begin{align*}
        \scl{\Lst \eps_1}{\eps_1}_* = -\int U|\nabla\Ms(\eps_1\sqrt{\vr})|^2+b\int\sqrt{\vr}\eps_1\y\cdot\nabla\Ms(\sqrt{\vr}\eps_1)+ \int (1-\sqrt{\vr}\bar\chi_\nu)\nabla\Psi_{\sqrt{\vr}\eps_1}\cdot\nabla U\Ms(\sqrt{\vr}\eps_1) \\
        + \int [\bar\chi_\nu\sqrt{\vr},\,\Ls+\nabla U\cdot\nabla\Psi_{\cdot}]\eps_1\Ms(\sqrt{\vr}\eps_1).
    \end{align*}
    By Lemma \ref{Lemma: coercivity of gradient M}, we have
    \begin{align*}
        -\int U|\nabla\Ms(\eps_1\sqrt{\vr})|^2 &\leq -c\int \frac{|\nabla(\eps_1\sqrt{\vr})|^2}{U} +C\dr{|\scl{\eps_1\sqrt{\vr}}{\Lambda U}|^2+|\scl{\eps_1\sqrt{\vr}}{\nabla U}|^2}.
    \end{align*}
     By the local orthogonality conditions
    \[
          \int_{\reall^2}\eps\Lambda U\chi(2\gamma\nu/\zeta_*) =\int_{\reall^2}\eps\nabla U\chi(2\gamma\nu/\zeta_*)=0
    \]
    where $\zeta_* < \zeta_0$, we obtain
    \begin{align*}
        |\scl{\eps_1\sqrt{\vr}}{\Lambda U}| &= \da{\int \eps_1(\sqrt{\vr}-\chi(2\gamma\nu/\zeta_*))\Lambda U} \\
         &\leq \dr{\int \frac{\eps_1^2\vr}{U}}^{\frac{1}{2}}\dr{\int (\Lambda U)^2U\rho^{-1}(\sqrt{\vr}-\chi(2\gamma\nu/\zeta_*))^2}^{\frac{1}{2}}\\
        &\lesssim \nu^2 \dr{\int \frac{\eps_1^2\vr}{U}}^{\frac{1}{2}}
        \lesssim \nu \dr{\int \frac{|\nabla\eps_1|^2\vr}{U}}^{\frac{1}{2}}
    \end{align*}
    where we use the Poincar\'e inequality (taking $\alpha=2$ in this case):
    \begin{equation}\label{Prop_main coer: Poincare for eps_in}
        \int \eps_1^2\vr(1+\gamma)^{2+\alpha} \leq C\frac{\zeta^\alpha_0}{\nu^\alpha}\int \eps^2_1 (1+\gamma)^2\leq C^2\int |\nabla\eps_1|^2(1+\gamma)^4\leq C^3\frac{\zeta_0^\alpha}{\nu^{\alpha}} \int \frac{|\nabla \eps_1|^2\vr}{U},\quad \forall \; \alpha\geq 0,
    \end{equation}
    for some universal $C>0$ when $\zeta_0$ is sufficiently small.
    Similar estimate holds for $|\scl{\eps_1\sqrt{\vr}}{\nabla U}|$. Thus, when $b$ is sufficiently small (recall that $b:= \beta\nu^2$), there exists a universal $\delta>0$ such that
    \begin{equation}\label{Prop_main coer: inner coer part 1}
        -\int U|\nabla \Ms(\eps_1\sqrt{\vr})|^2 \leq -\delta \int \frac{|\nabla \eps_1|^2\vr}{U}.
    \end{equation}
    Note, by \eqref{Prop_main coer: Poincare for eps_in}, that
    \[
        \int \frac{|\eps_1\nabla(\sqrt{\vr})|^2}{U} = \int \frac{b^2|\y|^2\eps_1^2\vr}{4U}\lesssim \zeta_0^4 \int \frac{|\nabla\eps_1|^2\rho}{U}.
    \]
    Therefore, when $\zeta_0$ is small enough,
    \[
        \frac{1}{2}\int \frac{|\nabla \eps_1|^2\vr}{U}\leq \int \frac{|\nabla (\sqrt{\vr}\eps_1)|^2}{U}\leq2\int \frac{|\nabla \eps_1|^2\vr}{U}.
    \]
    In addition, by Lemma \ref{lemma: properties of the adapted inn product} and Poincar\'e inequality,
    \[
        \da{b\int\sqrt{\vr}\eps_1\y\cdot\nabla\Ms(\sqrt{\vr}\eps_1)}\lesssim b\dr{\int\frac{\vr\eps_1^2|\y|^2}{U}}^{\frac{1}{2}}\dr{\int\frac{|\nabla(\sqrt{\vr}\eps_1)|^2}{U}}^{\frac{1}{2}}\lesssim \zeta_0^2\int\frac{|\nabla\eps_1|^2\vr}{U}.
    \]
    Thus, when $b$ and $\zeta_0$ are sufficiently small, we have
    \[
        -\int U|\nabla\Ms(\eps_1\sqrt{\vr})|^2+b\int\sqrt{\vr}\eps_1\y\cdot\nabla\Ms(\sqrt{\vr}\eps_1)\leq -\delta\int\frac{|\nabla\eps_1|^2\vr}{U}
    \]
    for some universal $\delta>0$.
    Next, by \eqref{appendix: pointwise est of Poisson field by weighted L^2 norm} and \eqref{appendix: pointwise est of Poisson field by weighted L^2 norm, when int u = 0}, we know that
    \begin{align*}
        |\Psi_{\eps_1\sqrt{\vr}}(\y)|^2 \lesssim \log^2(4+|\y|)\int \frac{\eps_1^2\vr}{U},\quad
        |\nabla\Psi_{\eps_1\sqrt{\vr}}(\y)|^2 \lesssim \int \frac{|\nabla(\sqrt{\vr}\eps_1)|^2}{U}.
    \end{align*}
    Therefore, by the above pointwise estimates of the Poisson field and \eqref{Lemma_est of M: L^2 est},
    \begin{align*}
        &\da{\int (1-\sqrt{\vr}\bar\chi_\nu)\nabla\Psi_{\sqrt{\vr}\eps_1}\cdot\nabla U\Ms(\sqrt{\vr}\eps_1)}\\
        \leq &\dr{\int \frac{|\nabla(\sqrt{\vr}\eps_1)|^2}{U}}^{\frac{1}{2}}\dr{\int (1-\sqrt{\vr}\bar\chi_\nu)^2\frac{|\nabla U|^2}{U}}^{\frac{1}{2}}\dr{\int U|\nabla \Ms(\sqrt{\vr}\eps_1)|^2}^{\frac{1}{2}}\\
        \leq & \dr{\int \frac{|\nabla(\sqrt{\vr}\eps_1)|^2}{U}}^{\frac{1}{2}}\dr{\int (1-\sqrt{\vr}\bar\chi_\nu)^2\frac{|\nabla U|^2}{U}}^{\frac{1}{2}}\dr{\int \frac{\eps_1^2\rho}{U}}^{\frac{1}{2}}.
    \end{align*}
    Note that when $\sqrt{b}\gamma\ll 1$, $1-\sqrt{\vr} = \frac{b\gamma^2}{4}+\Oc(b^2\gamma^4)$. Then, we have the estimates
    \begin{align*}
        \da{\int (1-\sqrt{\vr}\bar\chi_\nu)^2\frac{|\nabla U|^2}{U}}&\lesssim \da{\int_{\{|\y|<b^{-\frac{1}{3}}\}} \frac{b^2\gamma^4}{(1+\gamma)^6}}+\da{\int_{\{|\y|\geq b^{-\frac{1}{3}}\}} \frac{1}{(1+\gamma)^6}}
        \lesssim b^{\frac{5}{4}}.
    \end{align*}
    Combining with the previous estimate and \eqref{Prop_main coer: Poincare for eps_in}, we obtain
    \begin{equation}\label{Prop_main coer: inner coer part 2}
        \da{\int (1-\sqrt{\vr}\bar\chi_\nu)\nabla\Psi_{\sqrt{\vr}\eps_1}\cdot\nabla U\Ms(\sqrt{\vr}\eps_1)}\lesssim b^{\frac{1}{8}}\int \frac{|\nabla \eps_1|^2\vr}{U}.
    \end{equation}
    For the remaining terms, since $\nabla(\bar\chi_\nu) \eps_1 = 0$ (because of their disjoint supports),
    \[
        [\bar\chi_\nu\sqrt{\vr},\,\Ls+\nabla U\cdot\nabla\Psi_{\cdot}]\eps_1 = -2\nabla(\sqrt{\vr})\cdot\nabla\eps_1-\Delta(\sqrt{\vr})\eps_1+\nabla(\sqrt{\vr})\cdot\nabla\Psi_U\eps_1+b\y\cdot\nabla(\sqrt{\vr})\eps_1.
    \]
    Therefore, by \eqref{Prop_main coer: Poincare for eps_in} we have
    \begin{equation*}
        \int \frac{|[\bar\chi_\nu\sqrt{\vr},\,\Ls+\nabla U\cdot\nabla\Psi_{\cdot}]\eps_1|^2}{U} \lesssim \int \frac{b^2\gamma^2|\nabla\eps_1|^2\vr+\dr{b^2+b^4\gamma^4+\frac{b^2\gamma^2}{(1+\gamma)^2}}\eps_1^2\vr}{U}\lesssim \zeta_0^2b\int \frac{|\nabla\eps_1|^2\vr}{U}.
    \end{equation*}
    Then, by \eqref{Lemma_est of M: L^2 est}, \eqref{Prop_main coer: Poincare for eps_in} and Cauchy's inequality, we obtain
    \begin{equation}\label{Prop_main coer: inner coer part 3}
        \da{\int [\bar\chi_\nu\sqrt{\vr},\,\Ls+\nabla U\cdot\nabla\Psi_{\cdot}]\eps_1\Ms(\sqrt{\vr}\eps_1)} \lesssim \zeta_0^2 \int \frac{|\nabla\eps_1|^2\vr}{U}.
    \end{equation}
    Finally, combining \eqref{Prop_main coer: inner coer part 1}, \eqref{Prop_main coer: inner coer part 2}, \eqref{Prop_main coer: inner coer part 3} and Poincar\'e inequality \eqref{Prop_main coer: Poincare for eps_in}, it holds, when $\zeta_0$ and $b$ are sufficiently small, that
    \begin{equation}\label{Prop_main coer: inner coer}
          \scl{\Lst \eps_1}{\eps_1}_* \leq -\delta\dr{\int \frac{|\nabla\eps_1|^2\vr}{U}+b\int\frac{\eps^2_1\vr}{U}}.
    \end{equation}
    for some universal $\delta>0$. 
    \\[4pt]
    \noindent\underline{Coercivity of $\scl{\Lst \eps_2}{\eps_2}_*$:}
    First, integrating by parts, we have
    \begin{align*}
       \scl{\Lst \eps_2}{\eps_2}_* = -\int |\nabla\eps_2|^2\bar\chi_\nu^2\om + \int \eps_2^2\nabla\cdot(\om\bar\chi_\nu\nabla\bar\chi_\nu)+\int 2(U-b) \eps^2_2 \bar\chi^2_\nu \om \\
       - \int\nabla U\cdot\nabla\Psi_{\bar\chi_\nu\sqrt{\vr}\eps_2}\eps_2\bar\chi^2_\nu\om-\int\sqrt{\vr}\bar\chi_\nu\Lst\eps_2\Psi_{\sqrt{\vr}\bar\chi_\nu\eps_2}.
    \end{align*}
    As for the second term above, observe that $\nabla\bar\chi_\nu$ is supported in $\{|\log\nu|/\nu\leq\gamma\leq 2|\log\nu|/\nu\}$ and
    \[
        e^{-\frac{b\gamma^2}{2}}\leq e^{-\frac{\beta|\log\nu|^2}{2}}\lesssim \nu^N,\quad \forall\; \frac{|\log\nu|}{\nu}\leq\gamma\leq \frac{2|\log\nu|}{\nu}
    \]
    for any fixed $N\gg 1$ when $\nu$ is sufficiently small. Thus, we have the estimate
    \begin{equation}\label{Prop_main coer: outer coer part 1, (a)}
        \da{\int\eps_2^2\nabla\cdot(\om\bar\chi_\nu\nabla\bar\chi_\nu)}\leq \nu^{100}\|\eps\|^2_{L^\infty(\gamma\geq |\log\nu|/\nu)}.
    \end{equation}
    Also, since $2(U-b)\leq -b$ when $\gamma > \zeta_0/\nu$ and $\nu$ is small enough, we have
    \begin{equation}\label{Prop_main coer: outer coer part 1, (b)}
        \int 2(U-b) \eps^2_2 \bar\chi^2_\nu \om\leq -b\int \eps_2^2\bar\chi^2_\nu\om.
    \end{equation}
    For the fourth term, we use the $2$D Hardy–Littlewood–Sobolev (HLS) inequality \eqref{appendix: HLS ineq}:
    \[
        \|\nabla\Psi_{\sqrt{\vr}\bar\chi_\nu\eps_2}\|_{L^4}\lesssim  \dr{\int \frac{\eps_2^2\bar\chi^2_\nu\vr}{U}}^{\frac{1}{2}}=\|\eps_2\bar\chi_\nu\|_{L^2_\om},
    \]
    and estimate by Cauchy's inequality that
    \begin{equation}\label{Prop_main coer: outer coer part 2}
        \da{\int\nabla U\cdot\nabla\Psi_{\bar\chi_\nu\sqrt{\vr}\eps_2}\eps_2\bar\chi^2_\nu\om}\leq \dr{\int \bar\chi^2_\nu\eps_2^2\om}^{\frac{1}{2}}\dr{\int_{\{\gamma\geq \zeta_0/\nu\}}|\nabla U|^4\om^2}^{\frac{1}{4}}\|\nabla\Psi_{\sqrt{\vr}\bar\chi_\nu\eps_2}\|_{L^4}\lesssim b^{\frac{5}{4}}\int \eps_2^2\bar\chi^2_\nu\om.
    \end{equation}
    For the remaining terms, we divide them into three groups:
    \begin{align*}
        \int \sqrt{\vr}\bar\chi_\nu\Lst\eps_2\Psi_{\sqrt{\vr}\bar\chi_\nu\eps_2} = \int \sqrt{\vr}\bar\chi_\nu\Psi_{\sqrt{\vr}\bar\chi_\nu\eps_2}\Delta\eps_2+\int \sqrt{\vr}\bar\chi_\nu\Psi_{\sqrt{\vr}\bar\chi_\nu\eps_2}(-\nabla\eps_2\cdot\nabla\Psi_U-b\nabla\cdot(\y\eps_2)+2\eps_2U)\\
        -\int \sqrt{\vr}\bar\chi_\nu\Psi_{\sqrt{\vr}\bar\chi_\nu\eps_2}\nabla\Psi_{\sqrt{\vr}\bar\chi_\nu\eps_2}\cdot\nabla U.
    \end{align*}
    As before, by \eqref{appendix: pointwise est of Poisson field by weighted L^2 norm}, we have the pointwise estimate of the Poisson field
    \[
        |\Psi_{\sqrt{\vr}\bar\chi_\nu\eps_2}(\y)|^2\lesssim \log(4+|\y|)\int \eps^2_2\bar\chi^2_\nu\om.
    \]
    Integrating by parts, we have
    \[
       \int \sqrt{\vr}\bar\chi_\nu\Psi_{\sqrt{\vr}\bar\chi_\nu\eps_2}\Delta\eps_2 = \int \Delta(\sqrt{\vr}\bar\chi_\nu\Psi_{\sqrt{\vr}\bar\chi_\nu\eps_2})\eps_2.
    \]
    By chain rule, when the derivative hits $\bar\chi_\nu$, we use the $L^\infty$-control of $\eps_2$ as before, and when the derivative hits elsewhere we use either the pointwise estimate of $\Psi_{\sqrt{\vr}\bar\chi_\nu\eps_2}$ or the $L^4$-estimate of $\nabla\Psi_{\sqrt{\vr}\bar\chi_\nu\eps_2}$:
    \begin{align*}
        &\da{\int(2\nabla\bar\chi_\nu\cdot\nabla(\sqrt{\vr}\Psi_{\sqrt{\vr}\bar\chi_\nu\eps_2})+\Delta\bar\chi_\nu\sqrt{\vr}\Psi_{\sqrt{\vr}\bar\chi_\nu\eps_2})\eps_2}\lesssim \nu^{100}\|\eps_2\bar\chi_\nu\|_{L^2_\om}\cdot\|\eps_2\|_{L^\infty(\gamma\geq |\log\nu|/\nu)},\\
        &\da{\int \nabla\sqrt{\vr}\cdot\nabla\Psi_{\sqrt{\vr}\bar\chi_\nu\eps_2}\eps_2}\lesssim b\|\gamma\sqrt{U}\|_{L^4(\gamma\geq \zeta_0/\nu)}\|\nabla\Psi_{\sqrt{\vr}\bar\chi_\nu\eps_2}\|_{L^4}\|\eps_2\bar\chi_\nu\|_{L^2_\om}\lesssim b^{\frac{5}{4}}\|\eps_2\bar\chi_\nu\|^2_{L^2_\om},\\
        &\da{\int \Delta(\sqrt{\vr})\Psi_{\sqrt{\vr}\bar\chi_\nu\eps_2}\eps_2}\lesssim \|(b+b^2\gamma^2)\log(4+\gamma)\sqrt{U}\|_{L^2(\gamma\geq \zeta_0/\nu)}\|\eps_2\bar\chi_\nu\|^2_{L^2_\om}\lesssim b^\frac{5}{4}\|\eps_2\bar\chi_\nu\|^2_{L^2_\om},\\
        &\da{\int \bar\chi_\nu^2\vr\eps_2^2} \lesssim \frac{b^2}{\zeta_0^4}\|\eps_2\bar\chi_\nu\|^2_{L^2_\om}.
    \end{align*}
    In summary, we obtain
    \begin{equation}\label{Prop_main coer: outer coer part 3, (a)}
        \da{\int \Delta(\sqrt{\vr}\bar\chi_\nu\Psi_{\sqrt{\vr}\bar\chi_\nu\eps_2})\eps_2}\lesssim \nu^{100}\|\eps_2\bar\chi_\nu\|_{L^2_\om}\cdot\|\eps_2\|_{L^\infty(\gamma\geq |\log\nu|/\nu)}+b^{\frac{5}{4}}\|\eps_2\bar\chi_\nu\|^2_{L^2_\om}.
    \end{equation}
    The estimate of the second group is more direct:
    \begin{align}\label{Prop_main coer: outer coer part 3, (b)}
        &\quad\da{\int \sqrt{\vr}\bar\chi_\nu\Psi_{\sqrt{\vr}\bar\chi_\nu\eps_2}(-\nabla\eps_2\cdot\nabla\Psi_U-b\nabla\cdot(\y\eps_2)+2\eps_2U)}\nonumber\\
        &\lesssim \|\eps\|^2_{L^2_\om}\dr{\int_{\{\zeta_0/\nu\leq\gamma\leq |\log\nu|/\nu\}} (U^2+b^2)U\log^2(4+\gamma)}^{\frac{1}{2}}\nonumber\\
        &\quad+\|\eps_2\bar\chi_\nu\|_{L^2_\om}\|\bar\chi_\nu\nabla\eps_2\|_{L^2_\om}\dr{\int_{\{\zeta_0/\nu\leq\gamma\leq |\log\nu|/\nu\}}(b^2\gamma^2+|\nabla\Psi_U|^2)U\log^2(4+\gamma)}^{\frac{1}{2}}\nonumber\\
        &\lesssim b^{\frac{5}{4}}\|\eps_2\bar\chi_\nu\|^2_{L^2_\om}+b^{\frac{1}{4}}\|\bar\chi_\nu\nabla\eps_2\|^2_{L^2_\om}.
    \end{align}
    For the last term, Cauchy's inequality yields
    \begin{align}\label{Prop_main coer: outer coer part 3, (c)}
        \da{\int \sqrt{\vr}\bar\chi_\nu\Psi_{\sqrt{\vr}\bar\chi_\nu\eps_2}\nabla\Psi_{\sqrt{\vr}\bar\chi_\nu\eps_2}\cdot\nabla U}&\lesssim \|\eps_2\bar\chi_\nu\|_{L^2_\om}\|\Psi_{\sqrt{\vr}\bar\chi_\nu\eps_2}\|_{L^4}\|\log(4+\gamma)\nabla U\|_{L^\frac{4}{3}(\gamma\geq \zeta_0/\nu)}\nonumber\\
        &\lesssim b^{\frac{3}{2}}\|\eps_2\bar\chi_\nu\|^2_{L^2_\om}.
    \end{align}
    Combining \eqref{Prop_main coer: outer coer part 3, (a)}, \eqref{Prop_main coer: outer coer part 3, (b)} and \eqref{Prop_main coer: outer coer part 3, (c)}, we obtain
    \begin{equation}\label{Prop_main coer: outer coer part 3}
        \da{\int \sqrt{\vr}\bar\chi_\nu\Lst\eps_2\Psi_{\sqrt{\vr}\bar\chi_\nu\eps_2}}\lesssim b^{\frac{5}{4}}\|\eps_2\bar\chi_\nu\|^2_{L^2_\om} + b^{\frac{1}{4}}\|\bar\chi_\nu\nabla\eps_2\|^2_{L^2_\om}+\nu^{100}\|\eps_2\|^2_{L^\infty(\gamma\geq |\log\nu|/\nu)}.
    \end{equation}
    Finally, combining \eqref{Prop_main coer: outer coer part 1, (a)}, \eqref{Prop_main coer: outer coer part 1, (b)}, \eqref{Prop_main coer: outer coer part 2} and \eqref{Prop_main coer: outer coer part 3}, we have
    \begin{equation}\label{Prop_main coer: outer coer}
        \scl{\Lst\eps_2}{\eps_2}_*\leq -\delta\dr{\|\bar\chi_\nu\nabla\eps_2\|^2_{L^2_\om}+b\|\bar\chi_\nu\eps_2\|^2_{L^2_\om}}+C\nu^{100}\|\eps_2\|^2_{L^\infty(\gamma\geq |\log\nu|/\nu)},
    \end{equation}
    for some universal $\delta,C>0$, when $\zeta_0$ is sufficiently small.
                                           \\[4pt]
    \noindent\underline{Estimates of $\scl{\Lst \eps_1}{\eps_2}_*+\scl{\Lst \eps_2}{\eps_1}_*$:} The methods to estimate these interaction terms are the same as the previous ones. We remark that the interaction happens only in a relatively small region as $\zeta_0$ is meant to be small. We will later exploit this smallness to control the interaction terms. Through integration by parts,
    \begin{align*}
        &\scl{\Lst \eps_1}{\eps_2}_*+\scl{\Lst \eps_2}{\eps_1}_* = -2\int \om\nabla\eps_1\cdot\nabla\eps_2+4\int (U-b)\eps_1\eps_2\omega \\
        &\quad-\int \nabla U\cdot\nabla\Psi_{\sqrt{\vr}\eps_1}\eps_2\om -\int \sqrt{\vr}\bar\chi_\nu\Lst\eps_1 \Psi_{\sqrt{\vr}\bar\chi_\nu\eps_2}-\int \nabla U\cdot\nabla\Psi_{\sqrt{\vr}\eps_2}\eps_1\om -\int \sqrt{\vr}\bar\chi_\nu\Lst\eps_2 \Psi_{\sqrt{\vr}\bar\chi_\nu\eps_1}.
    \end{align*}
    Note that the terms on the second line above are all lower-order terms, as we estimated in the previous steps. For the second term on the right-hand side,
    \[
        4\int (U-b)\eps_1\eps_2 \leq -2b\int \bar \chi_0(1-\bar \chi_0)\eps^2 \leq 0,
    \]
    which has the desired sign. Thus, it remains to estimate the first term. Since $\nabla\eps_1 = \bar \chi_0\nabla\eps + \frac{\nu}{\zeta_0}(\nabla\chi)(\gamma\nu/\zeta_0)\eps$ and $\nabla\eps_2 = (1-\bar \chi_0)\nabla\eps - \frac{\nu}{\zeta_0}(\nabla\chi)(\gamma\nu/\zeta_0)\eps$, we have
    \begin{align*}
        \da{\int \om\nabla\eps_1\cdot\nabla\eps_2} = \da{\int_{\{\zeta_0/\nu\leq\gamma\leq 2\zeta_0/\nu\}} \om\nabla\eps_1\cdot\nabla\eps_2}
        \lesssim \int_{\{\zeta_0/\nu\leq\gamma\leq 2\zeta_0/\nu\}}|\nabla\eps|^2\om+\int_{\{\zeta_0/\nu\leq\gamma\leq 2\zeta_0/\nu\}}\frac{\eps^2}{(1+\gamma)^2}\om.
    \end{align*}
    In summary, we have the estimate
    \begin{equation}\label{Prop_main coer: intersection est}
        \da{\scl{\Lst \eps_1}{\eps_2}_*+\scl{\Lst \eps_2}{\eps_1}_*}\lesssim\int_{\{\zeta_0/\nu\leq\gamma\leq 2\zeta_0/\nu\}}|\nabla\eps|^2\om+\int_{\{\zeta_0/\nu\leq\gamma\leq 2\zeta_0/\nu\}}\frac{\eps^2}{(1+\gamma)^2}\om+\text{l.o.t.},
    \end{equation}
    where the lower order terms (l.o.t.)  can be absorbed into other terms when $b$ is small enough.
    \\[4pt]
    \underline{Global coercivity:} 
    Now we are ready to derive the full coercivity based on the established estimates. First, by \eqref{Prop_main coer: inner coer}, \eqref{Prop_main coer: outer coer} and \eqref{Prop_main coer: intersection est}, when $b$ is sufficiently small, there exist universal constants $\delta,c,C>0$ such that
    \begin{align}\label{Prop_main coer: pre global coer}
    \begin{split}
        \scl{\Lst \eps}{\eps}_* \leq -\delta \int \dr{b\eps^2+|\nabla\eps|^2}\bar\chi^2_\nu\om + C\int_{\{\zeta_0/\nu\leq\gamma\leq 2\zeta_0/\nu\}}\dr{|\nabla\eps^2|+\frac{\eps^2}{(1+\gamma)^2}}\om \\
        + C\nu^{100}\|\eps_2\|^2_{L^\infty(\gamma\geq |\log\nu|/\nu)}.
    \end{split}
    \end{align}        
    By the Hardy-Poincar\'e type inequality \eqref{appendix: one Hardy-Poincare type ineq in soliton}, there exists $C'>0$ such that
    \begin{equation*}
        \int_{\{\gamma\leq \frac{1}{\nu}\}}\frac{\eps^2}{(1+\gamma)^2}\om\leq\int \frac{\eps^2\chi^2(\gamma\nu)}{(1+\gamma)^2}\om\leq C' \int_{\{\gamma\leq \frac{2}{\nu}\}} (b\eps^2+|\nabla\eps|^2)\om.
    \end{equation*}
    We choose an integer $N_0>0$ such that 
    \[
         N_0 > \frac{4CC'}{\delta},
    \]
    and fix $0<\zeta_*\ll 1$ small enough such that $\zeta_*2^{N_0}$ will satisfy all the smallness requirements for $\zeta_0$ in the previous steps. Now we apply the following dyadic argument:
    \begin{align*}
        \frac{\delta}{2}\int\dr{b\eps^2+|\nabla\eps|^2}\bar\chi^2_\nu\om&\geq \frac{\delta}{2}\int_{\{\gamma\leq\frac{2}{\nu}\}}\dr{b\eps^2+|\nabla\eps|^2}\om \geq \frac{\delta}{4C'}\int_{\{\gamma\leq\frac{1}{\nu}\}}\dr{|\nabla\eps|^2+\frac{\eps^2}{(1+\gamma)^2}}\om \\
        &\quad\geq \sum_{j=0}^{N_0-1} \frac{\delta}{4C'}\int_{\{\frac{\zeta_*2^{j}}{\nu}\leq\gamma\leq\frac{\zeta_*2^{j+1}}{\nu}\}}\dr{|\nabla\eps|^2+\frac{\eps^2}{(1+\gamma)^2}}\om \\
        &\quad\geq  \frac{N_0\delta}{4C'}\min_{0\leq j\leq N_0}\left\{\int_{\{\frac{\zeta_*2^{j}}{\nu}\leq\gamma\leq\frac{\zeta_*2^{j+1}}{\nu}\}}\dr{|\nabla\eps|^2+\frac{\eps^2}{(1+\gamma)^2}}\om \right\}.
    \end{align*}
    Let $0\leq n_0\leq N_0$ be the integer such that
    \[
        \int_{\{\frac{\zeta_*2^{n_0}}{\nu}\leq\gamma\leq\frac{\zeta_*2^{n_0+1}}{\nu}\}}\dr{|\nabla\eps|^2+\frac{\eps^2}{(1+\gamma)^2}}\om  = \min_{0\leq j\leq N_0}\left\{\int_{\{\frac{\zeta_*2^{j}}{\nu}\leq\gamma\leq\frac{\zeta_*2^{j+1}}{\nu}\}}\dr{|\nabla\eps|^2+\frac{\eps^2}{(1+\gamma)^2}}\om \right\},
    \]
    and define $\zeta_0 = \zeta_*2^{n_0}$. It follows by the choice of $\zeta_0$ and the definition of $N_0$ that
    \[
        C\int_{\{\zeta_0/\nu\leq\gamma\leq 2\zeta_0/\nu\}}\dr{|\nabla\eps^2|+\frac{\eps^2}{(1+\gamma)^2}}\om < \frac{\delta}{2}\int\dr{b\eps^2+|\nabla\eps|^2}\bar\chi^2_\nu\om.
    \]
    This completes the proof of the proposition. 
\end{proof}
The coercivity result can be stated equivalently in the parabolic variables:
\begin{corollary}\label{main coercivity corollary}
There exist constants $\delta,\zeta_*,C,\nu_*>0$, such that for any $0<\nu<\nu_*$ and any $\eps$ satisfying $\int\frac{\eps^2+|\nabla\eps|^2}{U_\nu}<+\infty$ and the orthogonality conditions
\[
    \int_{\reall^2}\eps\Lambda U_\nu\chi(\zeta/\zeta_*)\sqrt{\vr_\nu} = \int_{\reall^2}\eps\nabla U_\nu \chi(\zeta/\zeta_*)\sqrt{\vr_\nu} = 0,
\]
we have
\begin{equation}\label{main coercivity corollary ineq}
    \scl{\Lst^\zeta_\nu\eps}{\eps}_{\nu,*}\leq -\delta\dr{\int_{\reall^2}\frac{|\nabla\eps|^2\chi^2_\nu}{U_\nu}\vr_\nu+\int_{\reall^2}\frac{\eps^2\chi^2_\nu}{U_\nu}\vr_\nu}+C\nu^{100}\|\eps\|^2_{L^\infty(\zeta\geq 1/\nu)}. 
\end{equation}
\end{corollary}
At the end of this section, we introduce a higher-order coercivity result of the linearized operator $\Ls_0$, which will be used in the $H^1$ energy estimate in the inner region. 
\begin{proposition}[Higher order dissipation structure]\label{proposition: high order coercivity}
    There exists $\delta>0$, such that for any $\eps$ that satisfies $\int\frac{\eps^2+|\nabla\eps|^2+|\nabla^{(2)}\eps|^2}{U}<+\infty$ and the orthogonality conditions
    \[
        \scl{\eps}{\Lambda U} = \scl{\eps}{\pa_\rho U}=\scl{\eps}{\pa_\xi U}=0,
    \]
    it holds that
    \begin{equation}
        \int \frac{|\Ls_0\eps|^2}{U}> \delta\dr{ \int (1+\gamma)^4|\Delta\eps|^2+\int (1+\gamma)^2|\nabla\eps|^2+\int \eps^2+\int\frac{|\nabla\Psi_\eps|^2}{(1+\gamma)^4}}.
    \end{equation}
\end{proposition}
\begin{proof}
    See \textit{Proposition 2.8} in \cite{Raphael2014}. We remark that although the orthogonality conditions there are different, the proof remains valid as long as $\eps$ lies in some subspace whose intersection with $\Span\{\Lambda U,\pa_\rho U,\pa_\xi U\}$ is $\{0\}$. In particular, it can be applied here.
\end{proof}

\section{Construction of Blowup Solutions}\label{sec: construction of blowup}
In this section, we start constructing the finite-time blowup solution. We first decompose the whole solution into an approximate one plus a perturbation function, where we will introduce modulation parameters driving the evolution of the perturbation. We set up the bootstrap assumptions in Definition \ref{definition: bootstrap}, and then derive modulation equations in Lemma \ref{Lemma: modulation equations}. Finally, we perform a series of energy estimates in the inner region and outer region, respectively, for the perturbation.  
\subsection{Decomposition of the Solution and Formulation of the Linearized Problem}

Consider the following decomposition of the solution:
\begin{equation}\label{decomp of fully approx solution}
w(\bar r, \bar z, \tau)=U_\nu(\zeta) + P(\zeta,\tau)+\eps(\br,\bz,\tau):=W(\zeta,\tau)+\eps(\br,\bz,\tau),
\end{equation}
where we denote
\[
   P(\zeta,\tau):= a(\tau)(\vp_{1,\nu}(\zeta)-\vp_{0,\nu}(\zeta)).
\]
Now we study the evolution of the solution in the near field. Inserting the decomposition $w = U_\nu+P+\eps$ into \eqref{Parabolic System}, we obtain the equation for $\eps$
\begin{equation}\label{evolution of epsilon}
\partial_\tau \varepsilon = \Ls^\zeta_\nu \varepsilon + L(\varepsilon) + NL(\varepsilon) + E,
\end{equation}
where the extra linear terms
\begin{align*}
    L(\eps) &=  -\nabla\cdot\dr{W\nabla\Theta_{\eps}+P\nabla\Psi_{\eps}+\eps\nabla(\Phi_W - \Psi_{U_\nu})}\\
    &\quad+\frac{1}{\br+R/\mu}(\pa_{\br} \eps -\eps\pa_{\br}\Phi_W-W\pa_{\br}\Phi_{\eps})+\frac{R_{\tau}}{\mu}\pa_{\br}\eps,
\end{align*}
the nonlinear terms
\begin{align*}
    NL(\eps)= -\nabla\cdot(\eps\nabla\Phi_{\eps})-\frac{1}{\br+R/\mu}\eps\pa_{\br}\Phi_\eps. 
\end{align*}
and the generated error
\begin{align}\label{generated error}
\begin{split}
    E& = -P_\tau + \Ls^\zeta_\nu P -\nabla\cdot(W\nabla\Theta_W+P\nabla\Psi_P)+\dr{\frac{\nu_\tau}{\nu}-\beta}\Lambda U_\nu\\
    &\quad+\frac{1}{\br+R/\mu}(\pa_{\br} W - W\pa_{\br}\Phi_W) + \frac{R_\tau}{\mu}\pa_{\br} W,
\end{split}
\end{align}
where we can compute that
\begin{align*}
    P_\tau &= a_\tau(\varphi_{1, \nu}(\zeta) - \varphi_{0, \nu}(\zeta))  
    +a(\tau)\frac{\nu_\tau}{\nu}\nu\pa_\nu(\varphi_{1, \nu}(\zeta) - \varphi_{0, \nu}(\zeta)).
\end{align*}
We require $\eps$ to be even in $z$-variable (which is preserved by the evolution) and impose the local orthogonality conditions:
\begin{align}\label{local orthorgonality conditions}
    \int_{\reall^2}\eps\chi_*(\zeta)\;d\br d \bz = \int_{\reall^2}\eps\Lambda U_\nu\chi_*(\zeta)\;d\br d \bz = \int_{\reall^2}\eps\nabla U_\nu\chi_*(\zeta)\;d\br d \bz =0,
\end{align}
which are preserved by the modulation parameters $a(\tau),\nu(\tau), R_\tau/\mu$ and the even symmetry in $z$.
Recall the definition of the inner norm
 \begin{equation}\label{def of inner norm}
     \|f\|_\inn :=\dr{\int_{\reall^2}\frac{\nu^2f^2\chi^2_\nu}{U_\nu}e^{-\frac{\beta\zeta^2}{2}}}^{\frac{1}{2}}.
 \end{equation}

\begin{proposition}[Decomposition of the generated error]\label{prop: decomp of the gene error}
The generated error can be decomposed as
\begin{align}\label{Prop3: decomposition of generated error}
    E = \mod_0\varphi_{0,\nu}+\mod_1\varphi_{1,\nu}+\frac{R_\tau}{\mu}\pa_{\br} U_\nu+\tilde E,
\end{align}
where
\begin{align*}
    \mod_0 &= a_\tau - 2a\beta\dr{1+\frac{1}{2\log(\nu)}}-16\nu^2\dr{\frac{\nu_\tau}{\nu}-\beta},\\
    \mod_1 &= -a_\tau +\frac{a(\tau)\beta}{\log(\nu)}.
\end{align*}
    Then, we have the weighted $L^2$-estimate for the error:
    \begin{align}\label{Prop3: weighted L2 estimate of tilde E}
        \|\tilde E\|_{\inn}\lesssim \frac{\nu^2+|a|}{|\log\nu|} +\da{\frac{\nu_\tau}{\nu}}\frac{|a|}{|\log\nu|}+|a|\sqrt{|\log\nu|}\cdot\left|\frac{R_\tau}{\mu}\right|+\frac{a^2}{\nu}. 
    \end{align}
    In addition, we have the following estimates for the local $L^2$-projections of $\tilde E$ onto $1,\Lambda U_\nu,\nabla U_\nu$:
    \begin{align}\label{Prop3: local L2 estimates of projections}
    \begin{split}
        &\da{\scl{\tilde E}{\chi_*}}\lesssim \frac{\nu^2+|a|}{|\log\nu|}+\da{\frac{\nu_\tau}{\nu}}\frac{|a|}{|\log\nu|}+a^2|\log\nu|,\\
        &\da{\scl{\tilde E}{\chi_*\Lambda U_\nu}}\lesssim 1+\frac{a^2}{\nu^4}+\da{\frac{\nu_\tau}{\nu}\cdot\frac{a}{\nu^2}}+|a|,\\
        &\da{\scl{\tilde E}{\pa_{\br}U_\nu\chi_*}}\lesssim \da{\frac{aR_\tau}{\nu^4\mu}},\quad \da{\scl{\tilde E}{\pa_{\bz}U_\nu\chi_*}} = 0.
    \end{split}
    \end{align}
\end{proposition}

\begin{proof}
Recall, from Proposition \ref{proposition 1}, that
\[
    \Ls^\zeta_\nu\varphi_{i,\nu} = 2\beta\dr{1-i+\frac
    {1}{2\log(\nu)}}\varphi_{i,\nu}+R_i,
\]
and
\[
    \varphi_{i,\nu}(\zeta) = -\frac{1}{16\nu^2}\Lambda U_\nu(\zeta)\chi_*+\tilde\varphi_i(\zeta).
\]
We obtain the decomposition \eqref{Prop3: decomposition of generated error}, with
\begin{align}\label{expression of tilde E}
\begin{split}
    \tilde E &=   -a(\tau)\frac{\nu_\tau}{\nu}\nu\pa_\nu(\varphi_{1, \nu}(\zeta) - \varphi_{0, \nu}(\zeta))-\nabla\cdot(W\nabla\Theta_W+P\nabla\Psi_P)\\
    &\quad+\dr{\frac{\nu_\tau}{\nu}-\beta}(\Lambda U_\nu(\zeta) +16\nu^2 \varphi_{0,\nu}(\zeta))+\frac{1}{\br+R/\mu}(\pa_{\br} W - W\pa_{\br}\Phi_W) 
    + \frac{R_\tau}{\mu}\pa_{\br} P+a(\tau)(R_1 - R_0).
\end{split}
\end{align}
    The proof of the estimates relies on the pointwise estimates derived in Proposition \ref{proposition 1}.\\[3pt]
    \underline{Estimate of $\|\tilde E\|_\inn$:} First, by \eqref{prop1: improved est near the origin},
    \begin{align*}
        \|\nu\pa_\nu(\varphi_{1,\nu}-\varphi_{0,\nu})\|^2_{\inn}&\lesssim \int_{0}^{+\infty}\frac{\nu^4\zeta^5\log^4(2+\zeta/\nu)\log^2(4+\zeta)}{(\nu+\zeta)^{8}}e^{-\frac{\zeta^2}{2}}\;d\zeta+\frac{1}{\log^2(\nu)}\int_{0}^{+\infty}\zeta^5\log^2(4+\zeta)e^{-\frac{\zeta^2}{2}}\;d\zeta \\
        &\lesssim \frac{1}{|\log\nu|^2}.
    \end{align*}
    Second, by \eqref{prop1: pointwise est 1}
    \begin{align*}
        \|\Lambda U_\nu + 16\nu^2\varphi_{0,\nu}\|^2_{\inn}& = \|16\nu^2\tilde\varphi_{0,\nu}\|^2_{L^2_{\omega_\nu}}\\
        &\lesssim \nu^4\int_{0}^{+\infty}\frac{\nu^4\zeta^5\log^4(2+\zeta/\nu)\log^2(4+\zeta)}{(\nu+\zeta)^8}e^{-\frac{\zeta^2}{2}}\;d\zeta+\nu^4\int_{0}^{+\infty}\frac{\zeta^5\log^2(4+\zeta)}{|\log\nu|^2(\nu+\zeta)^4}e^{-\frac{\zeta^2}{2}}\;d\zeta\\
        &\lesssim \frac{\nu^4}{|\log\nu|^2}.
    \end{align*}
    Third, by \eqref{prop1: improved est near the origin},
    \begin{align*}
        \|\pa_{\br}P\|^2_{\inn} = \left\|a(\tau)\frac{\br}{\zeta}\pa_\zeta(\varphi_{1,\nu} - \varphi_{0,\nu})\right\|^2_{\inn}
        \lesssim a^2\int\frac{\zeta^2\log^2(4+\zeta)}{(\nu+\zeta)^4}e^{-\frac{\zeta^2}{2}}\;d\br d\bz
        \lesssim a^2|\log\nu|.
    \end{align*}
    Fourth, by \eqref{prop1: pointwise est for R_i}, we have $|R_i(\zeta)(\nu+\zeta)^2e^{-\frac{\zeta^2}{4}}|\lesssim \frac{1}{|\log\nu|(1+\zeta)^4}$ for any $\zeta\geq 0$, and it follows that
    \[
        \|R_i\|_{\inn}\lesssim \frac{1}{|\log\nu|}.
    \]
     Next, we estimate the term $\nabla\cdot(P\nabla\Psi_P) = \pa_\zeta P\pa_\zeta\Psi-P^2$. By \eqref{prop1: improved est near the origin} we obtain the following pointwise estimates ($k=0,1,2$):
    \begin{equation}\label{prop_est of E: pointwise est of P}
        |\pa^k_\zeta P(\zeta)| \lesssim \frac{|a|\zeta^{2-k}\log(2+\zeta)}{(\nu+\zeta)^4}, 
    \end{equation}
    and for $\zeta=\Oc(1)$,
    \begin{align}\label{prop_est of E: poinwise est of nabla Psi_P}
        |\pa_\zeta\Psi_P(\zeta)| = \da{\frac{a(\tau)}{\zeta}\int_{0}^{\zeta}r(\varphi_{1,\nu}(r)-\varphi_{0,\nu}(r))\;dr}
        \lesssim \frac{|a|}{\zeta}\int_{0}^{\zeta}\frac{r^3}{(\nu+r)^4}\;dr
        \lesssim \frac{|a|}{\zeta}\log(1+\zeta^4/\nu^4).
    \end{align}
    Then, we have
    \begin{align*}
        \|\pa_\zeta P\pa_\zeta\Psi_P\|^2_{\inn}&\lesssim a^4\int_{0}^{+\infty}\frac{\zeta\log^2(4+\zeta)}{(\nu+\zeta)^4}\log^2(1+\zeta/\nu)e^{-\frac{\zeta^2}{2}}\;d\zeta \lesssim \frac{a^4}{\nu^2},
    \end{align*}
    and
    \begin{align*}
        \|P^2\|^2_{\inn}&\lesssim a^4\int_{0}^{+\infty}\frac{\zeta^9\log^4(4+\zeta)}{(\nu+\zeta)^{12}}e^{-\frac{\zeta^2}{2}}\; d\zeta\lesssim \frac{
        a^4
        }{\nu^2}.
    \end{align*}
Finally, by Lemma \ref{appendix: difference of 2d 3d Poisson field}, we know that
\[
    \nabla\Phi_W(\zeta) = \nabla\Psi_W(\zeta) + \Oc(\mu^s),\quad\forall\; \zeta>0, 
\]
for some $s>0$. It follows, in particular, that $\|\nabla\Theta_W\|_{L^\infty(\zeta\leq \zeta_*)} = \Oc(\mu^s/\nu^l)$. Thus, the rest terms are all of lower orders (recall that $\mu = \Oc(\nu^k)$ for any $k>0$ ). This, combined with the estimates above, concludes the local $L^2$-estimate of $\tilde E$. Similarly, based on pointwise estimates \eqref{prop1: pointwise est 1}\eqref{prop1: improved est near the origin}\eqref{prop1: pointwise est for R_i}, we can derive the estimates of the local $L^2$-projections.\\[3pt]
\underline{Estimate of $\scl{\tilde E}{\chi_*}$}: First, through integration by parts,
\begin{align*}
    \da{\scl{\nabla\cdot(P\nabla \Psi_P)}{\chi_*}} =\da{\int P\nabla\Psi_P\cdot\nabla \chi_*} \leq a^2|\log\nu|.
\end{align*}
 Second, by the eigenproblem equation, we note that
\begin{align*}
    &\Ls^\zeta_{\nu} \varphi_{0,\nu} = 2\beta\dr{1+\frac{1}{2\log(\nu)}}\varphi_{0,\nu}+R_0,\\
    \Longleftrightarrow \quad&\tilde\varphi_0=\varphi_{0,\nu}+\frac{1}{16\nu^2}\Lambda U_\nu(\zeta)\chi_\nu(\zeta) = \frac{1}{2\beta+\beta/\log(\nu)}\dr{\Ls^\zeta_\nu\varphi_{0,\nu}-R_0}+\frac{1}{16\nu^2}\Lambda U_\nu(\zeta)\chi_\nu(\zeta),
\end{align*}
where
\begin{align*}
    \Ls^\zeta_\nu\varphi_{0,\nu} &= \nabla\cdot\dr{U_\nu\nabla\dr{\frac{\varphi_{0,\nu}}{U_\nu} - \Psi_{\varphi_{0,\nu}}}}-\beta\Lambda\varphi_{0,\nu}\\
     &= \nabla\cdot\dr{U_\nu\nabla\dr{\frac{\tilde\varphi_{0}}{U_\nu} - \Psi_{\tilde\varphi_{0}}}}-\beta\Lambda\tilde\varphi_{0}+\frac{\beta}{16\nu^2}\Lambda^2 U_\nu,\quad \forall\;\zeta\leq \zeta_*.
\end{align*}
By the divergence structure and pointwise estimate \eqref{prop1: pointwise est 1}, we have
\[
    \da{\int\chi_*\Lambda\tilde\varphi_{0}\;d\br d\bz} = \da{\int\zeta\pa_\zeta\chi_*\tilde\varphi_{0}\;d\br d\bz}\lesssim \frac{1}{|\log\nu|},
\]
and
\begin{align*}
    \da{\int\chi_*\nabla\cdot\dr{U_\nu\nabla\dr{\frac{\tilde\varphi_{0}}{U_\nu} - \Psi_{\tilde\varphi_{0}}}}} = \da{\int\pa_\zeta\chi_*U_\nu\pa_\zeta\dr{\frac{\tilde\varphi_{0}}{U_\nu} - \Psi_{\tilde\varphi_{0}}}}\\
     = \da{\int\pa_\zeta\chi_*\dr{\pa_\zeta\tilde\varphi_{0} - \tilde\varphi_{0}\pa_\zeta\Psi_{U_\nu} - U_\nu\pa_\zeta\Psi_{\tilde\varphi_{0}}}}\lesssim\frac{1}{|\log\nu|},
\end{align*}
since by \eqref{prop1: pointwise est 1},
\[
    \da{\pa_\zeta\Psi_{\tilde\varphi_{0}}(\zeta)}\lesssim\da{\frac{1}{\zeta}\int_{0}^{\zeta}\frac{r^3}{(\nu+r)^4|\log\nu|}+\frac{\nu^2r^3\log(1+r/\nu)}{(\nu+r)^6}\;dr}\lesssim \frac{\log(1+\zeta/\nu)}{|\log\nu|\zeta}. 
\]
In addition, due to the cancellation $2U_\nu(\zeta)/\nu^2+\Lambda U_\nu(\zeta)/\nu^2 = \Oc(\nu^2)$ for $\zeta\in  [\zeta_*/2,\zeta_*]$, 
\begin{align*}
    \bigg|\int\chi_*\frac{1}{2+1/\log(\nu)}\cdot\frac{1}{16\nu^2}\Lambda^2 U_\nu&+\frac{1}{16\nu^2}\Lambda U_\nu\bigg| 
    = \da{\int\frac{1}{2+1/\log(\nu)}\cdot\frac{\pa_\zeta\chi_*}{16\nu^2}\zeta\Lambda U_\nu+\frac{\pa_\zeta\chi_*}{16\nu^2}\zeta U_\nu}\\
    &\lesssim \da{\int\frac{\pa_\zeta\chi_*}{32\nu^2}\zeta\Lambda U_\nu+\frac{\pa_\zeta\chi^*}{16\nu^2}\zeta U_\nu}+\frac{1}{|\log\nu|}\cdot\da{\int\frac{\pa_\zeta\chi_*}{\nu^2}\zeta\Lambda U_\nu}
    \lesssim \frac{1}{|\log\nu|},
\end{align*}
and by \eqref{prop1: pointwise est for R_i},
\[
    \da{\int\chi_* R_0\;d\br d\bz}\lesssim\frac{1}{|\log\nu|}.
\]
Combining these estimates we obtain
\[
    \da{\scl{\Lambda U_\nu+16\nu^2\varphi_{0,\nu}}{\chi_*}}\lesssim\frac{\nu^2}{|\log\nu|}.
\]
Note that $\nu\pa_\nu(\nu^{-2}V(\zeta/\nu)) = -\nu^{-2}\Lambda V(\zeta/\nu)$, where $V = V_2$ or $V=\tilde V_2$ in the construction of $\vp_{i,\nu}$ in Proposition \ref{proposition 1}. Exploiting this divergence structure, we have the estimate
\begin{equation*}
    \da{\scl{a\frac{\nu_\tau}{\nu}\nu\pa_\nu(\varphi_{1,\nu}-\varphi_{0,\nu})}{\chi_*}}\lesssim \da{a\frac{\nu_\tau}{\nu}}\dr{\nu^2+\int_{0}^{2\zeta_*}\frac{\zeta^3}{|\log\nu|(\nu+\zeta)^2}\;d\zeta}\lesssim \da{\frac{\nu_\tau}{\nu}}\frac{|a|}{|\log\nu|}.
\end{equation*}
The estimates of the rest terms are more straightforward
\begin{align*}
    &\da{\scl{aR_i}{\chi_*}}\lesssim |a|\int_{0}^{2\zeta_*}\frac{\zeta^3}{(\nu+\zeta)^2}\frac{\log(\zeta)}{|\log(\nu)|}+\frac{\nu\zeta^3}{(\nu+\zeta)^3}\;d\zeta\lesssim \frac{|a|}{|\log\nu|},\\
    &\da{\scl{\pa_{\br}P}{\chi_*}} = 0,
\end{align*}
\underline{Estimate of the rest terms:} The methods are similar, and we briefly summarize them below.\\
For $\scl{\tilde E}{\Lambda U_\nu\chi_*}$:
\begin{align*}
    &\da{\scl{P^2}{\Lambda U_\nu\chi_*}}\lesssim a^2\int_{0}^{2\zeta_*}\frac{\nu^2\zeta^5}{(\nu+\zeta)^{12}}\;d\zeta\lesssim \frac{a^2}{\nu^4},\\
    &\da{\scl{\Lambda U_\nu+16\nu^2\varphi_{0,\nu}}{\Lambda U_\nu\chi_*}}\lesssim \frac{1}{|\log\nu|}\int_{0}^{2\zeta_*}\frac{\nu^4\zeta^3}{(\nu+\zeta)^8}\;d\zeta+\nu^6\int_{0}^{\zeta_*}\frac{\zeta^3\log(1+\zeta/\nu)}{(\nu+\zeta)^{10}}\;d\zeta\lesssim 1,\\
    &\da{\scl{a\frac{\nu_\tau}{\nu}\nu\pa_\nu(\varphi_{1,\nu}-\varphi_{0,\nu})}{\Lambda U_\nu\chi_*}}\lesssim \da{a\frac{\nu_\tau}{\nu}}\int_{0}^{2\zeta_*}\frac{\nu^4\zeta^3}{(\nu+\zeta)^{10}}+\frac{\nu^2\zeta^3}{|\log\nu|(\nu+\zeta)^6}\;d\zeta\lesssim \da{\frac{\nu_\tau a}{\nu^3}},\\
    &\da{\scl{\pa_\zeta P\pa_\zeta\Psi_P}{\Lambda U_\nu\chi_*}}\lesssim a^2\int_{0}^{2\zeta_*}\frac{\nu^2\zeta\log(1+\zeta/\nu)}{(\nu+\zeta)^8}\;d\zeta\lesssim \frac{a^2}{\nu^4},\\
    &\da{\scl{aR_i}{\Lambda U_\nu\chi_*}}\lesssim |a|\int_{0}^{2\zeta_*}\frac{\nu^2\zeta^3\log(\zeta)}{|\log\nu|(\nu+\zeta)^6}+\frac{\nu^3\zeta^3}{(\nu+\zeta)^7}\;d\zeta\lesssim |a|,\\
    &\da{\scl{\pa_{\br}P}{\Lambda U_\nu\chi_*}} =0 .
\end{align*}
 For $\scl{\tilde E}{\pa_{\br}U_\nu\chi_*}$:
\begin{align*}
    \da{\scl{\pa_{\br}P}{\pa_{\br}U_\nu\chi_*}}\lesssim a\int\frac{\nu^2\br^2\chi_*}{(\nu+\zeta)^{10}}\;d\br d\bz \lesssim \frac{a}{\nu^4},
\end{align*}
and the other terms are all zero.
\end{proof}

\subsection{Bootstrap Regime and Modulation Equations}
As we will see in the energy estimates, it only suffices to estimate the gradient of $\eps$ in the region $\zeta\lesssim 1$. Thus, we define 
\[
    \eps^*:= \chi^*\eps.
\]
Recall that $\chi^*(\zeta)=\chi(\zeta/\zeta^*)$ and
\[
    \|f\|^2_{L^2(U_\nu)}:=\int\frac{\nu^2f^2}{U_\nu}.
\]
Now, we are ready to set up our bootstrap assumptions:
\begin{definition}[Bootstrap]\label{definition: bootstrap}
We say that a solution $w$ of \eqref{Parabolic System} lies in the bootstrap regime $\BS(\tau_0,\tau_*, \zeta^*, M_0, \{K_i\}_{i=1}^7)$ if it satisfies the following: on time interval $[\tau_0,\tau_*]$, $w$ admits the decomposition \eqref{decomp of fully approx solution} where the perturbation $\eps$ satisfies the local orthogonality conditions \eqref{local orthorgonality conditions}. In addition, the following estimates hold on $[\tau_0,\tau_*]$:\\
(\romannumeral 1) \textup{(Modulation parameters)}
\begin{align*}
    \frac{1}{K_1}e^{-\sqrt{\beta\tau+M_0}}\leq&\nu(\tau)\leq K_1e^{-\sqrt{\beta\tau+M_0}},\\
    |a(\tau)&-8\nu^2(\tau)| \leq \frac{K_2\nu^2}{|\log\nu|},\\
    &\da{\frac{R_\tau}{\mu}} \leq\frac{K_3\nu}{|\log\nu|}
\end{align*}
(\romannumeral 2) \textup{(Remainders)}
\begin{align*}
    &\left\|\eps\right\|_{\inn} \leq K_4\frac{\nu^2}{|\log\nu|},\\
    &\|\nabla\eps^*\|_{L^2(U_\nu)} \leq K_5\frac{\nu^2}{|\log\nu|},\\
    &\|\eps\|_{H^2(\frac{1}{2}\zeta^*\leq\zeta\leq 4\zeta^*)}\leq K_6\frac{\nu^2}{|\log\nu|},\\
    &\|\eps(1+\zeta)^\frac{3}{2}\|_{L^{\infty}(\zeta\geq\zeta^*)}\leq \frac{K_7}{\sqrt{\beta\tau+M_0}}e^{-2\sqrt{\beta\tau+M_0}}.
\end{align*}
\end{definition}
Note that by the bootstrap assumptions on $\nu$ and $ \|\eps(1+\zeta)^\frac{3}{2}\|_{L^{\infty}(\zeta\geq\zeta^*)}$, we have
\[
    \|\eps(1+\zeta)^\frac{3}{2}\|_{L^{\infty}(\zeta\geq\zeta^*)}\leq \frac{K_7C(K_1)\nu^2}{|\log\nu|}.
\]
The reason why we make such assumption on $ \|\eps(1+\zeta)^\frac{3}{2}\|_{L^{\infty}(\zeta\geq\zeta^*)}$ is a technical treatment to avoid the oscillatory behavior of $\nu$ in time when doing integration, the details of which can be found in Lemma \ref{lemma: L infty control of eps}.  

\begin{lemma}[Modulation equations]\label{Lemma: modulation equations}
Assume that the solution is in the bootstrap regime \\$\BS(\tau_0,\tau_*, \zeta^*, M_0, \{K_i\}_{i=1}^7)$ defined in Definition \ref{definition: bootstrap}. Then, the following estimates hold for any $\tau\in[\tau_0,\tau_*]$:
\begin{align}\label{Lemma modulation eq: inequalities}
\begin{split}
    &|\mod_0|=\da{a_\tau - 2a\beta\dr{1+\frac{1}{2\log(\nu)}}-16\nu^2\dr{\frac{\nu_\tau}{\nu}-\beta}} \leq C\dr{\|\eps\|_\inn+\|\nabla\eps^*\|_{L^2(U_\nu)}}+\frac{C(K_i)\nu^2}{|\log\nu|^2},\\
    &|\mod_1|=\da{-a_\tau +\frac{a(\tau)\beta}{\log(\nu)}} \leq \frac{C(K_1,K_2,K_4,K_5)\nu^2}{|\log\nu|^2}+\frac{C(K_i)\nu^2}{|\log\nu|^3},\\
    &\da{\frac{R_\tau}{\mu}} \leq \frac{C(K_4,K_5)\nu}{|\log\nu|}+\frac{C(K_i)\nu}{|\log\nu|^2}.
\end{split}
\end{align}
\end{lemma}
\begin{proof}
    The strategy of the proof is the following: Since the evolution of the modulation parameters is determined by the preservation of the (local) orthogonality conditions \eqref{local orthorgonality conditions}, we take time derivatives of the orthogonality equations and use energy bounds for $\eps$ to obtain the desired estimates.\\[3pt]
    \underline{Estimate of $\mod_1$ by projection to $\chi_*$:}
   By the orthogonality condition \eqref{local orthorgonality conditions}, we obtain 
    \begin{equation}\label{lemma1: 1}
        0 = \frac{d}{d\tau}\scl{\eps}{\chi_*} = \scl{\pa_\tau\eps}{\chi_*} = \scl{\Ls^\zeta_\nu\eps+L(\eps)+NL(\eps)+E}{\chi_*}.
    \end{equation}
    Recall that (as $\scl{\pa_{\br}U_\nu}{\chi_*}=0$)
    \[
        \scl{E}{\chi_*} = \mod_0\scl{\varphi_{0,\nu}}{\chi_*}+\mod_1\scl{\varphi_{1,\nu}}{\chi_*}+\scl{\tilde E}{\chi_*}.
    \]
    Then, by \eqref{inner eigenfunction} \eqref{Prop3: local L2 estimates of projections} and the fact that $\scl{\frac{1}{\nu^2}\Lambda U_\nu}{\chi_*} = \Oc(1)$, we have
    \begin{equation}\label{lemma1: 2}
        |\scl{\varphi_{0,\nu}}{\chi_*}|\lesssim 1,\quad |\scl{\varphi_{1,\nu}}{\chi_*}|\gtrsim |\log\nu|,\quad \da{\scl{\tilde E}{\chi_*}}\lesssim \frac{\nu^2+|a|}{|\log\nu|}+\da{a\frac{\nu_\tau}{\nu}}+C(K_i)\nu^3.
    \end{equation}
    We remark that the gain of $|\log\nu|$ in $\scl{\vp_{1,\nu}}{\chi_*}$ will enable us to control $\mod_1$ by $\mod_0$ and gain a $|\log\nu|$ smallness in the estimate of $\mod_1$. Next, we estimate the terms containing $\eps$. Since the operator $\Ms^\zeta_\nu$ is self-adjoint in $(L^2(\reall^2),\scl{\cdot}{\cdot})$, which follows from the self-adjointness of $(-\Delta)^{-1}$, we have
    \begin{align*}
        |\scl{\Ls^\zeta_\nu\eps}{\chi_*}| &= \da{\scl{U_\nu\nabla\Ms^\zeta_\nu(\eps)-\beta y\eps}{\nabla\chi_*}}\\
        &\leq\da{\scl{\Ms^\zeta_\nu(\eps)}{\nabla\cdot (U_\nu\nabla\chi_*)}}+\da{\scl{\beta y\eps}{\nabla\chi_*}}.
    \end{align*}
    Since $\nabla\cdot (U_\nu\nabla\chi_*)$ is compactly supported in $[\zeta_*,2\zeta_*]$ and is of size $\Oc(\nu^2)$, i.e., $|\nabla\cdot (U_\nu\nabla\chi_*)|\lesssim \nu^2\one_{\{\zeta_*\leq\zeta\leq2\zeta_*\}}$, by \eqref{Lemma_est of M: L^2 est} we have the estimate:
    \[
        \da{\scl{\Ms^\zeta_\nu(\eps)}{\nabla\cdot (U_\nu\nabla\chi_*)}}\lesssim \int_{\{\zeta_*\leq\zeta\leq 2\zeta_*\}} |\eps| + \da{\scl{\nabla\Psi_\eps}{U_\nu\nabla\chi_*}}\lesssim \|\eps\|_\inn+\nu^2\|\eps(1+\zeta)^\frac{3}{2}\|_{L^\infty(\zeta\geq \zeta^*)},
    \]
    where we use the pointwise estimate of the Poisson field by \eqref{appendix: pointwise est of Poisson field radial} and \eqref{appendix: pointwise est of Poisson field nonradial}:
    \begin{equation}\label{Lemma_mod eq: L^infty est of nabla Psi_eps in the middle range}
        \|\nabla\Psi_\eps\|_{L^\infty(\zeta_*\leq \zeta\leq 2\zeta_*)} \lesssim \dr{\int \eps^2(1+\zeta)^\frac{1}{3}}^\frac{1}{2}\lesssim \frac{1}{\nu^2}\|\eps\|_\inn+\|\eps(1+\zeta)^\frac{3}{2}\|_{L^\infty(\zeta\geq \zeta^*)}. 
    \end{equation}
    It follows that
    \begin{equation}\label{lemma1: 3}
        |\scl{\Ls^\zeta_\nu\eps}{\chi_*}| \lesssim \|\eps\|_{\inn} + \nu^2\|\eps(1+\zeta)^\frac{3}{2}\|_{L^\infty(\zeta\geq \zeta_*)}.
    \end{equation}
    Due to Lemma \ref{appendix: difference of 2d 3d Poisson field}, we neglect the terms of order $\Oc(\mu^s)$ and estimate the rest terms in $L(\eps)$. By \eqref{Lemma_mod eq: L^infty est of nabla Psi_eps in the middle range} and the pointwise estimates \eqref{prop_est of E: pointwise est of P} and \eqref{prop_est of E: poinwise est of nabla Psi_P}, we have 
    \begin{align}\label{lemma1: 4}
    \begin{split}
        \quad\da{\scl{-\nabla\cdot(P\nabla\Psi_\eps+\eps\nabla(\Psi_W - \Psi_{U_\nu}))+\frac{R_\tau}{\mu}\pa_{\br}\eps}{\chi_*}} &=\da{\scl{P\nabla\Psi_\eps+\eps\nabla\Psi_P+\frac{R_\tau}{\mu}e_1\eps}{\nabla\chi_*}}  \\
        &\lesssim \dr{\frac{|a|}{\nu^2}+\da{\frac{R_\tau}{\mu}}}\|\eps\|_{\inn} + |a|\cdot\|\eps(1+\zeta)^\frac{3}{2}\|_{L^\infty(\zeta\geq \zeta_*)}. 
    \end{split}
    \end{align}
    Finally, for the nonlinear terms, neglect the part of order $\Oc(\mu)$. By Cauchy's inequality and \eqref{Lemma_mod eq: L^infty est of nabla Psi_eps in the middle range}, we have:
    \begin{align}\label{lemma1: 5}
        \da{\scl{-\nabla\cdot(\eps\nabla\Psi_\eps)}{\chi_*}} = \da{\scl{\eps\nabla\Psi_\eps}{\nabla\chi_*}}
        &\leq \|\eps\|_{\inn}\|\nabla\Psi_\eps\|_{L^\infty(\zeta_*\leq \zeta\leq 2\zeta_*)}\nonumber\\
        &\leq \dr{\frac{1}{\nu^2}\|\eps\|_{\inn}+\|\eps(1+\zeta)^\frac{3}{2}\|_{L^\infty(\zeta\geq\zeta^*)}}\|\eps\|_{\inn}.
    \end{align}
    By \eqref{lemma1: 1} and collecting all the estimates \eqref{lemma1: 2}\eqref{lemma1: 3}\eqref{lemma1: 4}\eqref{lemma1: 5}, we obtain
    \begin{align}\label{lemma1: mod_1 equation}
        |\mod_1| &\lesssim\frac{1}{|\log\nu|}\bigg(|\mod_0|+\dr{1+\frac{|a|}{\nu^2}+\da{\frac{R_\tau}{\mu}}}\|\eps\|_\inn+(\nu^2+|a|)\|\eps(1+\zeta)^\frac{3}{2}\|_{L^\infty(\zeta\geq\zeta^*)}\nonumber\\
        &\quad+\frac{1}{\nu^2}\|\eps\|^2_\inn+\frac{1}{\nu^2}\|\eps\|_\inn\|\eps(1+\zeta)^\frac{3}{2}\|_{L^\infty(\zeta\geq\zeta^*)}+\frac{\nu^2+|a|}{|\log\nu|}+\da{a\frac{\nu_\tau}{\nu}}+C(K_i)\nu^3\bigg)\nonumber\\
        &\lesssim \frac{|\mod_0|}{|\log\nu|}+\frac{\nu^2C(K_1,K_2,K_4)}{|\log\nu|^2}+\frac{C(K_i)\nu^2}{|\log\nu|^3}.
    \end{align}
    \underline{Estimate of $\mod_0$ by projection to $\Lambda U_\nu\chi_*$:}
    Similarly, we compute
    \begin{align}\label{lemma1: 6}
    0 &= \frac{d}{d_\tau}\scl{\eps}{\Lambda U_\nu\chi_*} = \scl{\pa_\tau\eps}{\Lambda U_\nu\chi_*}+\scl{\eps}{\pa_\tau\Lambda U_\nu\chi_*} \nonumber\\&=\scl{\Ls^\zeta_\nu\eps+L(\eps)+NL(\eps)+E}{\Lambda U_\nu\chi_*} +\scl{\eps}{\frac{\nu_\tau}{\nu}\nu\pa_\nu\Lambda U_\nu\chi_*}. 
    \end{align}
    Note, by \eqref{inner eigenfunction} and \eqref{Prop3: local L2 estimates of projections}, that
    \begin{equation}
        \da{\scl{\varphi_{i,\nu}}{\Lambda U_\nu\chi_*}} \gtrsim \frac{1}{\nu^2}\scl{\Lambda U_\nu}{\Lambda U_\nu\chi_*}\gtrsim \frac{1}{\nu^4},\quad \da{\scl{\tilde E}{\Lambda U_\nu\chi_*}}\lesssim 1+\frac{a^2}{\nu^4}+\frac{C(K_i)}{|\log\nu|}.
    \end{equation}
    Through integration by parts and Cauchy's inequality, 
    \begin{align}
        \quad\da{\scl{\Ls^\zeta_{0,\nu}\eps}{\Lambda U_\nu\chi_*}}&=\da{\scl{\nabla\eps-\eps\nabla\Psi_{U_\nu}-U_\nu\nabla\Psi_\eps}{\nabla(\Lambda U_\nu\chi_*)}}\nonumber\\
        &\leq \dr{\|\eps\|_\inn+\|\nabla\eps^*\|_{L^2(U_\nu)}}\da{\scl{\frac{U_\nu}{\nu^2}}{(\nabla(\Lambda U_\nu\chi_*))^2}}^{\frac{1}{2}}+\da{\scl{\nabla\Psi_\eps U_\nu}{\nabla(\Lambda U_\nu\chi_*)}}\nonumber\\
        &\lesssim \frac{1}{\nu^4} \dr{\|\eps\|_\inn+\|\nabla\eps^*\|_{L^2(U_\nu)}}+\frac{1}{\nu^3}\|\eps(1+\zeta)^\frac{3}{2}\|_{L^\infty(\zeta\geq\zeta^*)},
    \end{align}
    where we apply pointwise estimates of Poisson field \eqref{appendix: pointwise est of Poisson field radial and nonradial, parabolic} \eqref{appendix: pointwise est of the gradient of Poisson field based on L infty of u} to estimate
    \begin{align*}
        \da{\int \frac{\nu^4\nabla\Psi_\eps}{(\nu+\zeta)^9}}&\lesssim \da{\int \frac{\nu^4(1+\zeta/\nu)}{(\nu+\zeta)^{10}}\dr{\int \eps^*(\nu+\zeta)^2}^\frac{1}{2}}+\da{\int\frac{\nu^4}{(\nu+\zeta)^9}}\cdot\|\eps(1+\zeta)^\frac{3}{2}\|_{L^\infty(\zeta\geq\zeta^*)}\\
        &\lesssim \frac{1}{\nu^4}\|\nabla\eps^*\|_{L^2(U_\nu)}+\frac{1}{\nu^3}\|\eps(1+\zeta)^\frac{3}{2}\|_{L^\infty(\zeta\geq\zeta^*)}.
    \end{align*}
    Moreover,
    \begin{equation}
        \da{\scl{\Lambda\eps}{\Lambda U_\nu\chi_*}}\leq \da{\scl{\frac{U_\nu}{\nu^2}}{(\y\cdot\nabla(\Lambda U_\nu\chi_*))^2}}^{\frac{1}{2}}\cdot\|\eps\|_{\inn} \lesssim \frac{1}{\nu^3}\|\eps\|_{\inn}. 
    \end{equation}
    For $L(\eps)$, by \eqref{prop1: improved est near the origin} and similar methods as adapted above:
    \begin{align*}
        &\da{\scl{\nabla P\cdot\nabla \Psi_\eps}{\Lambda U_\nu\chi_*}} \lesssim \frac{|a|}{\nu^4}\|\nabla\eps^*\|_{L^2(U_\nu)}+\frac{|a|}{\nu^3}\|\eps(1+\zeta)^\frac{3}{2}\|_{L^\infty(\zeta\geq\zeta^*)},\\  
        &\da{\scl{P\eps}{\Lambda U_\nu\chi_*}} \lesssim \da{\scl{\frac{U_\nu}{\nu^2}}{(P\Lambda U_\nu\chi_*)^2}}^{\frac{1}{2}}\|\eps\|_{\inn}\leq \frac{|a|}{\nu^5}\|\eps\|_{\inn},\\
        &\da{\scl{\nabla\eps\cdot\nabla\Psi_P}{\Lambda U_\nu\chi_*}} \lesssim \da{\scl{\frac{U_\nu}{\nu^2}}{(\nabla\Psi_P\Lambda U_\nu\chi_*)^2}}^{\frac{1}{2}}\|\nabla\eps^*\|_{\inn}\leq \frac{|a|}{\nu^4}\|\nabla\eps^*\|_{\inn},\\
        &\da{\scl{\frac{R_\tau}{\mu}\pa_{\br}\eps}{\Lambda U_\nu\chi_*}}\lesssim \frac{1}{\nu^3}\da{\frac{R_\tau}{\mu}}\|\nabla\eps^*\|_{\inn}.
    \end{align*}
    As for the nonlinear term,
    \begin{align*}
         \da{\scl{\nabla\eps\cdot\nabla\Psi_\eps}{\Lambda U_\nu\chi_*}} &\lesssim \|\nabla\eps^*\|_{L^2(U_\nu)}\dr{\int \frac{U_\nu}{\nu^2}(\Lambda U_\nu)^2|\nabla\Psi_\eps|^2}^\frac{1}{2}\\
         &\lesssim \frac{1}{\nu^4}\|\nabla\eps^*\|^2_\inn+\frac{1}{\nu^3}\|\nabla\eps^*\|_{L^2(U_\nu)}\|\eps(1+\zeta)^\frac{3}{2}\|_{L^\infty(\zeta\geq\zeta^*)},
    \end{align*}
    and
    \begin{align*}
        \da{\scl{\eps^2}{\Lambda U_\nu\chi_*}}\lesssim \frac{1}{\nu^6}\|\eps\|^2_\inn.
    \end{align*}
    At last, 
    \begin{equation*}
        \da{\scl{\eps}{\frac{\nu_\tau}{\nu}\nu\pa_\nu\Lambda U_\nu\chi_*}}\lesssim \frac{1}{\nu^3}\da{\frac{\nu_\tau}{\nu}}\|\eps\|_{\inn}.
    \end{equation*}
    Inserting all these estimates into \eqref{lemma1: 6}, we obtain
    \begin{equation}\label{lemma1: mod_0 equation}
        |\mod_0|\lesssim |\mod_1|+ \|\eps\|_\inn+\|\nabla\eps^*\|_{L^2(U_\nu)}+\frac{C(K_i)\nu^2}{|\log\nu|^2}.
    \end{equation}
    Combining with \eqref{lemma1: mod_1 equation}, we can further refine
    \begin{align*}
        &|\mod_0|\leq C\dr{\|\eps\|_\inn+\|\nabla\eps^*\|_{L^2(U_\nu)}}+ \frac{C(K_i)\nu^2}{|\log\nu|^2},\\
        &|\mod_1|\leq \frac{C(K_1,K_2,K_4,K_5)\nu^2}{|\log\nu|^2}+\frac{C(K_i)\nu^2}{|\log\nu|^3}.
    \end{align*}
    \underline{Estimate of $\frac{R_\tau}{\mu}$ by projection to $\pa_{\br}U_\nu\chi_*$:} As before, we compute $0=\pa_\tau\scl{\eps}{\pa_{\br}U_\nu}$. Note that
    \begin{equation}
        \scl{E}{\pa_{\br}U_\nu\chi_*} = \frac{R_\tau}{\mu}\scl{\pa_{\br}U_\nu}{\pa_{\br}U_\nu\chi_*}+\scl{\tilde E}{\pa_{\br}U_\nu\chi_*},
    \end{equation}
    where
    \begin{equation}
        \scl{\pa_{\br}U_\nu}{\pa_{\br}U_\nu\chi_*}\gtrsim\frac{1}{\nu^4},\quad \da{\scl{\tilde E}{\pa_{\br}U_\nu\chi_*}}\lesssim \da{\frac{R_\tau}{\nu^2\mu}}.
    \end{equation}
    The estimates of the scalar products with terms containing $\eps$ are similar, with everything amplified by $1/\nu$ compared to the scalar products with $\Lambda U_\nu$ (since $|\pa_{\br}U_\nu|\lesssim \frac{1}{\nu}|\Lambda U_\nu|$ for any $\zeta\leq 2\zeta^*$), and we will not repeat them here. To summarize, we have
    \begin{equation}\label{lemma1: mod_R equation}
        \da{\frac{R_\tau}{\mu}}\leq \frac{C(K_4,K_5)\nu}{|\log\nu|}+\frac{C(K_i)\nu}{|\log\nu|^2}.
    \end{equation}
\end{proof}
A direct consequence of the modulation estimates is the following control on $|\frac{\nu_\tau}{\nu}|$.
\begin{corollary}
    Assume that the solution is in the bootstrap regime given in Definition \ref{definition: bootstrap}. Then, it holds that
    \begin{equation}\label{cor of mod est: est of nu_tau}
        \da{\frac{\nu_\tau}{\nu}}\leq \frac{C(K_2,K_4,K_5)}{|\log\nu|}.
    \end{equation}
\end{corollary}
\begin{proof}
    Inserting $|a-8\nu^2|\leq \frac{K_2\nu^2}{|\log\nu|}$ into the estimate of $|\mod_0+\mod_1|$ given by \eqref{Lemma modulation eq: inequalities}, the result follows.
\end{proof}

\subsection{Energy Estimates}

\subsubsection{$L^2$ Inner Estimate}
\newcommand{\wei}{\chi_\nu\sqrt{\vr_\nu}} 
Now we establish an important $L^2$-monotonicity result for $\eps$. One technical treatment is needed to avoid a loophole in the energy estimates: Due to an incompatibility between the decomposition of generated error and the local orthogonality conditions, the modulation estimates are not small enough to close the $L^2$ energy estimate. However, by projecting out the direction of the first approximate eigenfunction, we are able to get rid of this issue. This is possible thanks to the special structure of 
the adapted inner product as well as the slow decay of the kernel (e.g. $\gamma^2\Lambda U$ does not belong to $L^1$), which makes the aforementioned projection an acceptable modification to the original norm. \\
Note, by Lemma \ref{lemma: properties of the adapted inn product} and orthogonality conditions \eqref{local orthorgonality conditions}, we have
\[
    \int\frac{\eps^2(\chi_\nu)^2\vr_\nu}{U_\nu}-C\nu^2\int\frac{\eps^2(\chi_\nu)^2\vr_\nu}{U_\nu}\lesssim\int \eps\chi_\nu\sqrt{\vr_\nu}\Ms_\nu^\zeta(\eps\chi_\nu\sqrt{\vr_\nu}) \lesssim\int\frac{\eps^2(\chi_\nu)^2\vr_\nu}{U_\nu}.
\]
Thus, $ \scl{\eps}{\eps}_{\nu,*}=\int \eps\chi_\nu\sqrt{\vr_\nu}\Ms_\nu^\zeta(\eps\chi_\nu\sqrt{\vr_\nu})$ is  equivalent to the norm $\frac{1}{\nu^2}\|\eps\|^2_\inn$.
Define
\[
    d_0:= \frac{\int \wei\varphi_{0,\nu}\Ms^\zeta_\nu(\eps\wei)}{\int \wei\varphi_{0,\nu}\Ms^\zeta_\nu(\wei\varphi_{0,\nu})}\lesssim \frac{\|\eps\|_\inn}{|\log\nu|},
\]
as we have
\[
    \da{\int \wei\varphi_{0,\nu}\Ms^\zeta_\nu(\eps\wei)}\lesssim \frac{1}{\nu^2}\|\eps\|_\inn,\quad  \da{\int \wei \varphi_{0,\nu}\Ms^\zeta_\nu(\varphi_{0,\nu}\wei)}\gtrsim \frac{|\log\nu|}{\nu^2}.
\]
We project out the $\varphi_{0,\nu}$ direction of $\eps$, and consider the evolution of 
\begin{align*}
    &\quad\int\eps\wei\Ms^\zeta_\nu(\eps\wei) - d_0\int\eps\wei\Ms^\zeta_\nu(\wei\varphi_{0,\nu}) 
    \\&= \int\eps\wei\Ms^\zeta_\nu(\eps\wei)-\frac{(\int \wei\varphi_{0,\nu}\Ms^\zeta_\nu(\eps\wei))^2}{\int \wei\varphi_{0,\nu}\Ms^\zeta_\nu(\wei\varphi_{0,\nu})}\sim \frac{1}{\nu^2}\|\eps\|^2_\inn.
\end{align*}
Recall that $\varphi_{0,\nu} = -\frac{1}{16\nu^2}\Lambda U_\nu\chi_\nu+\tilde\varphi_{0}$, and 
\[
    \Ms^\zeta_\nu(\wei\varphi_{0,\nu}) = -\frac{1}{8\nu^2}+\Ms^\zeta_\nu(-\frac{1}{16\nu^2}\Lambda U_\nu(\wei\chi_\nu-1))+\Ms^\zeta_\nu(\wei\tilde\varphi_{0}).
\]
We also recall the pointwise estimate
\[
    |\pa_\zeta^k \tilde \varphi_0 (\zeta)| + | \pa_\zeta^k \nu \pa_\nu \tilde \varphi_0(\zeta)| \lesssim \dr{\frac{\nu^2\zeta^{2-k}\log(1+\zeta/\nu)}{(\nu+\zeta)^6}+\frac{\zeta^{2-k} }{ |\log\nu|(\nu + \zeta)^{4}}}(1+\log(\zeta)\one_{\{\zeta>1\}}).
\]
In the following argument, since $\frac{1}{\nu^2}(\wei\chi_\nu-1)$ and $\tilde\varphi_{0}$ are always estimated together, for brevity we denote
\[
    \bar\varphi_0:= -\frac{1}{16\nu^2}\Lambda U_\nu(\wei\chi_\nu-1)+\wei\tilde\varphi_0.
\]
By the pointwise estimates, it is helpful to note that
\begin{equation*}
    \int \frac{|\bar\varphi_0|^2}{U_\nu}\lesssim \frac{1}{\nu^2}, \quad \int\frac{|\nabla\bar\varphi_0|^2}{U_\nu}\lesssim \frac{|\log\nu|}{\nu^2}.
\end{equation*}

\begin{lemma}[Control of $\|\eps\|_\inn$]\label{lemma: L^2 energy estimate}
    Let $w$ be a solution in the bootstrap regime $\BS(\tau_0,\tau_*, \zeta^*, M_0, \{K_i\}_{i=1}^7)$. Then, the following estimate holds on $[\tau_0,\tau_*]$:
    \begin{align}\label{L^2 energy estimate: inequality in statement}
        &\quad\frac{d}{d\tau}\dr{\frac{1}{2}\int\eps\chi_\nu\sqrt{\vr_\nu}\Ms^\zeta_\nu(\eps\chi_\nu\sqrt{\vr_\nu})-\frac{d_0}{2}\int \eps\sqrt{\vr_\nu}\chi_\nu\Ms^\zeta_\nu(\vp_{0,\nu}\sqrt{\vr_\nu}\chi_\nu)}\nonumber\\
        &\leq -\frac{\delta_0}{\nu^2}\dr{\|\eps\|^2_{\inn}+\|\nabla\eps\|^2_{\inn}}+ \frac{C\nu^2}{|\log\nu|^2}+\frac{C(K_i)\nu^2}{|\log\nu|^\frac{7}{3}}
    \end{align}
\end{lemma}
\begin{proof} The first half of the proof estimates the main part (i.e., the leading-order dynamics) \\$\frac{d}{d\tau}\frac{1}{2}\int\eps\chi_\nu\sqrt{\vr_\nu}\Ms^\zeta_\nu(\eps\chi_\nu\sqrt{\vr_\nu})$ which yields damping. Then, the second half deals with the correction term, which projects out the $\mod_0$ direction of the main part. \\
\textbf{Step 1: leading-order dynamics}\\
    First, by \eqref{evolution of epsilon} and the definition of $\scl{\cdot}{\cdot}_{\nu,*}$ in \eqref{def of adapted inner scl, parabolic}, we have
    \begin{align*}
        \frac{1}{2}\frac{d}{d\tau}\scl{\eps}{\eps}_{\nu,*}& =\scl{\pa_\tau\eps}{\eps}_{\nu,*}+\frac{1}{2}\scl{\frac{\pa}{\pa_\tau}\dr{\frac{\vr_\nu\chi^2_\nu}{U_\nu}}}{\eps^2}-\scl{\frac{\pa}{\pa_\tau}\dr{\sqrt{\vr_\nu}\chi_\nu}\eps}{\tilde\Psi_\eps}\\
        &=\scl{\Lst^\zeta_\nu\eps}{\eps}_{\nu,*}+\scl{(\Ls^\zeta_\nu - \Lst^\zeta_\nu)\eps}{\eps}_{\nu,*}+\scl{L(\eps)}{\eps}_{\nu,*}+\scl{NL(\eps)}{\eps}_{\nu,*}+\scl{E}{\eps}_{\nu,*}\\
        &\quad+\frac{1}{2}\scl{\frac{\pa}{\pa_\tau}\dr{\frac{\vr_\nu\chi^2_\nu}{U_\nu}}}{\eps^2}-\scl{\frac{\pa}{\pa_\tau}\dr{\sqrt{\vr_\nu}\chi_\nu}\eps}{\tilde\Psi_\eps}.
    \end{align*}
    \underline{Damping term:} By the coercivity of the modified linearized operator \eqref{main coercivity corollary ineq}, we have the damping
    \begin{equation*}
     \scl{\Lst^\zeta_\nu\eps}{\eps}_{\nu,*}\leq -\delta\dr{\int\frac{\eps^2\chi_\nu^2\vr_\nu}{U_\nu}+\int\frac{|\nabla\eps|^2\chi_\nu^2\vr_\nu}{U_\nu}}+C\nu^{100}\|\eps\|^2_{L^\infty(\zeta\geq 1/\nu)},
    \end{equation*}
    for some universal $\delta,C>0$.\\
    \underline{Estimate of term $(\Ls^\zeta_\nu-\Lst^\zeta_\nu)\eps$:}
    Note that
    \begin{equation}
        (\Ls^\zeta_\nu - \Lst^\zeta_\nu)\eps = \nabla U_\nu\cdot(\nabla\tilde\Psi_\eps - \nabla\Psi_\eps) = \nabla U_\nu\cdot\nabla\Psi_{(\sqrt{\vr_\nu}\chi_\nu-1)\eps},
    \end{equation}
    Integrating by parts, we have
    \begin{align*}
        \scl{\nabla U_\nu\cdot\nabla\Psi_{(\sqrt{\vr_\nu}\chi_\nu-1)\eps}}{\eps}_{\nu,*}&=\int \sqrt{\vr_\nu}\chi_\nu\nabla U_\nu\cdot\nabla\Psi_{(1-\sqrt{\vr_\nu}\chi_\nu)\eps}\Ms^\zeta_\nu(\sqrt{\vr_\nu}\chi_\nu\eps)\\
        &=\int \sqrt{\vr_\nu}\chi_\nu U_\nu(1-\sqrt{\vr_\nu}\chi_\nu)\eps\Ms^\zeta_\nu(\sqrt{\vr_\nu}\chi_\nu\eps)\\
        &\quad-\int U_\nu\nabla(\sqrt{\vr_\nu}\chi_\nu)\cdot\nabla\Psi_{(1-\sqrt{\vr_\nu}\chi_\nu)\eps}\Ms^\zeta_\nu(\sqrt{\vr_\nu}\chi_\nu\eps)\\
        &\quad-\int\sqrt{\vr_\nu}\chi_\nu U_\nu\nabla\Psi_{(1-\sqrt{\vr_\nu}\chi_\nu)\eps}\cdot\nabla\Ms^\zeta_\nu(\sqrt{\vr_\nu}\chi_\nu\eps)\\
        & =: \Rmnum{1}+\Rmnum{2}+\Rmnum{3}.
    \end{align*}
    By Cauchy's inequality and \eqref{Lemma_est of M: L^2 est}, we obtain
    \begin{align*}
        \da{\Rmnum{1}} \lesssim \dr{\int\vr_\nu\chi^2_\nu U_\nu(1-\sqrt{\vr_\nu}\chi_\nu)^2\eps^2}^{\frac{1}{2}}\dr{\int\frac{\eps^2\chi_\nu^2\vr_\nu}{U_\nu}}^{\frac{1}{2}}.
   \end{align*}
   Then, applying $U_\nu(\zeta)(1-\sqrt{\vr_\nu}\chi_\nu)^2\lesssim (\nu+\zeta)^2$, inequality \eqref{appendix: one Hardy-Poincare type ineq in parabolic}, the control of the outer norm, and Cauchy's inequality, we have the estimate
   \begin{equation*}
       |\Rmnum{1}| \leq \frac{\delta}{10}\dr{\int\frac{(\eps^2+|\nabla\eps|^2)\chi_\nu^2\vr_\nu}{U_\nu}}+C\nu^2\int\frac{\eps^2\chi_\nu^2\vr_\nu}{U_\nu}+C\nu^{100}\|\eps\|^2_{L^\infty(\zeta\geq\zeta^*)}.
   \end{equation*}
   Next, by \eqref{appendix: pointwise est of Poisson field radial} and \eqref{appendix: pointwise est of Poisson field nonradial}, we have the pointwise estimate of the Poisson field (taking $\alpha=\frac{1}{4}$):
   \begin{align}\label{Lemma_monoto of inner norm: pointwise est of the diff of Poisson field}
       \int_{0}^{2\pi}|\nabla\Psi_{(1-\sqrt{\vr_\nu}\chi_\nu)\eps}(\theta,\zeta)|^2\;d\theta &\lesssim \frac{1+\one_{\{\zeta\leq 1\}}\log\zeta}{(1+\zeta)^{\frac{1}{2}}}\int (1-\sqrt{\vr_\nu}\chi_\nu)^2\eps^2(1+\zeta)^{\frac{1}{2}}\nonumber\\
       &\lesssim C(\zeta^*)\frac{1+\one_{\{\zeta\leq 1\}}\log\zeta}{(1+\zeta)^{\frac{1}{2}}}\dr{\|\eps\|^2_\inn+\|\eps(1+\zeta)^{\frac{3}{2}}\|^2_{L^\infty(\zeta\geq \zeta^*)}},
   \end{align}
   where we apply $(1-\sqrt{\vr_\nu}\chi_\nu)^2(1+\zeta)^{\frac{1}{2}}\lesssim (\nu+\zeta)^4\one_{\{\zeta<1\}}+(1+\zeta)^{\frac{1}{2}}\one_{\{\zeta\geq1\}}$. Then, by Cauchy's inequality and \eqref{Lemma_est of M: L^2 est},
   \begin{align*}
       \da{\Rmnum{2}}&\lesssim \frac{1}{\nu}\|\eps\|_\inn\dr{\int U_\nu|\nabla(\sqrt{\vr_\nu}\chi_\nu)|^2\cdot|\nabla\Psi_{(1-\sqrt{\vr_\nu}\chi_\nu)\eps}|^2}^{\frac{1}{2}}\\
       &\lesssim\frac{1}{\nu}\|\eps\|_\inn\dr{\int_{0}^{+\infty} U_\nu|\nabla(\sqrt{\vr_\nu}\chi_\nu)|^2\zeta\int_{0}^{2\pi}|\nabla\Psi_{(1-\sqrt{\vr_\nu}\chi_\nu)\eps}|^2\;d\theta d\zeta}^{\frac{1}{2}}\\
       &\lesssim C(\zeta^*)\frac{|\log\nu|}{\nu}\|\eps\|_\inn\dr{\|\eps\|_\inn+\|\eps(1+\zeta)^{\frac{3}{2}}\|_{L^\infty(\zeta\geq \zeta^*)}}.
   \end{align*}
   For $\Rmnum{3}$, using \eqref{Lemma_est of M: H^1 est}, \eqref{appendix: one Hardy-Poincare type ineq in parabolic} and Cauchy's inequality, we have
   \begin{align*}
       \da{\Rmnum{3}}&\leq C\dr{\int U_\nu\vr_\nu\chi_\nu^2\cdot|\nabla\Psi_{(1-\sqrt{\vr_\nu}\chi_\nu)\eps}|^2}^{\frac{1}{2}}\dr{\int\frac{|\nabla(\sqrt{\vr_\nu}\chi_\nu\eps)|^2}{U_\nu}}^{\frac{1}{2}}\\
       &\leq \frac{\delta}{10}\dr{\int\frac{(\eps^2+|\nabla\eps|^2)\chi_\nu^2\vr_\nu}{U_\nu}}+C(\zeta^*)|\log\nu|^2\dr{\|\eps\|^2_\inn+\|\eps(1+\zeta)^{\frac{3}{2}}\|^2_{L^\infty(\zeta\geq\zeta^*)}}.
   \end{align*}
   Finally, combining the estimates of $\Rmnum{1},\Rmnum{2},\Rmnum{3}$ above, we obtain
   \begin{equation}\label{Lemma_monoto of inner norm: Ls - Lst est}
       \da{\scl{(\Ls^\zeta_\nu-\Lst^\zeta_\nu)\eps}{\eps}_{\nu,*}}\leq \frac{\delta}{5}\dr{\int\frac{(\eps^2+|\nabla\eps|^2)\chi_\nu^2\vr_\nu}{U_\nu}}+C(\zeta^*)|\log\nu|^2\dr{\|\eps\|^2_\inn+\|\eps(1+\zeta)^{\frac{3}{2}}\|^2_{L^\infty(\zeta\geq\zeta^*)}}.
   \end{equation}
   \underline{Estimate of term $L(\eps)$:} 
    In the following, we denote the $\Oc(\mu^s)$ terms (under the bootstrap assumption) as the lower order terms (l.o.t.), as $\mu^s = \Oc(\nu^k)$ for any fixed $k>0$ when $\nu$ is sufficiently small. By Lemma \ref{appendix: difference of 2d 3d Poisson field}, we know that
    \[
        L(\eps) = -\nabla\cdot(\eps\nabla\Psi_P+P\nabla\Psi_\eps)+\frac{R_\tau}{\mu}\pa_{\br}\eps  + \text{l.o.t.}.
    \]
    Integrating by parts, we obtain
    \begin{align}
        \scl{-\nabla\cdot(\eps\nabla\Psi_P+P\Psi_\eps)}{\eps}_{\nu,*} &= \int \nabla(\sqrt{\vr_\nu}\chi_\nu)\cdot(\eps\nabla\Psi_P+P\nabla\Psi_\eps)\Ms^\zeta_\nu(\sqrt{\vr_\nu}\chi_\nu\eps)\nonumber\\
        &\quad+\int \sqrt{\vr_\nu}\chi_\nu(\eps\nabla\Psi_P+P\Psi_\eps)\cdot\nabla\Ms^\zeta_\nu(\sqrt{\vr_\nu}\chi_\nu\eps).
    \end{align}
    By Lemma \ref{lemma: properties of the adapted inn product}, inequality \eqref{appendix: one Hardy-Poincare type ineq in parabolic}, and the control of the outer $L^\infty$-norm, we first have
    \begin{equation}\label{Lemma_monoto of inner norm: collected Ms est}
        \int U_\nu |\Ms^\zeta_\nu(\eps\sqrt{\vr_\nu}\chi_\nu)|^2\lesssim \frac{1}{\nu^2}\|\eps\|^2_\inn,\quad \int U_\nu |\nabla\Ms^\zeta_\nu(\eps\sqrt{\vr_\nu}\chi_\nu)|^2\lesssim \frac{1}{\nu^2}\dr{\|\eps\|^2_\inn+\|\nabla\eps\|^2_\inn+\nu^{100}\|\eps\|^2_{L^\infty(\zeta\geq\zeta^*)}}.
    \end{equation}
    Then, thanks to Cauchy's inequality, it remains to estimate
    \[
        \int (\vr_\nu\chi^2_\nu+|\nabla(\sqrt{\vr}\chi_\nu)|^2)(|\eps\nabla\Psi_P|^2+|P\nabla\Psi_\eps|^2)\frac{1}{U_\nu}.
    \]
    By the pointwise estimate $|\nabla\Psi_P(\zeta)|\lesssim |a(\tau)|\log(1+\zeta/\nu)/\zeta\lesssim \frac{|a|}{\nu}$ and \eqref{appendix: one Hardy-Poincare type ineq in parabolic}, we have
    \[
        \int (\vr_\nu\chi^2_\nu+|\nabla(\sqrt{\vr_\nu}\chi_\nu)|^2)\frac{|\eps\nabla\Psi_P|^2}{U_\nu}\lesssim \frac{|a|^2}{\nu^4}\dr{\|\eps\|^2_\inn+\|\nabla\eps\|^2_\inn+\nu^{100}\|\eps\|^2_{L^\infty(\zeta\geq\zeta^*)}}.
    \]
    For the other part, we consider the decomposition $\nabla\Psi_\eps = \nabla\tilde \Psi_{\eps}+\nabla\Psi_{(1-\sqrt{\vr_\nu}\chi_\nu)\eps}$. By \eqref{Lemma_monoto of inner norm: pointwise est of the diff of Poisson field},
    \begin{align*}
         \int (\vr_\nu\chi^2_\nu+|\nabla(\sqrt{\vr_\nu}\chi_\nu)|^2)\frac{|P\nabla\Psi_{(1-\sqrt{\vr_\nu\chi_\nu})\eps}|^2}{U_\nu}\leq \frac{C(\zeta^*)a^2}{\nu^2}\dr{\|\eps\|^2_\inn+\|\eps(1+\zeta)^{\frac{3}{2}}\|^2_{L^\infty(\zeta\geq \zeta^*)}}.
    \end{align*}
    By the Hardy-Littlewood-Sobolev inequality \eqref{appendix: HLS ineq in our setting}:
    \[
    \|\nabla\tilde\Psi_\eps\|_{L^4}\lesssim \frac{1}{\nu}\|\eps\|_\inn\|U_\nu\|^{\frac{1}{2}}_{L^2}\lesssim\frac{1}{\nu^{\frac{3}{2}}}\|\eps\|_\inn,
    \]
    we have
    \begin{align*}
        \int (\vr_\nu\chi^2_\nu+|\nabla(\sqrt{\vr_\nu}\chi_\nu)|^2)\frac{|P\nabla\tilde\Psi_{\eps}|^2}{U_\nu}\lesssim \|\nabla\tilde\Psi_\eps\|^2_{L^4}\dr{\int\frac{(\vr_\nu\chi^2_\nu+|\nabla(\sqrt{\vr_\nu}\chi_\nu)|^2)^2P^4}{U^2_\nu}}^{\frac{1}{2}}\lesssim \frac{a^2}{\nu^5}\|\eps\|^2_{\inn}.
    \end{align*}
For the last term in $L(\eps)$, by Cauchy's inequality,
\begin{equation*}
    \da{\frac{R_\tau}{\mu}\scl{\eps}{\pa_{\br}\eps}_{\nu,*}}\lesssim \frac{1}{\nu^2}\da{\frac{R_\tau}{\mu}}\|\eps\|_\inn\|\nabla\eps\|_\inn.
\end{equation*}
Finally, collecting all the estimates above and by the bootstrap assumptions, we obtain
\begin{align}\label{Lemma_monoto of inner norm: L(eps) est}
    \da{\scl{L(\eps)}{\eps}_{\nu,*}}&\leq C\dr{\frac{|a|}{\nu^{\frac{3}{2}}}+\da{\frac{R_\tau}{\nu}}^{\frac{1}{2}}}\dr{\int\frac{(\eps^2+|\nabla\eps|^2)\chi^2_\nu\vr_\nu}{U_\nu}}+\frac{C|a|}{\nu^2}\|\eps(1+\zeta)^{\frac{3}{2}}\|^2_{L^\infty(\zeta\geq \zeta^*)}+C(\zeta^*,K_i)\nu^3\nonumber\\
    &\leq  \frac{\delta}{10}\dr{\int\frac{(\eps^2+|\nabla\eps|^2)\chi^2_\nu\vr_\nu}{U_\nu}}+\frac{C|a|}{\nu^2}\|\eps(1+\zeta)^{\frac{3}{2}}\|^2_{L^\infty(\zeta\geq \zeta^*)}+C(\zeta^*,K_i)\nu^3,
\end{align}
where the second inequality above holds when $\nu$ is sufficiently small.
\\
\underline{Estimate of the nonlinear term $NL(\eps)$: } Again, the terms of order $\Oc(\mu^s)$ are treated as lower order terms. It then suffices to estimate the term $-\nabla\cdot(\eps\nabla\Psi_\eps)$. Integrating by parts, we have
\[
   \scl{-\nabla\cdot(\eps\nabla\Psi_\eps)}{\eps}_{\nu,*} = \int \nabla(\sqrt{\vr_\nu}\chi_\nu)\cdot(\eps\nabla\Psi_\eps)\Ms^\zeta_\nu(\sqrt{\vr_\nu}\chi_\nu\eps)+\int \sqrt{\vr_\nu}\chi_\nu\eps\nabla\Psi_\eps\cdot\nabla\Ms^\zeta_\nu(\sqrt{\vr_\nu}\chi_\nu\eps).
\]
As before, by Cauchy's inequality and \eqref{Lemma_monoto of inner norm: collected Ms est}, it suffices to estimate
\[
    \int (\vr_\nu\chi_\nu^2+|\nabla(\sqrt{\vr_\nu}\chi_\nu)|^2)\eps^2|\nabla\Psi_\eps|^2\frac{1}{U_\nu}.
\]
To this end, we decompose $\nabla\Psi_\eps = \nabla\Psi_{\chi^*\eps}+\nabla\Psi_{(1-\chi^*)\eps}$, where we recall $\chi^*(\zeta):=\chi(\zeta/\zeta^*)$. For the first part, We further decompose: 
\begin{align*}
     \int (\vr_\nu\chi_\nu^2+|\nabla(\sqrt{\vr_\nu}\chi_\nu)|^2)\eps^2|\nabla\Psi_{\eps\chi^*}|^2\frac{1}{U_\nu} &= \int (\vr_\nu\chi_\nu^2+|\nabla(\sqrt{\vr_\nu}\chi_\nu)|^2)\eps^2(\chi^*)^2|\nabla\Psi_{\eps\chi^*}|^2\frac{1}{U_\nu}\\
     &\quad+\int (\vr_\nu\chi_\nu^2+|\nabla(\sqrt{\vr_\nu}\chi_\nu)|^2)\eps^2(1-(\chi^*)^2)|\nabla\Psi_{\eps\chi^*}|^2\frac{1}{U_\nu}\\
     &:= T_1+T_2.
\end{align*}
By Cauchy's inequality, Sobolev embedding $H^1(\reall^2)\hookrightarrow L^4(\reall^2)$, and HLS inequality \eqref{appendix: HLS ineq}, 
\begin{align*}
    |T_1|&\lesssim \|\nabla\Psi_{\eps\chi^*}\|^2_{L^4}\cdot\dr{\int \frac{\eps^4(\chi^*)^4(\nu+\zeta)^8}{\nu^4}}^{\frac{1}{2}}
    \lesssim \|\eps\chi^*\|^2_{L^2}\cdot \dr{\int \da{\nabla\dr{\frac{\eps\chi^*(\nu+\zeta)^2}{\nu}}}^2}\\
    &\quad\quad\lesssim \dr{\int \frac{\eps^2(\chi^*)^2(\nu+\zeta)^2}{\nu^2}}\cdot\dr{\int\frac{|\nabla\eps|^2(\chi^*)^2}{U_\nu}+\int\frac{\eps^2(\nu+\zeta)^2(\chi^*)^2}{\nu^2}}.
\end{align*}
By the $L^\infty$ control of $\eps$ in the far field, Cauchy's inequality, and HLS inequality,
\begin{align*}
    |T_2|\lesssim \frac{1}{\nu^2}\|\eps\|^2_{L^\infty(\zeta\geq \zeta^*)}\|\nabla\Psi_{\eps\chi^*}\|^2_{L^4}\lesssim \nu^2C(K_i)\int \frac{\eps^2(\chi^*)^2(\nu+\zeta)^2}{\nu^2} .
\end{align*}
In summary, by the inequality \eqref{appendix: one Hardy-Poincare type ineq in parabolic}, we obtain
\begin{align}\label{Lemma_monoto of inner norm: nonlinear est part one}
    \int (\vr_\nu\chi_\nu^2+|\nabla(\sqrt{\vr_\nu}\chi_\nu)|^2)\eps^2|\nabla\Psi_{\eps\chi^*}|^2\frac{1}{U_\nu}&\lesssim \dr{\frac{1}{\nu^4}\|\eps\|^2_\inn+\nu^2C(K_i)+\int\frac{|\nabla\eps|^2(\chi^*)^2}{U_\nu}}\int\frac{(|\nabla\eps|^2+\eps^2)\chi^2_\nu\sqrt{\vr_\nu}}{U_\nu}\nonumber\\
    &\lesssim \frac{C(K_i)}{|\log\nu|^2}\int\frac{(|\nabla\eps|^2+\eps^2)\chi^2_\nu\sqrt{\vr_\nu}}{U_\nu}.
\end{align}
For the second part, we apply \eqref{appendix: pointwise est of the gradient of Poisson field based on L infty of u}, and obtain (choosing $p=\frac{3}{2}$)
\begin{equation*}
    \|\nabla\Psi_{(1-\chi^*)\eps}\|_{L^\infty}\lesssim \|(1-\chi^*)\eps\|_{L^\infty}+\|(1-\chi^*)\eps\|_{L^\frac{3}{2}}\lesssim \|\eps(1+\zeta)^{\frac{3}{2}}\|_{L^\infty(\zeta\geq \zeta^*)}.
\end{equation*}
It follows that
\begin{align}\label{Lemma_monoto of inner norm: nonlinear est part two}
\begin{split}
    &\int \frac{(\vr_\nu\chi_\nu^2+|\nabla(\sqrt{\vr_\nu}\chi_\nu)|^2)\eps^2|\nabla\Psi_{\eps(1-\chi^*)}|^2}{U_\nu}\\
    &\quad\lesssim \|\eps(1+\zeta)^{\frac{3}{2}}\|^2_{L^\infty(\zeta\geq\zeta^*)}\int \frac{(\vr_\nu\chi_\nu^2+|\nabla(\sqrt{\vr_\nu}\chi_\nu)|^2)\eps^2}{U_\nu}\\
    &\quad\lesssim \frac{1}{\nu^2}\|\eps(1+\zeta)^{\frac{3}{2}}\|^2_{L^\infty(\zeta\geq\zeta^*)}\dr{{\|\eps\|^2_\inn+\|\nabla\eps\|^2_\inn+\nu^{100}\|\eps\|^2_{L^\infty(\zeta\geq\zeta^*)}}}.
\end{split}
\end{align}
Finally, combining \eqref{Lemma_monoto of inner norm: nonlinear est part one} and \eqref{Lemma_monoto of inner norm: nonlinear est part two}, we obtain the nonlinear estimate
\begin{align}\label{Lemma_monoto of inner norm: nonlinear estimate}
    \da{\scl{NL(\eps)}{\eps}_{\nu,*}}
    &\leq \frac{C(K_i)}{|\log\nu|}\int\frac{(|\nabla\eps|^2+\eps^2)\chi^2_\nu\sqrt{\vr_\nu}}{U_\nu}
    +C(K_i)\nu^{3}\nonumber\\
    &\leq \frac{\delta}{10}\int\frac{(|\nabla\eps|^2+\eps^2)\chi^2_\nu\sqrt{\vr_\nu}}{U_\nu}
    +C(K_i)\nu^{3},
\end{align}
where the second inequality holds when $\nu$ is sufficiently small.\\
\underline{Estimate of the generated error $E$:} Recall that
\[
    E = \mod_0\varphi_{0,\nu}+\mod_1\varphi_{1,\nu}+\frac{R_\tau}{\mu}\pa_{\br}U_\nu+\tilde E.
\]
By the algebraic identity $\Ms^\zeta_\nu (\Lambda U_\nu) = -2$, the orthogonality conditions \eqref{local orthorgonality conditions} and the decomposition $\varphi_{i,\nu} = -\frac{1}{16\nu^2}\Lambda U_\nu\chi_\nu+\tilde \varphi_i$, we have
\begin{align*}
    \da{\scl{\eps}{\mod_i\varphi_{i,\nu}}_{\nu,*}}&\lesssim |\mod_i|\dr{\da{\int \eps\sqrt{\vr_\nu}\chi_\nu\Ms^\zeta_\nu\Big(\frac{1}{\nu^2}(1-\sqrt{\vr_\nu}\chi^2_\nu)\Lambda U_\nu\Big)}+\da{\int \eps\sqrt{\vr_\nu}\chi_\nu\Ms^\zeta_\nu(\chi_\nu\sqrt{\vr_\nu}\tilde\varphi_i)}}\\
    &\lesssim  \frac{\delta}{10\nu^2}\|\eps\|^2_\inn + \frac{|\mod_i|^2}{\nu^2}
\end{align*}
Similarly, by the algebraic identity $\Ms^\zeta_\nu (\nabla U_\nu)=0$, we have
\begin{align*}
    \da{\int \eps\sqrt{\vr_\nu}\chi_\nu\Ms^\zeta_\nu\Big(\frac{R_\tau}{\mu}\pa_{\br}U_\nu\sqrt{\vr_\nu}\chi_\nu\Big)}& \lesssim \da{\int \eps\sqrt{\vr_\nu}\chi_\nu\Ms^\zeta_\nu\Big(\frac{R_\tau}{\mu}\pa_{\br}U_\nu(1-\sqrt{\vr_\nu}\chi_\nu)\Big)}\\
    &\lesssim \frac{\delta}{10\nu^2}\|\eps\|^2_\inn+\nu^2|\log\nu|\cdot\da{\frac{R_\tau}{\mu}}^2.  
\end{align*}
As for $\tilde E$, by Cauchy's inequality,
\[
    \da{\scl{\tilde E}{\eps}_{\nu,*}} \leq \frac{C}{\nu}\|\eps\|_\inn\dr{\int \frac{\tilde E^2}{U_\nu}}^\frac{1}{2}\leq \frac{\delta}{10\nu^2}\|\eps\|^2_\inn + \frac{C(\nu^2+|a|)}{|\log\nu|^2}+\frac{C(K_i)\nu^2}{|\log\nu|^3}.
\]
In summary, we have
\begin{equation*}
    \da{\scl{\eps}{E-\mod_0\vp_{0,\nu}}_{\nu,*}}\leq \frac{3\delta}{10\nu^2}\|\eps\|^2_\inn+\frac{C}{\nu^2}|\mod_1|^2 + \frac{C(\nu^2+|a|)}{|\log\nu|^2}+\frac{C(K_i)\nu^2}{|\log\nu|^3}.
\end{equation*}
\underline{Estimate of time derivative terms:} Once we note $\frac{\pa}{\pa_\tau} = \nu_\tau\frac{\pa}{\pa_\nu}$, then the estimates are straightforward from definition. First, we have
\begin{equation*}
    \da{\scl{\frac{\pa}{\pa_\tau}\dr{\frac{\vr_\nu\chi^2_\nu}{U_\nu}}}{\eps^2}}\lesssim \frac{\nu_\tau}{\nu}\dr{\frac{1}{\nu^2}\|\eps\|^2_\inn+\nu^{100}\|\eps\|^2_{L^\infty(\zeta\geq\zeta^*)}}.
\end{equation*}
Second, by $\pa_\nu(\sqrt{\vr_\nu}\chi_\nu) = \frac{\zeta\nu_\tau}{|\log\nu|^2\nu}\chi'(\zeta/|\log\nu|)\sqrt{\vr_\nu}$ and \eqref{appendix: pointwise est of Poisson field by weighted L^2 norm}, we have
\begin{equation*}
    \da{\scl{\eps\pa_\tau\dr{\sqrt{\vr_\nu}\chi_\nu}}{\tilde\Psi_\eps}_{\nu,*}}\lesssim \nu^{100}\nu_\tau\|\eps\|_{L^\infty(\zeta\geq\zeta^*)}\|\eps\|_\inn\lesssim \|\eps\|^2_\inn+\nu^{100}
\end{equation*}
\underline{Conclusion of Step 1:} Finally, collecting all the estimates above, we obtain, when $\nu$ is sufficiently small,
\begin{align}\label{lemma_L2 est: main part esti}
    &\quad\pa_\tau \frac{1}{2}\int \eps\sqrt{\vr_\nu}\chi_\nu\Ms^\zeta_\nu(\eps\sqrt{\vr_\nu}\chi_\nu) - \mod_0\int \eps\sqrt{\vr_\nu}\chi_\nu\Ms^\zeta_\nu(\vp_{0,\nu}\sqrt{\vr_\nu}\chi_\nu)\nonumber\\
    &\leq -\frac{\delta}{4\nu^2}\dr{\|\eps\|^2_\inn+\|\nabla\eps\|^2_\inn}+\frac{C}{\nu^2}|\mod_1|^2+\frac{C(K_2)\nu^2}{|\log\nu|^2}+C(K_i)\nu^4.
\end{align}
\textbf{Step 2: Correction term estimate}\\
Now we are to estimate the extra terms induced by $\pa_\tau (d_0\int\eps\wei\Ms^\zeta_\nu(\wei\varphi_{0,\nu}))$.
We write
\begin{align*}
    \pa_\tau\dr{d_0\int\eps\wei\Ms^\zeta_\nu(\wei\varphi_{0,\nu})} &= 2d_0\int\pa_\tau\eps\wei\Ms^\zeta_\nu(\wei\varphi_{0,\nu})\\
    &\;+ 2d_0\int\eps\pa_\tau(\wei)\Ms^\zeta_\nu(\wei\varphi_{0,\nu})+2d_0\int\eps\wei\Ms^\zeta_\nu(\pa_\tau\bar\varphi_{0})\\
    &\;+2d_0\int\eps\wei\bar\varphi_{0}\pa_\tau\dr{\frac{1}{U_\nu}}
    +2d_0\pa_\tau\dr{-\frac{1}{8\nu^2}}\int\eps\wei\\
    +&\pa_\tau\dr{\frac{1}{\int\wei\vp_{0,\nu}\Ms^\zeta_\nu(\wei\vp_{0,\nu})}}\dr{\int\wei\vp_{0,\nu}\Ms^\zeta_\nu(\wei\eps)}^2.
\end{align*}
\underline{Estimate of $d_0\int\pa_\tau\eps\wei\Ms^\zeta_\nu(\wei\varphi_{0,\nu})$:} Plug in the evolution equation for $\eps$:
\[
    \pa_\tau\eps = \Delta\eps-\nabla\cdot(U_\nu\nabla\Psi_\eps+\eps\nabla\Psi_{U_\nu})-\beta\Lambda\eps+L(\eps)+NL(\eps)+E,
\]
and we estimate term by term. First, through integration by parts and Cauchy's inequality,
\begin{align*}
    \da{d_0\int \wei\Delta \eps \Ms^\zeta_\nu(\wei\varphi_{0,\nu})}&\lesssim\frac{1}{\nu^2}\da{d_0\int \wei\Delta\eps}+\da{d_0\int \wei\Delta\eps\Ms^\zeta_\nu(\bar\varphi_{0})}\\
    &\lesssim \frac{|d_0|}{\nu^2}\da{\int (1+\zeta^2)\eps\wei}+\da{d_0\int (\wei\nabla\eps-\eps\nabla(\wei))\cdot\nabla\Ms^\zeta_\nu(\bar\varphi_0)}\\
    &\quad+\da{d_0\int \Delta(\wei)\eps\Ms^\zeta_\nu(\bar\varphi_0)}\\
    &\lesssim \frac{1}{\nu^2|\log\nu|^{\frac{1}{3}}}(\|\eps\|^2_\inn+\|\nabla\eps\|^2_\inn)+\nu^{100}\|\eps\|_{L^\infty(\zeta\geq\zeta^*)}\|\eps\|_\inn.
\end{align*}
where we use the estimate $|\int\eps \wei|\lesssim |\log\nu|^\frac{1}{2}(\int(\nu+\zeta)^2\eps^2\wei)^\frac{1}{2}+(\int(\nu+\zeta)^4\eps^2\wei)^\frac{1}{2}$, and (when $\nu$ is sufficiently small)
\begin{align*}
    \int\frac{\zeta^4(\nu+\zeta)^4\eps^2\wei}{\nu^2}&\lesssim \int_{\{\zeta\leq |\log\nu|^{\frac{2}{3}}\}}\frac{\zeta^4(\nu+\zeta)^4\eps^2\wei}{\nu^2}+\int_{\{\zeta> |\log\nu|^{\frac{2}{3}}\}}\frac{\zeta^4(\nu+\zeta)^4\eps^2\wei}{\nu^2}\\
    &\lesssim \frac{|\log\nu|^{\frac{4}{3}}}{\nu^2}\|\nabla\eps\|^2_\inn + \nu^{100}\|\eps\|^2_{L^\infty(\zeta\geq\zeta^*)}.
\end{align*}
Similarly, by the pointwise estimates of the Poisson field \eqref{appendix: pointwise est of Poisson field radial and nonradial, parabolic} \eqref{appendix: pointwise est of the gradient of Poisson field based on L infty of u}, we have the estimate
\begin{equation*}
    \da{d_0\int \wei\nabla\cdot(U_\nu\nabla\Psi_\eps+\eps\nabla\Psi_{U_\nu})\Ms^\zeta_\nu(\wei\varphi_{0,\nu})}\lesssim \frac{1}{|\log\nu|^\frac{1}{2}\nu^2}(\|\eps\|^2_\inn+\|\nabla\eps\|^2_\inn)+\|\eps(1+\zeta)^{\frac{3}{2}}\|^2_{L^\infty(\zeta\geq\zeta^*)}.
\end{equation*}
and (the extra linear term is smaller, but here a rough estimate is enough)
\begin{equation*}
    \da{d_0\int\wei L(\eps)\Ms^\zeta_\nu(\wei\varphi_{0,\nu})}\leq \frac{C(K_i)}{|\log\nu|^\frac{1}{2}\nu^2}(\|\eps\|^2_\inn+\|\nabla\eps\|^2_\inn)+C(K_i)\|\eps(1+\zeta)^{\frac{3}{2}}\|^2_{L^\infty(\zeta\geq\zeta^*)}+\Oc(\mu^s).
\end{equation*}
As for the scaling term, the estimate is similar to the $\Delta \eps$ term:
\begin{equation*}
     \da{d_0\int\wei \beta\Lambda\eps\Ms^\zeta_\nu(\wei\varphi_{0,\nu})}\lesssim \frac{1}{\nu^2|\log\nu|^\frac{1}{3}}\dr{\|\eps\|^2_\inn+\|\nabla\eps\|^2_\inn}+\Oc(\nu^{100}).
\end{equation*}
For the nonlinear term, by the estimates of Poisson filed and Cauchy's inequality:
\begin{align*}
    \da{d_0\int\wei  NL(\eps)\Ms^\zeta_\nu(\wei\varphi_{0,\nu})}&\lesssim \da{\frac{d_0}{\nu^2}\int \nabla(\wei)\cdot\nabla\Psi_\eps\eps}+\da{d_0\int\nabla(\wei)\cdot\nabla\Psi_\eps\eps\Ms^\zeta_\nu(\bar\vp_0)}\\
    &\quad+\da{d_0\int\wei\eps\nabla\Psi_\eps\cdot\nabla\Ms^\zeta_\nu(\bar\vp_0)}\\
    &\lesssim \frac{\|\eps\|^2_\inn+\|\nabla\eps\|^2_\inn}{\nu^2|\log\nu|^\frac{1}{2}}\dr{\frac{1}{\nu}\|\nabla\eps\|_\inn+|\log\nu|\cdot\|\eps(1+\zeta)^\frac{3}{2}\|_{L^\infty(\zeta\geq \zeta^*)}},
\end{align*}
where we use the estimates
\begin{align*}
    &\da{\frac{d_0}{\nu^2}\int \wei\zeta\pa_\zeta\Psi_\eps\eps}\lesssim \frac{|d_0|}{\nu} \dr{\int\frac{\eps^2\chi_\nu^2\vr_\nu(\nu+\zeta)^4}{\nu^2}}^{\frac{1}{2}}\dr{\int_{\{\zeta\leq|2\log\nu|\}} \frac{\zeta^2}{(\nu+\zeta)^4}|\nabla\Psi_\eps|^2}^\frac{1}{2}\\
    &\quad \lesssim \frac{\|\eps\|^2_\inn}{\nu^2|\log\nu|}\dr{\frac{1}{\nu}\|\nabla\eps\|_\inn+|\log\nu|^\frac{1}{2}\|\eps(1+\zeta)^\frac{3}{2}\|_{L^\infty(\zeta\geq \zeta^*)}},\\
    &\da{d_0\int\zeta\wei\pa_\zeta\Psi_\eps\eps\Ms^\zeta_\nu(\bar\vp_0)}\lesssim |d_0|\dr{\int \frac{\zeta^2(\nu+\zeta)^2\eps^2\chi_\nu^2\vr_\nu}{\nu^2}}^\frac{1}{2}\dr{\int\frac{\nu^2}{(\nu+\zeta)^2}|\nabla\Psi_\eps|^2|\Ms^\zeta_\nu(\bar\vp_0)|^2}^\frac{1}{2}\\
    &\quad \lesssim \frac{\|\eps\|^2_\inn}{\nu^2|\log\nu|}\dr{\frac{1}{\nu}\|\nabla\eps\|_\inn+|\log\nu|\cdot\|\eps(1+\zeta)^\frac{3}{2}\|_{L^\infty(\zeta\geq \zeta^*)}},\\
    &\da{d_0\int\wei\eps\nabla\Psi_\eps\cdot\nabla\Ms^\zeta_\nu(\bar\vp_0)}\lesssim |d_0|\dr{\int \frac{(\nu+\zeta)^2\eps^2\chi_\nu^2\vr_\nu}{\nu^2}}^\frac{1}{2}\dr{\int\frac{\nu^2}{(\nu+\zeta)^2}|\nabla\Psi_\eps|^2|\nabla\Ms^\zeta_\nu(\bar\vp_0)|^2}^\frac{1}{2}\\
    &\quad \lesssim \frac{\|\eps\|_\inn\|\nabla\eps\|_\inn}{\nu^2|\log\nu|^\frac{1}{2}}\dr{\frac{1}{\nu}\|\nabla\eps\|_\inn+|\log\nu|\cdot\|\eps(1+\zeta)^\frac{3}{2}\|_{L^\infty(\zeta\geq \zeta^*)}}.
\end{align*}
Finally, for the generated error, 
\begin{align*}
    \da{d_0\int\wei  (E-\mod_0\vp_{0,\nu})\Ms^\zeta_\nu(\wei\varphi_{0,\nu})}\lesssim \frac{|\mod_1|\cdot\|\eps\|_\inn}{\nu^2}+ \frac{\|\eps\|_\inn}{|\log\nu|^2}+C(K_i)\nu\|\eps\|_\inn, 
\end{align*}
where we use the estimate $|\int \tilde E\wei| \lesssim \frac{\nu^2}{|\log\nu|}+C(K_i)\nu^3$ and the fact $\int \pa_{\br}U_\nu\wei\Ms^\zeta_\nu(\varphi_{0,\nu}\wei)$=0. Finally, combining all these estimates, we obtain
\begin{align*}
    \da{d_0\int (\pa_\tau\eps-\mod_0\vp_{0,\nu})\wei\Ms^\zeta_\nu(\wei\vp_{0,\nu})}\lesssim C(K_i)\frac{\|\eps\|^2_\inn+\|\nabla\eps\|^2_\inn}{\nu^2|\log\nu|^\frac{1}{3}}+C(K_i)\frac{\nu^2}{|\log\nu|^3}.
\end{align*}
\underline{Estimate of the rest time derivative terms:} By $\pa_\tau \wei = \frac{\zeta}{|\log\nu|^2}\frac{\nu_\tau}{\nu}\chi'(\zeta/|\log\nu|)$ and the $L^\infty$ estimate of $\eps$ in the far field, we have
\begin{equation*}
    \da{d_0\int\eps\pa_\tau(\wei)\Ms^\zeta_\nu(\wei\vp_{0,\nu})}\leq  C(K_i)\nu^{100},
\end{equation*}
when $\nu$ is sufficiently small. Similarly, by Cauchy's inequality, we have
\begin{align*}
    \da{d_0\int\eps\wei\Ms^\zeta_\nu(\pa_\tau(\bar\varphi_{0})}+\da{d_0\int\eps\wei\bar\varphi_{0}\pa_\tau\dr{\frac{1}{U_\nu}}}+\da{d_0\pa_\tau\dr{-\frac{1}{8\nu^2}}\int\eps\wei}\leq C(K_i)\frac{\|\eps\|^2_\inn+\|\nabla\eps\|^2_\inn}{\nu^2|\log\nu|}. 
\end{align*}
For the last term,
\begin{align*}
    \da{\pa_\tau\dr{\frac{1}{\int\wei\vp_{0,\nu}\Ms^\zeta_\nu(\wei\vp_{0,\nu})}}}&\lesssim \frac{|\pa_\tau(\nu^{-2}\int\bar\vp_0)|+|\pa_\tau\int \bar\vp_0\Ms^\zeta_\nu(\bar\vp_0)|}{(\int\wei\vp_{0,\nu}\Ms^\zeta_\nu(\wei\vp_{0,\nu}))^2}\lesssim \frac{\nu^2}{|\log\nu|}.  
\end{align*}
Then, we obtain
\begin{equation*}
   \da{\pa_\tau\dr{\frac{1}{\int\wei\vp_{0,\nu}\Ms^\zeta_\nu(\wei\vp_{0,\nu})}}\dr{\int\wei\vp_{0,\nu}\Ms^\zeta_\nu(\wei\eps)}^2}\lesssim \frac{C(K_i)\|\eps\|^2_\inn}{\nu^2|\log\nu|}.
\end{equation*}
Finally, all the estimates above yields
\begin{equation}\label{Lemma L2 est: esti of the projection part}
    \pa_\tau\dr{d_0\int\eps\wei\Ms^\zeta_\nu(\wei\vp_{0,\nu})} = 2\mod_0\int\wei\vp_{0,\nu}\Ms^\zeta_\nu(\eps\wei)+\Oc\dr{C(K_i)\frac{\nu^2}{|\log\nu|^\frac{7}{3}}}.
\end{equation}
\textbf{Step 3: Conclusion:}\\
Combining \eqref{lemma_L2 est: main part esti} and \eqref{Lemma L2 est: esti of the projection part} yields the final result.
\end{proof}

\subsubsection{$H^1$ Inner Estimate}
By orthogonality conditions \eqref{local orthorgonality conditions} and Lemma \ref{Lemma: coercivity of gradient M}, we know that there exists a universal $C>0$, such that
\[
    \frac{1}{C}\int \frac{|\nabla\eps^*|^2}{U_\nu}-C\|\eps\|^2_\inn<-\int\Ls^\zeta_{0,\nu}(\eps^*)\Ms^\zeta_\nu(\eps^*) = \int U_\nu|\nabla\Ms^\zeta_\nu(\eps^*)|^2<C\int\frac{|\nabla\eps^*|^2}{U_\nu}.
\]
Thus, $\int U_\nu|\nabla\Ms^\zeta_\nu(\eps^*)|^2$ is equivalent to the norm $\|\nabla\eps^*\|_{L^2(U_\nu)}$ (with some negligible error). 

\begin{lemma}[Control of $\|\nabla\eps^*\|_{L^2(U_\nu)}$]\label{lemma: H^1 energy estimate}
    Let $w$ be a solution in the bootstrap regime $\BS(\tau_0,\tau_*, \zeta^*, M_0, \{K_i\}_{i=1}^7)$. Then, the following estimate holds on $[\tau_0,\tau_*]$:
    \begin{align}\label{lemma H^1 energy estimate: ineq in statement}
         \frac{1}{2}\frac{d}{d\tau}\int U_\nu|\nabla\Ms^\zeta_\nu(\eps^*)|^2 &\leq \frac{C}{\nu^2}\dr{\|\nabla\eps\|^2_\inn+\|\eps\|^2_\inn}+\frac{C}{\nu^2}\|\eps(1+\zeta)^\frac{3}{2}\|^2_{L^\infty(\zeta\geq\zeta^*)}\nonumber\\
       &\quad+\frac{C\nu^2}{|\log\nu|^2}+\frac{C(\zeta^*,K_i)\nu^2}{|\log\nu|^3},
    \end{align}
    for some universal $C$, some $C(\zeta^*,K_i)$ dependent on $\{\zeta^*\}\cup\{K_i\}_{i=1}^7$, and any constant $K>0$.
\end{lemma}
\begin{proof}
    First, the evolution of $\eps^*$ is:
    \begin{equation*}
        \frac{d}{d\tau}\eps^* = \Ls^\zeta_{0,\nu}\eps^*+[\chi^*,\Ls^\zeta_{0,\nu}]\eps+\chi^*\sqrt{\vr_\nu}\dr{-\beta\Lambda\eps+L(\eps)+NL(\eps)+E},
    \end{equation*}
    based on which we compute that
    \begin{align*}
         \frac{1}{2}\frac{d}{d\tau} \int U_\nu|\nabla\Ms^\zeta_\nu(\eps^*)|^2 &= \int U_\nu\nabla\Ms^\zeta_\nu(\eps^*)\cdot\nabla\Ms^\zeta_\nu(\pa_\tau\eps^*)+\frac{1}{2}\int(\pa_\tau U_\nu)|\nabla\Ms^\zeta_\nu(\eps^*)|^2\\
         &\quad+ \int U_\nu\nabla\Ms^\zeta_\nu(\eps^*)\cdot\nabla\dr{\eps\pa_\tau\dr{\frac{1}{U_\nu}}}\\
         =&-\int\Ls^\zeta_{0,\nu}\eps^*\Ms^\zeta_\nu(\Ls^\zeta_{0,\nu}\eps^*)+\int U_\nu\nabla\Ms^\zeta_\nu(\eps^*)\cdot\nabla\Ms^\zeta_\nu\dr{\chi^*\dr{-\beta\Lambda\eps+L(\eps)+NL(\eps)}}\\
         &\quad+\int U_\nu\nabla\Ms^\zeta_\nu(\eps^*)\cdot\nabla\Ms^\zeta_\nu\dr{\chi^*E+[\chi^*,\Ls^\zeta_{0,\nu}]\eps}\\
         &\quad+\frac{1}{2}\int(\pa_\tau U_\nu)|\nabla\Ms^\zeta_\nu(\eps^*)|^2
         + \int U_\nu\nabla\Ms^\zeta_\nu(\eps^*)\cdot\nabla\dr{\eps^*\pa_\tau\dr{\frac{1}{U_\nu}}}.
    \end{align*}
    \underline{Coercivity in $H^2$:}
    Denote $\eps_2 := \Ls^\zeta_{0,\nu}\eps^*$, and decompose
    \[
        \eps_2 = a_0\Lambda U_\nu+ a_1\pa_{\br} U_\nu+\tilde \eps_2,
    \]
    where
    \[
        a_0 = \frac{\scl{\eps_2}{\Lambda U_\nu}}{\scl{\Lambda U_\nu}{\Lambda U_\nu}},\quad a_1 = \frac{\scl{\eps_2}{\pa_{\br} U_\nu}}{\scl{\pa_{\br} U_\nu}{\pa_{\br} U_\nu}}.
    \]
    By \eqref{appendix: pointwise est of Poisson field radial and nonradial, parabolic} and Cauchy's inequality, We have the estimate
    \begin{align*}
        |a_0|\lesssim \nu^2 \int |(\nabla\eps^*-U_\nu\nabla\Psi_{\eps^*}-\eps^*\nabla \Psi_{U_\nu})\cdot\nabla(\Lambda U_\nu)|\lesssim \frac{1}{\nu^2}\|\nabla\eps^*\|_{L^2}\lesssim \frac{C(K_i)}{|\log\nu|}. 
    \end{align*}
    Similarly, we have
    \begin{equation*}
        |a_1|\lesssim\frac{C(K_i)\nu}{|\log\nu|}.
    \end{equation*}
    Then, by the algebraic identities $\Ms^\zeta_\nu(\Lambda U_\nu)=-2\;, \Ms^\zeta_\nu(\nabla U_\nu)=0$, the divergence form of $\eps_2$, and \eqref{Lemma_est of M: L^2 coercivity}, we have 
    \[
       \int \eps_2\Ms^\zeta_\nu(\eps_2) = \int\tilde\eps_2\Ms^\zeta_\nu(\tilde\eps_2) \approx \int \frac{|\tilde\eps_2|^2}{U_\nu} \approx \int \frac{|\eps_2|^2}{U_\nu} +\frac{C(K_i)}{|\log\nu|^2},
    \]
    where we denote for two non-negative quantities $A\approx B$, if there exists a universal constant $c>0$, such that $cA\leq B\leq\frac{1}{c}A$.
    Then, we are to show that $\int\frac{|\eps_2|^2}{U_\nu}$ is equivalent to certain weighted $H^2$-norm for $\eps^*$.
    Define
    \[
       \tilde\eps^*:= \eps^* - \frac{\scl{\eps^*}{\Lambda U_\nu}}{\scl{\Lambda U_\nu}{\Lambda U_\nu}}\Lambda U_\nu - \frac{\scl{\eps^*}{\nabla U_\nu}}{\scl{\nabla U_\nu}{\nabla U_\nu}}\cdot\nabla U_\nu:= \eps^* - c_1\Lambda U_\nu - \mathbf{c}_2\cdot\nabla U_\nu.
    \]
    Thus, we have
    \[
        \scl{\tilde\eps^*}{\Lambda U_\nu} = \scl{\tilde\eps^*}{\nabla U_\nu}=0.
    \]
    By the local orthogonality conditions \eqref{local orthorgonality conditions}, we estimate that
    \begin{align*}
        |c_1| \lesssim \nu^2\da{\int\eps^*\Lambda U_\nu} \lesssim \nu^4\|\eps^*\|_{L^2},\quad
        |\mathbf{c}_2| \lesssim \nu^4\da{\int\eps^*\nabla U_\nu} \lesssim \nu^6\|\eps^*\|_{L^2}.
    \end{align*}
    Then, by the estimates above and Proposition \ref{proposition: high order coercivity} (in the parabolic variables),
    \begin{align*}
        \int\frac{|\Ls^\zeta_{0,\nu}\eps^*|^2}{U_\nu}=\int\frac{|\Ls^\zeta_{0,\nu}\tilde\eps^*|^2}{U_\nu}&\geq\delta\dr{\int\frac{|\Delta\tilde\eps^*|^2}{U_\nu}+\frac{1}{\nu^2}\int (\nu+\zeta)^2|\nabla\tilde\eps^*|^2+\frac{1}{\nu^2}\int{\tilde\eps^*}^2}\\
        &\geq \delta\dr{\int\frac{|\Delta\eps^*|^2}{U_\nu}+\frac{1}{\nu^2}\int (\nu+\zeta)^2|\nabla\eps^*|^2+\frac{1}{\nu^2}\int{\eps^*}^2} - C\nu^4\|\eps^*\|^2_{L^2}\\
        &\geq \frac{\delta}{2}\dr{\int\frac{|\Delta\eps^*|^2}{U_\nu}+\frac{1}{\nu^2}\int (\nu+\zeta)^2|\nabla\eps^*|^2+\frac{1}{\nu^2}\int{\eps^*}^2},
    \end{align*}
    when $\nu$ is small enough. Since $\eps^*$ is compactly supported, integration by parts yields the following control (one can, for example, apply the density argument by considering the functions in $\mathcal{D}(\reall^2)$ first):
    \begin{equation*}
        \int |\nabla^{(2)}\eps^*|^2\zeta^{2p} \leq C(p)\dr{\int|\Delta\eps^*|^2\zeta^{2p}+\int |\nabla\eps^*|^2\zeta^{2p-2} },\quad p=1,2.
    \end{equation*}
    It follows that there exists some $\delta'>0$, such that
    \begin{equation*}
        \int\frac{|\Ls^\zeta_{0,\nu}\eps^*|^2}{U_\nu}\geq \delta'\dr{\int\frac{|\nabla^{(2)}\eps^*|^2(\nu+\zeta)^4}{\nu^2}+\frac{1}{\nu^2}\int (\nu+\zeta)^2|\nabla\eps^*|^2+\frac{1}{\nu^2}\int{\eps^*}^2}.
    \end{equation*}
    For brevity, in the following we denote
    \[
        \|\eps^*\|^2_{H^2_\#}:= \int\frac{|\nabla^{(2)}\eps^*|^2(\nu+\zeta)^4}{\nu^2}+\frac{1}{\nu^2}\int (\nu+\zeta)^2|\nabla\eps^*|^2+\frac{1}{\nu^2}\int{\eps^*}^2.
    \]
    Finally, gathering all the results above, we obtain 
    \begin{align*}
        -\int\Ls^\zeta_{0,\nu}\eps^*\Ms^\zeta_\nu(\Ls^\zeta_{0,\nu}\eps^*) &\leq -\delta'\|\eps^*\|^2_{H^2_\#}+\frac{C}{\nu^2}\dr{\int(\nu+\zeta)^2|\nabla\eps^*|^2+\int\frac{(\nu+\zeta)^2(\eps^*)^2}{\nu^2}}\\
        &\leq-\delta'\|\eps^*\|^2_{H^2_\#}+ \frac{C}{\nu^4}\dr{\|\eps\|^2_\inn+\|\nabla\eps\|^2_\inn}\leq -\delta'\|\eps^*\|^2_{H^2_\#}+\frac{C(K_i)}{|\log\nu|^2},
    \end{align*} 
    and
    \begin{align*}
        -\int\Ls^\zeta_{0,\nu}\eps^*\Ms^\zeta_\nu(\Ls^\zeta_{0,\nu}\eps^*) &\leq -\delta\int\frac{\tilde\eps^2_2}{U_\nu},
    \end{align*}
    for some universal $\delta,\delta',C>0$. In other words, it means that there is a partial $H^2$-damping (i.e., damping in a certain finite codimensional subspace) and a full $H^2$-damping with an error of size $\frac{1}{|\log\nu|^2}$.\\
    \underline{Estimate of the scaling term $-\beta\Lambda\eps$:} Note that
    \[
        -\beta\int U_\nu\nabla\Ms^\zeta_\nu(\eps^*)\cdot\nabla\Ms^\zeta_\nu(\chi^*\Lambda\eps) = -2\beta\int U_\nu|\nabla\Ms^\zeta_\nu(\eps^*)|^2-\beta\int U_\nu\nabla\Ms^\zeta_\nu(\eps^*)\cdot\nabla\Ms^\zeta_\nu(\chi^*\y\cdot\nabla\eps),
    \]
    where the first term of the right-hand side has the desirable sign. As for the second term, by the identities $\frac{\y\cdot\nabla\eps^*}{U_\nu}=\y\cdot\nabla\dr{\frac{\eps^*}{U_\nu}}+\frac{\y\cdot\nabla U_\nu\eps^*}{U^2_\nu}$ and $\Psi_{\y\cdot\nabla\eps^*}=\y\cdot\nabla\Psi_{\eps^*}-2\Psi_{\eps^*}$ (since $\eps^*$ is compacted supported, the Poisson fields are all well defined, so that this identity can be verified by computing the Laplacian on the right-hand side), we have 
    \[
       \Ms^\zeta_\nu(\chi^*\y\cdot\nabla\eps) =\y\cdot\nabla\Ms^\zeta_\nu(\eps^*) -\Ms^\zeta_\nu(\y\cdot\nabla(\chi^*)\eps)+\frac{\y\cdot\nabla U_\nu\eps^*}{U^2_\nu}+2\Ms^\zeta_\nu(\eps^*)-\frac{2\eps^*}{U_\nu}.
    \]
    Then, through integration by parts, Cauchy's inequality and \eqref{Lemma_est of M: H^1 est}, we obtain
    \begin{equation*}
        \da{\beta\int U_\nu\nabla\Ms^\zeta_\nu(\eps^*)\cdot\nabla\Ms^\zeta_\nu(\chi^*\y\cdot\nabla\eps)}\leq C\dr{\frac{1}{\nu^2}\|\eps\|^2_\inn+\frac{1}{\nu^2}\|\nabla\eps\|^2_\inn},
    \end{equation*}
    for some $C=C(\zeta^*)$.
    It follows that
    \begin{equation}
         -\beta\int U_\nu\nabla\Ms^\zeta_\nu(\eps^*)\cdot\nabla\Ms^\zeta_\nu(\chi^*\Lambda\eps)\leq  C\dr{\frac{1}{\nu^2}\|\eps\|^2_\inn+\frac{1}{\nu^2}\|\nabla\eps\|^2_\inn}.
    \end{equation}
    \underline{Estimate of $L(\eps)$} As before, we neglect the terms of order $\Oc(\mu^s)$, thanks to \eqref{appendix: 3&2D near field approx}. Thus, 
    \[
        L(\eps) = -\nabla\cdot(\eps\nabla\Psi_P + P \nabla\Psi_\eps)+\frac{R_\tau}{\mu}\pa_{\br}\eps+\text{l.o.t.}.
    \]
    Recall the pointwise estimates $|\nabla\Psi_P(\zeta)|\lesssim \frac{|a|}{\nu}$ and $|\pa^k_\zeta P(\zeta)|\lesssim \frac{|a|\zeta^{2-k}\log(4+\zeta)}{(\nu+\zeta)^4}$. The following estimate relies on the structure of $\Ms^\zeta_\nu$, specifically $\Ms^\zeta_\nu\Lambda U_\nu = -2$ and $\Ms^\zeta_\nu\nabla U_\nu = 0$. First, we note that
    \begin{align*}
        \int \eps_2\Ms^\zeta_\nu(\chi^*\nabla\cdot(\eps\nabla\Psi_P+P\nabla\Psi_\eps)) &= \int \nabla\cdot(\eps^*\nabla\Psi_P+P\chi^*\nabla\Psi_\eps)\Ms^\zeta_\nu(\tilde\eps_2)\\
        &\quad-\int U_\nu\nabla\Ms^\zeta_\nu(\eps^*)\cdot\nabla\Ms^\zeta_\nu(\nabla\chi^*\cdot\nabla\Psi_P\eps+\nabla\chi^*\cdot\nabla\Psi_\eps P),
    \end{align*}
    where we use the decomposition $\eps_2 = \tilde\eps_2+a_0\Lambda U_\nu + a_1\pa_{\br} U_\nu$ and $\int \nabla\cdot(\eps^*\nabla\Psi_P+P\chi^*\nabla\Psi_\eps) = 0$. By Cauchy's inequality, we obtain
    \begin{align*}
        \da{\int \nabla\cdot(\eps^*\nabla\Psi_P+P\chi^*\nabla\Psi_\eps)\Ms^\zeta_\nu(\tilde\eps_2)}&\leq \frac{\delta}{10}\int\frac{\tilde\eps_2^2}{U_\nu}+C(\zeta^*)\int\frac{|\nabla\cdot(\eps^*\nabla\Psi_P+P\chi^*\nabla\Psi_\eps)|^2}{U_\nu}\\
        &\leq  \frac{\delta}{10}\int\frac{\tilde\eps_2^2}{U_\nu}+C(\zeta^*)\dr{\frac{a^2}{\nu^4}\|\nabla\eps\|^2_\inn+\frac{a^2}{\nu^2}\|\eps(1+\zeta)^\frac{3}{2}\|^2_{L^\infty(\zeta\geq\zeta^*)}},
    \end{align*}
    and
    \begin{align*}
        \da{\int U_\nu\nabla\Ms^\zeta_\nu(\eps^*)\cdot\nabla\Ms^\zeta_\nu(\nabla\chi^*\cdot\nabla\Psi_P\eps+\nabla\chi^*\cdot\nabla\Psi_\eps P)}\leq C(\zeta^*)\frac{|a|}{\nu^3}\dr{\|\eps\|^2_\inn+\|\nabla\eps\|^2_\inn+\|\eps(1+\zeta)^\frac{3}{2}\|^2_{L^\infty(\zeta\geq\zeta^*)}},
    \end{align*}
    where we use the elliptic regularity
    \begin{equation}\label{Lemma H1 est: elliptic regularity}
        \int_{\zeta^*\leq\zeta\leq2\zeta^*} |\nabla^{(2)}\Psi_\eps|^2\leq C(\zeta^*)  \int_{\frac{1}{2}\zeta^*\leq\zeta\leq 4\zeta^*} |\nabla\Psi_\eps|^2+\eps^2.
    \end{equation}
    With the same argument, we can estimate
    \begin{equation*}
        \da{\int\eps_2\Ms^\zeta_\nu(\chi^*\frac{R_\tau}{\mu}\pa_{\br}\eps)}\leq \frac{\delta}{10}\int\frac{\tilde\eps_2^2}{U_\nu}+\da{\frac{R_\tau}{\mu}}^2\frac{C}{\nu^2}(\|\eps\|^2_\inn+\|\nabla\eps\|^2_\inn).
    \end{equation*}
    In summary, we obtain
    \begin{equation*}
        \da{\int\eps_2\Ms^\zeta_\nu(\chi^*L(\eps))}\leq \frac{\delta}{5}\int\frac{\tilde\eps^2_2}{U_\nu}+\frac{C(\zeta^*,K_i)}{\nu}\dr{\|\eps\|^2_\inn+\|\nabla\eps\|^2_\inn}.
    \end{equation*}
    \underline{Estimate of the nonlinear term $NL(\eps)$:} We neglect the term of order $\Oc(\mu^s)$, and as before, rely on the structure of $\Ms^\zeta_\nu$. First, we write
    \begin{equation}\label{Lemma_H1 est: nonlinear part}
        \int\eps_2\Ms^\zeta_\nu(\chi^*\nabla\cdot(\eps\nabla\Psi_\eps)) =  \int\tilde\eps_2\Ms^\zeta_\nu(\nabla\cdot(\eps^*\nabla\Psi_\eps))- \int U_\nu\nabla\Ms^\zeta_\nu(\eps^*)\cdot\nabla\Ms^\zeta_\nu(\nabla\chi^*\cdot\nabla\Psi_\eps\eps).
    \end{equation}
    Then, we are to derive $L^\infty$-bounds for $\nabla\Psi_\eps$ and $\eps$. Denote $q^*(\gamma) = \nu^2\eps^*(\nu\gamma)$. Then, $\nabla\Psi_{q^*}(\gamma) = \nu\nabla\Psi_{\eps^*}(\nu\gamma)$. By Sobolev embedding and the pointwise estimates of the Poisson field \eqref{appendix: pointwise est of Poisson field radial}\eqref{appendix: pointwise est of Poisson field nonradial}, we have
    \[
        \|\nabla\Psi_{q^*}\|_{L^\infty} \lesssim \|\nabla q^*\|_{L^2}+\|\nabla \Psi_{q^*}\|_{L^2}\leq C(\zeta^*)\|\nabla\eps^*\|_{L^2(U_\nu)} \leq\frac{C(\zeta^*,K_i)\nu^2}{|\log\nu|}. 
    \]
    Thus,
    \[
        \|\nabla\Psi_{\eps^*}\|_{L^\infty}\leq \frac{C(\zeta^*K_i)\nu}{|\log\nu|}.
    \]
    Moreover, by \eqref{appendix: pointwise est of the gradient of Poisson field based on L infty of u}, we have 
    \[
        \|\nabla\Psi_{(1-\chi^*)\eps}\|_{L^\infty} \lesssim \|(1-\chi^*)\eps\|_{L^\infty} +\|(1-\chi^*)\eps\|_{L^\frac{3}{2}}\lesssim \|\eps(1+\zeta)^{\frac{3}{2}}\|_{L^\infty(\zeta\geq\zeta^*)}\leq C(K_i)
        \nu^2.
    \]
    In summary, we've shown that
    \begin{align*}
        \|\nabla\Psi_\eps\|_{L^\infty}\leq \frac{C(\zeta^*,K_i)\nu}{|\log\nu|}.
    \end{align*}
    Similarly, Sobolev embedding yields
    \begin{align*}
        \|q^*\|_{L^\infty}\lesssim \|q^*\|_{H^2} \lesssim \nu^2\|\eps^*\|_{H^2_\#}\lesssim \nu^2\int\eps_2\Ms^\zeta_\nu(\eps_2) +\frac{C(K_i)\nu^2}{|\log\nu|}. 
    \end{align*}
    It then follows that
    \[
        \|\eps^*\|_{L^\infty} \lesssim \int\eps_2\Ms^\zeta_\nu(\eps_2) + \frac{C(K_i)}{|\log\nu|} = \int\tilde\eps_2\Ms^\zeta_\nu(\tilde\eps_2) + \frac{C(K_i)}{|\log\nu|}
    \]
    Now coming back to \eqref{Lemma_H1 est: nonlinear part}, by Cauchy's inequality, we obtain
    \begin{align*}
        \da{\int\eps_2\Ms^\zeta_\nu(\nabla\cdot(\eps^*\nabla\Psi_\eps))}\leq \dr{\frac{\delta}{20}+\frac{C\|\eps\|_\inn}{\nu}}\int\frac{\tilde\eps^2_2}{U_\nu}+\frac{C(\zeta^*,K_i)}{|\log\nu|^2\nu^2}\dr{\|\eps\|^2_\inn+\||\nabla\eps\|^2_\inn},
    \end{align*}
    and
    \begin{align*}
        \da{ \int U_\nu\nabla\Ms^\zeta_\nu(\eps^*)\cdot\nabla\Ms^\zeta_\nu(\nabla\chi^*\cdot\nabla\Psi_\eps\eps)}\leq \frac{C(\zeta^*,K_i)}{\nu}\|\nabla\eps\|^2_\inn+\frac{C(\zeta^*)}{\nu^3}(\|\eps\|^4_\inn+\|\eps\|^4_{L^\infty(\zeta\geq \zeta^*)}),
    \end{align*}
    where we again use \eqref{Lemma H1 est: elliptic regularity}.
    In summary, by the bootstrap assumptions, we have the estimate 
    \begin{align*}
        \da{\int\eps_2\Ms^\zeta_\nu(\chi^*NL(\eps))} \leq \frac{\delta}{4}\int\frac{\tilde\eps^2_2}{U_\nu}+\frac{C(\zeta^*,K_i)\nu^2}{|\log\nu|^3}+\Oc(\mu^s),
    \end{align*}
    when $\nu$ is sufficiently small.\\[3pt]
    \underline{Estimate of the rest terms }: We write
    \[
        [\chi^*,\Ls^\zeta_{0,\nu}]\eps = -2\nabla\chi^*\cdot\nabla\eps -\eps\Delta\chi^*+\eps\nabla\chi^*\cdot\nabla\Psi_{U_\nu}+\nabla\cdot(U_\nu\nabla\Psi_{\eps^*}-\chi^*U_\nu\nabla\Psi_\eps)+U_\nu\nabla\chi^*\cdot\nabla\Psi_\eps.
    \]
    First, by Cauchy's inequality, 
    \begin{align*}
        \da{\int U_\nu\nabla\Ms^\zeta_\nu(\eps^*)\cdot\nabla\Ms^\zeta_\nu(-2\nabla\chi^*\cdot\nabla\eps -\eps\Delta\chi^*+\eps\nabla\chi^*\cdot\nabla\Psi_{U_\nu})}\leq \frac{C}{\nu^2}\dr{\|\nabla\eps\|^2_\inn+\|\eps\|^2_\inn}.
    \end{align*}
    Second, by the pointwise estimates of the Poisson field, the bootstrap assumptions and inequality \eqref{appendix: pointwise est of the gradient of Poisson field based on L infty of u},
    \begin{align*}
        \da{\int\eps_2\Ms(\nabla\cdot(U_\nu\nabla\Psi_{\eps^*}-\chi^*U_\nu\nabla\Psi_\eps))} &= \da{\int\tilde\eps_2\Ms(\nabla\cdot(U_\nu\nabla\Psi_{\eps^*}-\chi^*U_\nu\nabla\Psi_\eps))}\\
        &\leq \da{\int\tilde \eps_2\Ms(\nabla\cdot(U_\nu(1-\chi^*)\nabla\Psi_{\eps^*})} +\da{\int\tilde\eps_2\Ms(\nabla\cdot(U_\nu\chi^*\nabla\Psi_{\eps(1-\chi^*)})}\\
        &\leq \frac{\delta}{10}\int\frac{\tilde\eps_2^2}{U_\nu}+\frac{C}{\nu^2}\|\eps(1+\zeta)^\frac{3}{2}\|^2_{L^\infty(\zeta\geq\zeta^*)}+C(\zeta^*,K_i)\nu^3.
    \end{align*}
    Similarly, we have the estimate
    \begin{equation*}
        \da{\int U_\nu\nabla\Ms^\zeta_\nu(\eps^*)\cdot\nabla\Ms^\zeta_\nu(U_\nu\nabla\chi^*\cdot\nabla\Psi_\eps)}\leq \frac{C}{\nu^2}\dr{\|\nabla\eps\|^2_\inn+\|\eps\|^2_\inn}+C(\zeta^*,K_i)\nu^3.
    \end{equation*}
    To summarize, we have
    \begin{align*}
        \da{\int \eps_2\Ms^\zeta_\nu([\chi^*,\Ls^\zeta_{0,\nu}]\eps)}\leq \frac{\delta}{10}\int\frac{\tilde\eps^2_2}{U_\nu}+ \frac{C}{\nu^2}\dr{\|\nabla\eps\|^2_\inn+\|\eps\|^2_\inn}+\frac{C}{\nu^2}\|\eps(1+\zeta)^\frac{3}{2}\|^2_{L^\infty(\zeta\geq\zeta^*)}+C(\zeta^*,K_i)\nu^3.
    \end{align*}
    The estimate of the generated error relies on the structure of the operator $\Ms^\zeta_\nu$. Note that
    \begin{align*}
        &-\int U_\nu\nabla\Ms^\zeta_\nu(\eps^*)\cdot\nabla\Ms^\zeta_\nu(\chi^* E)=\int\eps_2\Ms^\zeta(\chi^* E) \\
        &\quad= -\sum_{i=0,1}\mod_i\int U_\nu\nabla\Ms^\zeta_\nu(\eps^*)\cdot\nabla\Ms^\zeta_\nu\dr{-\frac{1}{16\nu^2}(\chi^*-1)\Lambda U_\nu+\tilde\vp_i}
        +a_0\int\chi^*\tilde E+\int\tilde\eps_2\Ms^\zeta_\nu(\chi^*\tilde E).
    \end{align*}
    Then, by Cauchy's inequality, the estimates of $\tilde E$ and the estimate for $|\mod_i|$ in \eqref{Lemma modulation eq: inequalities}, we obtain
    \begin{align}
        \da{\int\eps_2\Ms^\zeta_\nu(\chi^*E)}\leq \frac{\delta}{10}\int\frac{\tilde\eps_2^2}{U_\nu}+\frac{C}{\nu^2}\dr{\|\eps\|^2_\inn+\|\nabla\eps\|^2_\inn}+\frac{C(\nu^2+|a|)}{|\log\nu|^2}+\frac{C(K_i)\nu^2}{|\log\nu|^3}.
    \end{align}
    Finally, for the extra time derivative terms, note that $\pa_\tau = \frac{\nu_\tau}{\nu}\nu\pa_\nu$, and the estimates are straightforward: 
    \begin{equation*}
        \frac{1}{2}\int(\pa_\tau U_\nu)|\nabla\Ms^\zeta_\nu(\eps^*)|^2
         + \int U_\nu\nabla\Ms^\zeta_\nu(\eps^*)\cdot\nabla\dr{\eps^*\pa_\tau\dr{\frac{1}{U_\nu}}}\leq \frac{C(K_i)}{\nu^2|\log\nu|}\dr{\|\eps\|^2_\inn+\|\nabla\eps\|^2_\inn}.
    \end{equation*}
    \underline{Conclusion:} Collecting all the estimates above, we obtain:
    \begin{align*}
       \frac{1}{2}\frac{d}{d\tau}\int U_\nu|\nabla\Ms^\zeta_\nu(\eps^*)|^2 &\leq \frac{C}{\nu^2}\dr{\|\nabla\eps\|^2_\inn+\|\eps\|^2_\inn}+\frac{C}{\nu^2}\|\eps(1+\zeta)^\frac{3}{2}\|^2_{L^\infty(\zeta\geq\zeta^*)}\\
       &\quad+\frac{C(\nu^2+|a|)}{|\log\nu|^2}+\frac{C(\zeta^*,K_i)\nu^2}{|\log\nu|^3},
    \end{align*}
    for some universal constants $\delta',C>0$ and constant $C(K_i)$ depending on $K_i\,(1\leq i\leq 7)$.
\end{proof}

\subsubsection{Higher Order Estimates in the Middle Range}
\begin{lemma}[$H^2$ control of $\eps$ in the middle range]\label{lemma: H^2 control in the middle}
    Let $w$ be a solution in the bootstrap regime\\ $\BS(\tau_0,\tau_*, \zeta^*, M_0, \{K_i\}_{i=1}^7)$. Then, for any $0<\zeta_1<\zeta_2$ and  $\tau\in [\tau_0,\tau_*]$ we have the following estimate
    \begin{equation}
        \frac{d}{d\tau}\|\eps\|^2_{H^2_*(\zeta_1,\zeta_2)}\leq -\delta(\zeta_1,\zeta_2)\|\eps\|^2_{H^2_*(\zeta_1,\zeta_2)} +\frac{C(\zeta_1,\zeta_2)K_4\nu^2}{|\log\nu|}|\mod_0|+ \frac{C(\zeta_1,\zeta_2)K_4^2\nu^4}{|\log\nu|^2}+\frac{C(\zeta_1,\zeta_2,K_i)\nu^4}{|\log\nu|^3},
    \end{equation}
    with the norm $\|\cdot\|_{H^2_*(\zeta_1,\zeta_2)}$ defined in \eqref{lemma higher order est in the middle: def of H^2 norm}, and positive constants 
    $\delta(\zeta_1,\zeta_2)$ and $C(\zeta_1,\zeta_2)$ depending only on $\zeta_1,\zeta_2$.
\end{lemma}
\begin{proof}
    The main idea of the proof is the parabolic regularity together with pointwise estimates in the middle range. To start with, from the control of the inner norm of $\eps$, we have
    \begin{equation}
        \|\eps\|_{L^2(\frac{1}{8}\zeta_1\leq \zeta\leq8\zeta_2)}\leq C(\zeta_1,\zeta_2)\|\eps\|_{\inn}\leq C(\zeta_1,\zeta_2)\frac{K_4\nu^2}{|\log\nu|}.
    \end{equation}
    The evolution of $\eps$ can be written as
    \begin{equation}
        \pa_\tau\eps = \Delta \eps +\Gc\cdot\nabla\eps+\Fc\eps-\nabla W\cdot\nabla\Psi_\eps-\nabla\cdot(\eps\nabla\Psi_\eps)+E+\text{l.o.t}, 
    \end{equation}
    where we recall $W= U_\nu+P$ and
    \begin{align*}
        \Fc:= 2W-2\beta,\quad \Gc := -\nabla\Psi_W-\beta\y+ \frac{R_\tau}{\mu}\mathbf{e}_1. 
    \end{align*}
    Here l.o.t. (lower order terms) denotes the terms that, up to second derivatives, can be estimated in the middle range with order $\Oc(\mu^s)$, and hence negligible in our analysis (we will omit it in the rest of this proof).
    Note that in the middle range $\frac{1}{8}\zeta_1\leq\zeta\leq 8\zeta_2$, we have for $k=0,1,2$,
    \begin{align*}
        |\pa_\zeta^{(k)}W|\leq C(\zeta_1,\zeta_2)\dr{ \nu^2+ |a(\tau)|},\quad |\pa_\zeta^{(k+1)}\Psi_W|\leq C(\zeta_1,\zeta_2)\dr{ 1+\frac{|a|}{\nu}}.
    \end{align*}
    Also, by \eqref{Prop3: decomposition of generated error}, \eqref{expression of tilde E}, \eqref{prop1: pointwise est 1}, and Lemma \ref{Lemma: modulation equations}, we have for any $\zeta\in [\frac{1}{8}\zeta_1,8\zeta_2]$ and any $k\geq 0$,
    \begin{equation*}
        |\pa_\zeta^{(k)}E|\leq C(\zeta_1,\zeta_2)|\mod_0| +\frac{C(\zeta_1,\zeta_2)\nu^2}{|\log\nu|}+\frac{C(\zeta_1,\zeta_2,K_i)\nu^2}{|\log\nu|^2}. 
    \end{equation*}
    For $j=0,1,2$, we define a family of cutoff functions:
    \begin{equation}\label{Lemma higher order est in the middle: cutoff func chi_j}
       \chi_j(\zeta) =\begin{cases}
           1,&\quad \frac{1}{2^{2-j}}\zeta_1\leq\zeta\leq 2^{2-j}\zeta_2,\\
           0,&\quad \zeta\in[0,\frac{1}{2^{3-j}}\zeta_1]\cup[2^{3-j}\zeta_2,+\infty).
       \end{cases}
    \end{equation}
    For brevity, we denote $C(K_i):=C(\zeta_1,\zeta_2,K_i:1\leq i\leq 7)$.\\
    \underline{$L^2$-evolution}: Compute that
    \begin{align*}
        \frac{1}{2}\frac{d}{d\tau}\|\eps\chi_0\|^2_{L^2} = \scl{\chi_0\dr{\Delta \eps +\Gc\cdot\nabla\eps+\Fc\eps-\nabla W\cdot\nabla\Psi_\eps-\nabla\cdot(\eps\nabla\Psi_\eps)+E}}{\chi_0\eps}.
    \end{align*}
    First, by Cauchy's inequality,
    \begin{equation*}
        \scl{\chi_0\Delta\eps}{\chi_0\eps} = -\|\chi_0\nabla\eps\|^2_{L^2}-\scl{\nabla\eps}{2\chi_0\eps\nabla\chi_0}\leq -\frac{1}{2}\|\chi_0\nabla\eps\|^2_{L^2}+C(\zeta_1,\zeta_2)\frac{K_4^2\nu^4}{|\log\nu|^2}.
    \end{equation*}
    Next, due to the pointwise estimates $|\Fc|+|\Gc|\lesssim 1$, we obtain
    \begin{equation*}
        \scl{\eps\chi_0}{\chi_0\dr{\Fc\eps+\Gc\cdot\nabla\eps}} \leq \frac{1}{8}\|\chi_0\nabla\eps\|^2_{L^2}+C(\zeta_1,\zeta_2)\frac{K^2_4\nu^4}{|\log\nu|^2}.
    \end{equation*}
    By Hardy-Littlewood-Sobolev (HLS) inequality, \eqref{appendix: pointwise est of Poisson field by weighted L^2 norm} and \eqref{appendix: pointwise est of the gradient of Poisson field based on L infty of u}, we have the following estimates of Poisson field: 
    \begin{equation}\label{Lemma_middle range: ptwise est of whole Poisson field}
         \|\nabla\Psi_{\eps\chi^*}\|_{L^4}\lesssim \frac{1}{\nu^{\frac{3}{2}}}\|\eps\|_\inn\lesssim \frac{K_4\sqrt{\nu}}{|\log\nu|},\quad\|\nabla\Psi_{\eps(1-\chi^*)}\|_{L^\infty}\lesssim\|\eps(1+\zeta)^\frac{3}{2}\|_{L^\infty(\zeta\geq\zeta^*)}\lesssim K_7\nu^2.
    \end{equation}
    Thus, by \eqref{Lemma_middle range: ptwise est of whole Poisson field},
    \[
        -\scl{\eps\chi_0}{\chi_0\nabla W\cdot\nabla\Psi_\eps}\lesssim \frac{C(K_i)\nu^\frac{9}{2}}{|\log\nu|^2}.
    \]
    As for the nonlinear term, using \eqref{Lemma_middle range: ptwise est of whole Poisson field} and the Sobolev embedding $W^{1,p}(\reall^2)\hookrightarrow L^{\frac{2p}{2-p}}(\reall^2)\;(p<2)$, we have
    \[
       -\scl{\eps\chi_0}{\chi_0\nabla\cdot(\eps\nabla\Psi_\eps)}\lesssim\|\eps\|_{L^2(\frac{1}{8}\zeta_1\leq\zeta\leq8\zeta_2)}\|\nabla(\chi_0\eps)\|^2_{L^2}+\frac{C(\zeta_1,\zeta_2,K_i)}{|\log\nu|}\|\chi_0\nabla\eps\|^2_{L^2}+\frac{C(\zeta_1,\zeta_2,K_i)\nu^4}{|\log\nu|^3}.
    \]
    Finally, by the pointwise estimate of $E$, 
    \[
       \scl{\eps\chi_0}{E\chi_0}\leq \frac{C(\zeta_1,\zeta_2)K_4\nu^2}{|\log\nu|}|\mod_0|+\frac{C(\zeta_1,\zeta_2)\nu^4}{|\log\nu|^2}+\frac{C(\zeta_1,\zeta_2,K_i)\nu^4}{|\log\nu|^3}.
    \]
    In summary, we have
    \begin{equation}\label{Lemma_middle range: zeroth order est}
        \frac{d}{d\tau}\|\chi_0\eps\|^2_{L^2}\leq -\frac{1}{4}\|\chi_0\nabla\eps\|^2_{L^2}+C(\zeta_1,\zeta_2)\frac{K_4\nu^2}{|\log\nu|}|\mod_0|+\frac{C(\zeta_1,\zeta_2)\nu^4}{|\log\nu|^2}+\frac{C(\zeta_1,\zeta_2,K_i)\nu^4}{|\log\nu|^3}. 
    \end{equation}
    \underline{Evolution of first derivatives:} Denote by $\pa$ either $\pa_{\br}$ or $\pa_{\bz}$. Then,
    \begin{align*}
        \frac{1}{2}\frac{d}{d\tau}\|\pa\eps\chi_1\|^2_{L^2} = \scl{\chi_1\pa\dr{\Delta \eps +\Gc\cdot\nabla\eps+\Fc\eps-\nabla W\cdot\nabla\Psi_\eps-\nabla\cdot(\eps\nabla\Psi_\eps)+E}}{\chi_1\pa\eps}.
    \end{align*}
    Similarly,
    \[
        \scl{\chi_1\pa\Delta\eps}{\pa\eps\chi_1} = -\|\chi_1\nabla\pa\eps\|^2_{L^2}-\scl{\nabla\pa\eps}{2\chi_1\nabla\chi_1\pa\eps}\leq -\frac{1}{2}\|\chi_1\nabla\pa\eps\|^2_{L^2}+C\|\pa\eps\chi_0\|^2_{L^2},
    \]
    and
    \[
        \scl{\pa\eps\chi_1}{\chi_1\pa(\Gc\cdot\nabla\eps+\Fc\eps)}\leq\frac{1}{16} \|\chi_1\nabla\pa\eps\|^2_{L^2}+C(\zeta_1,\zeta_2)\|\pa\eps\chi_0\|^2_{L^2}+C(\zeta_1,\zeta_2)\frac{K_4^2\nu^4}{|\log\nu|^2}.
    \]
    Next, through integration by parts and estimates of the Poisson field,
    \begin{align*}
        -\scl{\chi_1\pa\eps}{\chi_1\pa(\nabla W\cdot\nabla\Psi_\eps)} &= \scl{\chi_1^2\pa^2\eps}{\nabla W\cdot\nabla\Psi_\eps}+\scl{2\pa\chi_1\chi_1\pa\eps}{\nabla W\cdot\nabla\Psi_\eps}\\
        &\leq \frac{1}{8}\|\chi_1\pa^2\eps\|^2_{L^2}+C\|\pa\eps\chi_0\|^2_{L^2}+\frac{C(\zeta_1,\zeta_2,K_i)\nu^5}{|\log\nu|^2}.
    \end{align*}
    As for the nonlinear term, by the Sobolev embedding $H^2(\reall^2)\hookrightarrow L^\infty(\reall^2)$, we have
    \[
       \|\eps\chi_1\|_{L^\infty}\lesssim \|\eps\chi_0\|_{L^2}+\|\chi_0\nabla\eps\|_{L^2}+\|\chi_1\nabla^{(2)}\eps\|_{L^2}.
    \]
    Then, through integration by parts, the $L^\infty$ estimate above, Sobolev embedding and \eqref{Lemma_middle range: ptwise est of whole Poisson field}, we obtain
    \begin{align*}
        -\scl{\chi_1\pa\eps}{\chi_1\pa\dr{\nabla\cdot(\eps\nabla\Psi_\eps)}} &= \scl{\chi^2_1\pa^2\eps}{\nabla\eps\cdot\nabla\Psi_\eps-\eps^2}+\scl{2\pa\chi_1\chi_1\pa\eps}{\nabla\eps\cdot\nabla\Psi_\eps-\eps^2}\\
        &\leq \frac{C(\zeta_1,\zeta_2,K_i)}{|\log\nu|}\|\pa^2\eps\chi_1\|^2_{L^2}+\frac{C(\zeta_1,\zeta_2,K_i)}{|\log\nu|}\|\pa\eps\chi_0\|^2_{L^2}+\frac{C(\zeta_1,\zeta_2)K_4^2\nu^4}{|\log\nu|^2}.
    \end{align*}
    Finally, through the integration by parts,
    \begin{equation*}
        \scl{\chi_1\pa\eps}{\pa E\chi_1}=-\scl{\eps}{\pa(\chi^2_1\pa E)}\lesssim C(\zeta_1,\zeta_2)\frac{K_4\nu^2}{|\log\nu|}|\mod_0|+\frac{C(\zeta_1,\zeta_2)K_4\nu^4}{|\log\nu|^2}+\frac{C(\zeta_1,\zeta_2,K_i)\nu^4}{|\log\nu|^3}. 
    \end{equation*}
    In summary, when $\nu$ is sufficiently small, we have
     \begin{align}\label{Lemma_middle range: first order est}
        \frac{d}{d\tau}\|\chi_1\nabla\eps\|^2_{L^2}&\leq -\frac{1}{4}\|\chi_1\nabla^{(2)}\eps\|^2_{L^2}+C(\zeta_1,\zeta_2)\|\nabla\eps\chi_0\|^2_{L^2}+\frac{C(\zeta_1,\zeta_2)K_2\nu^2}{|\log\nu|}|\mod_0|\nonumber\\
        &\quad+\frac{C(\zeta_1,\zeta_2)K_4^2\nu^4}{|\log\nu|^2}+\frac{C(\zeta_1,\zeta_2,K_i)\nu^4}{|\log\nu|^3}. 
    \end{align}
    \underline{Evolution of second derivatives:} Denote generically $\pa^{2}$ a second order partial derivative (e.g. $\pa_{\br}\pa_{\br},\pa_{\br}\pa_{\bz}$). Then,
    \begin{align*}
        \frac{1}{2}\frac{d}{d\tau}\|\pa^2\eps\chi_2\|^2_{L^2} = \scl{\chi_2\pa^2\dr{\Delta \eps +\Gc\cdot\nabla\eps+\Fc\eps-\nabla W\cdot\nabla\Psi_\eps-\nabla\cdot(\eps\nabla\Psi_\eps)+E}}{\chi_2\pa^2\eps}.
    \end{align*}
    The estimates of the first three terms are identical:
    \[
        \scl{\chi_2\pa^2\Delta\eps}{\pa^2\eps\chi_2} = -\|\chi_2\nabla\pa^2\eps\|^2_{L^2}-\scl{\nabla\pa^2\eps}{2\chi_2\nabla\chi_2\pa^2\eps}\leq -\frac{1}{2}\|\chi_2\nabla\pa^2\eps\|^2_{L^2}+C\|\pa^2\eps\chi_1\|^2_{L^2},
    \]
    and
    \[
        \scl{\pa^2\eps\chi_2}{\chi_2\pa^2(\Gc\cdot\nabla\eps+\Fc\eps)}\leq\frac{1}{16} \|\chi_2\nabla\pa^2\eps\|^2_{L^2}+C(\zeta_1,\zeta_2)\|\pa^2\eps\chi_1\|^2_{L^2}+C(\zeta_1,\zeta_2)\|\pa\eps\chi_0\|^2_{L^2}+\frac{C(\zeta_1,\zeta_2)K_4^2\nu^4}{|\log\nu|^2}.
    \]
    Next, through integration by parts once,
    \[
        -\scl{\chi_2\pa^2\eps}{\chi_2\pa^2\dr{\nabla W\cdot\nabla\Psi_\eps}}=\scl{\chi^2_2\pa^3\eps+2\chi_2\pa\chi_2\pa^2\eps}{\nabla(\pa W)\cdot\nabla\Psi_\eps+\nabla W\cdot\nabla(\pa\Psi_\eps)}.
    \]
    The estimates are the same as before, except for the $\nabla W\cdot\nabla(\pa\Psi_\eps)$ term, which is done in the following way: by elliptic regularity, HLS inequality, \eqref{appendix: pointwise est of Poisson field radial} and \eqref{appendix: pointwise est of Poisson field nonradial},
    \begin{align}\label{Lemma_middle range est: second derivative est of Poisson field}
       \int_{\{\frac{1}{2}\zeta_*\leq\zeta\leq2\zeta^*\}} |\nabla^{(2)}\Psi_\eps|^2&\leq C(\zeta_1,\zeta_2) \int_{\{\frac{1}{4}\zeta_1\leq\zeta\leq4\zeta_2\}} \eps^2+C(\zeta_1,\zeta_2)\int_{\{\frac{1}{4}\zeta_1\leq\zeta\leq4\zeta_2\}}|\nabla\Psi_\eps|^2\\
       &\leq\frac{C(\zeta_1,\zeta_2)}{\nu^3}\|\eps\|^2_\inn+C(\zeta_1,\zeta_2)\|\eps(1+\zeta)^\frac{3}{2}\|^2_{L^\infty(\zeta\geq\zeta^*)}.
    \end{align}
    Thus, we have
    \begin{align*}
         -\scl{\chi_2\pa^2\eps}{\chi_2\pa^2\dr{\nabla W\cdot\nabla\Psi_\eps}}&\leq \frac{1}{8}\|\chi_2\pa^3\eps\|^2_{L^2}+C(\zeta_1,\zeta_2)\|\chi_1\nabla^{(2)}\eps\|^2_{L^2}+C(\zeta_1,\zeta_2)\|\chi_0\nabla\eps\|^2_{L^2}\\
         &\quad+\frac{C(\zeta_1,\zeta_2,K_i)\nu^5}{|\log\nu|^2}+\frac{C(\zeta_1,\zeta_2)|a|^2}{\nu^3}\frac{K_4^2\nu^4}{|\log\nu|^2}.
    \end{align*}
    As for the nonlinear terms, integrate by parts once:
    \[
        -\scl{\chi_2\pa^2\eps}{\chi_2\pa^2\dr{\nabla\cdot(\eps\nabla\Psi_\eps)}} = \scl{\chi^2_2\pa^3\eps+2\chi_2\pa\chi_2\pa^2\eps}{\nabla(\pa\eps)\cdot\nabla\Psi_\eps+\nabla\eps\cdot\nabla(\pa\Psi_\eps)-2\eps\pa\eps}.
    \]
    Note that all the local terms (i.e., terms not involving the Poisson field) together with the term $\nabla(\pa\eps)\cdot\nabla\Psi_\eps$ can be estimated in the same way as before. It then remains to deal with the term $\nabla\eps\cdot\nabla(\pa\Psi_\eps)$. By the Sobolev embedding $H^2(\reall^2)\hookrightarrow L^\infty(\reall^2)$,
    \[
        \|\nabla\eps\chi_2\|_{L^\infty} \lesssim \|\nabla^{(3)}\eps\chi_2\|_{L^2}+\|\nabla^{(2)}\eps\chi_1\|_{L^2}+\|\nabla\eps\chi_0\|_{L^2}.
    \]
    This, together with \eqref{Lemma_middle range est: second derivative est of Poisson field}, gives
    \[
        \|\nabla\eps\cdot\nabla(\pa\Psi_\eps)\|_{L^2}\lesssim \frac{C(\zeta_1,\zeta_2,K_i)}{|\log\nu|}\dr{\|\nabla^{(3)}\eps\chi_2\|_{L^2}+\|\nabla^{(2)}\eps\chi_1\|_{L^2}+\|\nabla\eps\chi_0\|_{L^2}}.
    \]
    Therefore, we obtain the nonlinear estimate
    \begin{align*}
        -\scl{\chi_2\pa^2\eps}{\chi_2\pa^2\dr{\nabla\cdot(\eps\nabla\Psi_\eps)}} &\leq \frac{C(\zeta_1,\zeta_2,K_i)}{|\log\nu|}\|\pa^3\eps\chi_2\|^2_{L^2}+C(\zeta_1,\zeta_2)\dr{\|\nabla^{(2)}\eps\chi_1\|^2_{L^2}+\|\nabla\eps\chi_0\|^2_{L^2}}+\frac{CK^2_4\nu^4}{|\log\nu|^2}
    \end{align*}
    The estimate for $E$ is the same:
    \begin{equation*}
        \scl{\chi_2\pa^2\eps}{\pa^2E\chi_2}=\scl{\eps}{\pa^2(\chi_2^2\pa^2 E)}\leq \frac{C(\zeta_1,\zeta_2)K_4\nu^2}{|\log\nu|}|\mod_0|+\frac{C(\zeta_1,\zeta_2)K_4\nu^4}{|\log\nu|^2}+\frac{C(\zeta_1,\zeta_2,K_i)\nu^4}{|\log\nu|^3} 
    \end{equation*}
    In summary, when $\nu$ is sufficiently small we have
   \begin{align}\label{Lemma_middle range: second order est}
        \frac{d}{d\tau}\|\chi_2\nabla^{(2)}\eps\|^2_{L^2}&\leq -\frac{1}{4}\|\chi_2\nabla^{(3)}\eps\|^2_{L^2}+C(\zeta_1,\zeta_2)\dr{\|\nabla^{(2)}\eps\chi_1\|^2_{L^2}+\|\nabla\eps\chi_0\|^2_{L^2}}\nonumber\\
        &\quad +\frac{C(\zeta_1,\zeta_2)K_4\nu^2}{|\log\nu|}|\mod_0|+\frac{C(\zeta_1,\zeta_2)K^2_4\nu^4}{|\log\nu|^2}+\frac{C(\zeta_1,\zeta_2,K_i)\nu^4}{|\log\nu|^3}.
    \end{align}
    \underline{Conclusion:} Combining \eqref{Lemma_middle range: zeroth order est} \eqref{Lemma_middle range: first order est} and \eqref{Lemma_middle range: second order est}, we know that there exists $C_0=C_0(\zeta_1,\zeta_2)>0$ sufficiently large, such that once we define
    \begin{align}\label{lemma higher order est in the middle: def of H^2 norm}
       &\|\eps\|^2_{H^2_*(\zeta_1,\zeta_2)}:= \|\eps\chi_0\|^2_{L^2}+\frac{1}{C_0}\|\nabla\eps\chi_1\|^2_{L^2}+\frac{1}{C^2_0}\|\nabla^{(2)}\eps\chi_2\|^2_{L^2},\\
       &\|\eps\|^2_{ H^3_*(\zeta_1,\zeta_2)}:= \|\nabla\eps\chi_0\|^2_{L^2}+\frac{1}{C_0}\|\nabla^{(2)}\eps\chi_1\|^2_{L^2}+\frac{1}{C^2_0}\|\nabla^{(3)}\eps\chi_2\|^2_{L^2},\nonumber
    \end{align}
    it holds that
    \begin{align*}
        \frac{d}{d\tau}\|\eps\|^2_{H^2_*(\zeta_1,\zeta_2)}\leq -\frac{1}{8}\|\eps\|^2_{H^3_*(\zeta_1,\zeta_2)} +\frac{C(\zeta_1,\zeta_2)K_4\nu^2}{|\log\nu|}|\mod_0|+ \frac{C(\zeta_1,\zeta_2)K_4^2\nu^4}{|\log\nu|^2}+\frac{C(\zeta_1,\zeta_2,K_i)\nu^4}{|\log\nu|^3}.
    \end{align*}
    Finally, the result follows from the Poincar\'e inequality
    \[
        \|\eps\|_{H^2_*(\zeta_1,\zeta_2)}\leq C'\|\eps\|_{ H^3_*(\zeta_1,\zeta_2)}+\frac{C(\zeta_1,\zeta_2)K_4\nu^2}{|\log\nu|}.
    \]
\end{proof}

\subsubsection{Far Field Estimate}
\begin{lemma}[$L^\infty$ control of $\eps$ in the far field]\label{lemma: L infty control of eps}
    There exists $M>0$, such that for any $M_0\geq M$ and $\zeta^*>M$, the following holds. Let $w$ be a solution in the bootstrap regime $\BS(\tau_0,\tau_*, \zeta^*, M_0, \{K_i\}_{i=1}^7)$. Then, for any $\tau\in [\tau_0,\tau_*]$ we have the following estimate
    \begin{align}\label{L infty control: specific result in Lemma}
        \|\eps(\tau)(1+\zeta)^{\frac{3}{2}}\|_{L^\infty( \zeta\geq \zeta^*)}&\leq C(\zeta^*,K_1)\frac{\nu^3(\tau)}{\nu^3(\tau_0)}\dr{\|\eps(\tau_0)(1+\zeta(\tau_0))^{\frac{3}{2}}\|_{L^\infty(\zeta\geq\zeta^*)}+\|\eps(\tau_0)\|_{L^\infty(\zeta\geq\frac{1}{2}\zeta^*)}}\nonumber\\
        &\quad+\frac{CK_7}{\zeta^*}\frac{e^{-2\sqrt{\beta\tau+M_0}}}{\sqrt{\beta\tau+M_0}}+\frac{C(\zeta^*,K_1,K_4,K_5,K_6)e^{-2\sqrt{\beta\tau+M_0}}}{\sqrt{\beta\tau+M_0}}+\frac{C(\zeta^*,K_i)e^{-2\sqrt{\beta\tau+M_0}}}{\beta\tau+M_0}.
    \end{align}
\end{lemma}
\begin{proof}
     To derive the far field $L^\infty$-control, we go all the way back to the original $3$D system \eqref{Introduction:3dks}, where we exploit the dissipating structure of the heat operator. Consider the following decomposition of the solution:
    \begin{equation}\label{Lemma_L infty control: decomp of solution}
        u(\x,t) = \frac{1}{\mu^2}W(\frac{r-R}{\mu},\frac{z}{\mu},\tau)+\frac{1}{\mu^2}\eps(\frac{r-R}{\mu},\frac{z}{\mu},\tau),
    \end{equation}
    where we recall $r = |(x_1,x_2)|$, $z = x_3$, $\x = (x_1,x_2,x_3)$, $\tau = -\log(T-t)+\log(T)$, and $\mu = \sqrt{T-t}=e^{-\tau/2+\log(T)/2}$ (without loss of generality, we can assume $T=1$ in the following). In addition, recall that
    \[
        W(\br,\bz,\tau) = U_\nu(\br,\bz)+P(\br,\bz,\tau),
    \]
    where we denote $\br: = (r-R)/\mu$ and $\bz:= z/\mu$ as before. One remark: since $R,\mu$ depend on time, so do $\br,\bz$, and we will specify their time dependence whenever necessary. Let $\eta(\br,\bz)$ be a smooth cutoff function such that
    \[
        \eta(\br,\bz)\equiv 1\text{ when }\sqrt{\br^2+\bz^2}\geq \zeta^*,\quad \eta(\br,\bz)\equiv 0\text{ when }\sqrt{\br^2+\bz^2}< \frac{1}{2}\zeta^*.
    \]
    Denote $\eps_*(\br,\bz,\tau): = \eps(\br,\bz,\tau)\eta(\br,\bz)$. Then, the evolution equation for $\eps_*$ can be written as
    \begin{equation}\label{L infty esti: evolution of eps_*}
        \pa_\tau\eps_* = \dr{\pa^2_{\br}+\pa^2_{\bz}+\frac{1}{\br+R/\mu}\pa_{\br}}\eps_*-\frac{1}{2}\Lambda \eps_*+\frac{R_\tau}{\mu}\pa_{\br}\eps_*+S(\br,\bz,\tau),
    \end{equation}
    where
    \begin{align}\label{L infty esti: def of S}
        S(\br,\bz,\tau) &= \eta\Big(\Delta W-\nabla\cdot(\eps\nabla\Phi_W+W\nabla\Phi_\eps+\eps\nabla\Phi_\eps+W\nabla\Phi_W)-\frac{1}{\br+R/\mu}(\eps\pa_{\br}\Phi_W+W\pa_{\br}\Phi_\eps+\eps\pa_{\br}\Phi_\eps\nonumber\\
        &\quad+W\pa_{\br}\Phi_W)
        +\frac{1}{\br+R/\mu}\pa_{\br}W
        -\frac{1}{2}\Lambda W +\frac{R_\tau}{\mu}\pa_{\br}W-\pa_\tau W \Big)-\Delta\eta \eps-2\nabla\eps\cdot\nabla\eta-\frac{1}{\br+R/\mu}\pa_{\br}\eta\eps\nonumber\\
        &\quad+\frac{1}{2}\dr{(\br,\bz)\cdot \nabla\eta}\eps-\frac{R_\tau}{\mu}\pa_{\br}\eta\eps\nonumber\\
        &=\eta\Big(-\nabla\cdot(\eps\nabla\Phi_W+W\nabla\Phi_\eps+\eps\nabla\Phi_\eps)-\frac{1}{\br+R/\mu}(\eps\pa_{\br}\Phi_W+W\pa_{\br}\Phi_\eps+\eps\pa_{\br}\Phi_\eps)+\mod_0\varphi_{0,\nu}\nonumber\\
        &\quad+\mod_1\varphi_{1,\nu}+\tilde E \Big)-\Delta\eta \eps-2\nabla\eps\cdot\nabla\eta-\frac{1}{\br+R/\mu}\pa_{\br}\eta\eps
        +\frac{1}{2}\dr{(\br,\bz)\cdot \nabla\eta}\eps-\frac{R_\tau}{\mu}\pa_{\br}\eta\eps
    \end{align}
    Now, due to the parabolic scaling, it is natural to relate \eqref{L infty esti: evolution of eps_*} to the standard heat equation. Denote
    \begin{equation}\label{L infty est: relation between eps and bu}
         \bu(\x,t) := \frac{1}{\mu^2(\tau)}\eps_*(\frac{r-R(\tau)}{\mu(\tau)},\frac{z}{\mu(\tau)},\tau),\quad \x\in\reall^3. 
    \end{equation}
    Then, $\bu$ solves the following heat equation with an axisymmetric force:
    \begin{equation}\label{Lemma L_infty: heat eq for bar u}
        \pa_t \bu(\x,t) = \Delta^{(3)}\bu(\x,t)+\frac{1}{\mu^4}S(\frac{r-R}{\mu},\frac{z}{\mu},-\log(T-t)).
    \end{equation}
    Now, we consider the evolution of \eqref{L infty esti: evolution of eps_*} in the time interval $[\tau_0,\tau]$, or equivalently, that of \eqref{Lemma L_infty: heat eq for bar u} in $[t_0,t]$. Given the initial data $\bu(\x,t_0) = \bu_0(\x)$, the solution of \eqref{Lemma L_infty: heat eq for bar u} can be expressed by the convolution of the $3$D heat kernel:
\begin{align*}
    &\bu(\x,t) = \int_{\reall^3}H(\x-\tilde\x,t-t_0)\bu_0(\tilde \x)\;d\x\\
    &\,+ \int_{0}^{t-t_0}\int_{\reall^3}H(\x-\tilde\x,t-t_0-s)\frac{1}{\mu^4(s+t_0)}S\dr{\frac{|(\tilde x_1,\tilde x_2)|-R(s+t_0)}{\mu(s+t_0)},\frac{\tilde x_3}{\mu(s+t_0)},-\log(T-t_0-s)}\;d\tilde\x ds,
\end{align*}
where
\[
    H(\x,t) = \frac{1}{(4\pi t)^{\frac{3}{2}}}e^{-\frac{|\x|^2}{4t}}.
\]
Thus, we obtain an explicit expression for $\eps_*$ through the relation \eqref{L infty est: relation between eps and bu} ($\eps_{*,0}(\br,\bz):=\eps_*(\br,\bz,\tau_0)$):
\begin{align*}
    &\eps_*(\br,\bz,\tau)=\\
    &\frac{\mu^2(\tau)}{\mu^2(\tau_0)}\int_{-\infty}^{+\infty}\int_{-\frac{R(\tau_0)}{\mu(\tau_0)}}^{+\infty}\int_{0}^{2\pi}\frac{\mu^3(\tau_0)(\tilde p+\frac{R(\tau_0)}{\mu(\tau_0)})}{(4\pi(e^{-\tau_0}-e^{-\tau}))^{\frac{3}{2}}}
    \exp\bigg(-\frac{\mu^2(\tau_0)}{4(e^{-\tau_0}-e^{-\tau})}\Big[\Big(\frac{\mu(\tau)}{\mu(\tau_0)}\big(\br+\frac{R(\tau)}{\mu(\tau)}\big)\\
    &\quad-\big(\tilde p+\frac{R(\tau_0)}{\mu(\tau_0)}\big)\cos(\theta)\Big)^2
    +\big(\tilde p+\frac{R(\tau_0)}{\mu(\tau_0)}\big)^2\sin^2(\theta)+\big(\frac{\mu(\tau)}{\mu(\tau_0)}\bz - \tilde q\big)^2\Big]\bigg)\eps_{*,0}(\tilde p,\tilde q)\, d\theta d\tilde p d\tilde q\\
    &+\int_{\tau_0}^{\tau}\frac{\mu^2(\tau)}{\mu^2(\tilde \tau)}\,d\tilde\tau\int_{-\infty}^{+\infty}\int_{-\frac{R(\tilde\tau)}{\mu(\tilde\tau)}}^{+\infty}\int_{0}^{2\pi}\frac{\mu^3(\tilde\tau)(\tilde p+\frac{R(\tilde\tau)}{\mu(\tilde\tau)})}{(4\pi(e^{-\tilde\tau}-e^{-\tau}))^{\frac{3}{2}}}
    \exp\bigg(-\frac{\mu^2(\tilde\tau)}{4(e^{-\tilde\tau}-e^{-\tau})}\Big[\Big(\frac{\mu(\tau)}{\mu(\tilde\tau)}\big(\br+\frac{R(\tau)}{\mu(\tau)}\big)\\
    &\quad-\big(\tilde p+\frac{R(\tilde\tau)}{\mu(\tilde\tau)}\big)\cos(\theta)\Big)^2
    +\big(\tilde p+\frac{R(\tilde\tau)}{\mu(\tilde\tau)}\big)^2\sin^2(\theta)+\big(\frac{\mu(\tau)}{\mu(\tilde\tau)}\bz - \tilde q\big)^2\Big]\bigg)S(\tilde p,\tilde q,\tilde\tau)\, d\theta d\tilde p d\tilde q\\
    &=: \frac{\mu^2(\tau)}{\mu^2(\tau_0)}\int_{\reall^2}\hH(\br,\bz,\tilde p,\tilde q,\tau,\tau_0)\eps_{*,0}(\tilde p,\tilde q)\,d\tilde p d\tilde q+\int_{\tau_0}^{\tau}\frac{\mu^2(\tau)}{\mu^2(\tilde\tau)}\,d\tilde\tau\int_{\reall^2}\hH(\br,\bz,\tilde p,\tilde q,\tau,\tilde\tau)S(\tilde p,\tilde q,\tilde \tau)\,d\tilde p d\tilde q\\
    & =: I_1(\br,\bz,\tau)+I_2(\br,\bz,\tau). 
\end{align*}
The expression above is nothing but the convolution with the heat kernel written in cylindrical coordinates, i.e., $\hH$. In the following estimates, we make use of the two key properties of the heat kernel: total mass $1$ (in $\reall^3$) and exponential decay. As before, denote the time-dependent variable $\zeta:= \sqrt{\br^2+\bz^2}$. Observe that
\[
    0<\sigma(\tau,\tau_0):=\frac{e^{-\tau_0}-e^{-\tau}}{\mu^2(\tau_0)} = 1-e^{\tau_0-\tau}< 1.
\]
First, when $\zeta(\tau)\leq 2\zeta^* \frac{\mu(\tau_0)}{\mu(\tau)}$, since the heat kernel has total mass $1$, we have
\[
    |I_1(\br,\bz,\tau)|\leq \frac{\mu^2(\tau)}{\mu^2(\tau_0)}\|\eps_{*,0}\|_{L^\infty}.
\]
Thus, in this domain we have
\begin{equation}\label{L infty esti: near region I_1 esti}
    \|I_1(\br,\bz,\tau)(1+\zeta)^\frac{3}{2}\|_{L^\infty(\zeta^*\leq\zeta\leq 2\zeta^*\mu(\tau_0)/\mu(\tau))}\leq 4{\zeta^*}^{\frac{3}{2}}\frac{\mu^{\frac{1}{2}}(\tau)}{\mu^{\frac{1}{2}}(\tau_0)}\|\eps_{*,0}\|_{L^\infty}.
\end{equation}
Second, when $\zeta(\tau)> 2\zeta^* \frac{\mu(\tau_0)}{\mu(\tau)}$, i.e., $\zeta(\tau)\frac{\mu(\tau)}{\mu(\tau_0)}\geq 2\zeta^*$, denoting $B(\delta)\subset\reall^2$ to be the ball centered at $(\frac{\mu(\tau)}{\mu(\tau_0)}\br,\frac{\mu(\tau)}{\mu(\tau_0)}\bz)$ with radius $\delta$ (to be determined), we split the integral into two parts:
\begin{align*}
    I_1(\br,\bz,\tau) &= \frac{\mu^2(\tau)}{\mu^2(\tau_0)}\int_{B(\delta)}\hH(\br,\bz,\tilde p,\tilde q,\tau,\tau_0)\eps_{*,0}(\tilde p,\tilde q)\,d\tilde p d\tilde q+\frac{\mu^2(\tau)}{\mu^2(\tau_0)}\int_{\reall^2\backslash B(\delta)}\hH(\br,\bz,\tilde p,\tilde q,\tau,\tau_0)\eps_{*,0}(\tilde p,\tilde q)\,d\tilde p d\tilde q\\
    &=: J_1(\br,\bz,\tau)+J_2(\br,\bz,\tau).
\end{align*}
By the decay property of $\eps_*$, we have the estimate
\begin{equation}\label{L infty esti: est of J_1}
    |J_1(\br,\bz,\tau)|\leq  \frac{\mu^2(\tau)}{\mu^2(\tau_0)}\frac{\|\eps_*(\tau_0)(1+\zeta(\tau_0))^{\frac{3}{2}}\|_{L^\infty(\zeta\geq\zeta^*)}}{(1+\frac{\mu(\tau)}{\mu(\tau_0)}\zeta(\tau)-\delta)^\frac{3}{2}}.
\end{equation}
Meanwhile, by the exponential decay of $\hH$, we have
\begin{equation}\label{L infty esti: est of J_2}
     |J_2(\br,\bz,\tau)|\lesssim  \frac{\mu^2(\tau)}{\mu^2(\tau_0)}\frac{\delta}{\sqrt{\sigma(\tau,\tau_0)}}e^{-\frac{\delta^2}{\sigma(\tau,\tau_0)}}\|\eps_*(\tau_0)\|_{L^\infty}.
\end{equation}
Therefore, when $\zeta^*$ is large enough, choosing $\delta: = \sqrt{\zeta(\tau)\frac{\mu(\tau)}{\mu(\tau_0)}}$ and combining \eqref{L infty esti: est of J_1}\eqref{L infty esti: est of J_2}, we have
\begin{equation}\label{L infty esti: far region I_1 est}
    |I_1(\br,\bz,\tau)|\lesssim \frac{1}{(1+\zeta(\tau))^{\frac{3}{2}}}\frac{\mu^{\frac{1}{2}}(\tau)}{\mu^{\frac{1}{2}}(\tau_0)}\dr{\|\eps(\tau_0)(1+\zeta(\tau_0))^{\frac{3}{2}}\|_{L^\infty(\zeta\geq\zeta^*)}+\|\eps_*(\tau_0)\|_{L^\infty}}.
\end{equation}
Also, we note that for large $M_0>0$, by the bootstrap assumption it holds that
\begin{align}\label{L infty esti: key scaling of mu}
    \frac{\mu(\tau)}{\mu(\tau_0)} = e^{-\frac{1}{2}(\tau-\tau_0)} &= \dr{e^{-(\sqrt{\tau/2+M_0}-\sqrt{\tau_0/2+M_0})}}^{\sqrt{\tau/2+M_0}+\sqrt{\tau_0/2+M_0}}\nonumber\\
    &< \dr{e^{-(\sqrt{\tau/2+M_0}-\sqrt{\tau_0/2+M_0})}}^{100} \leq C(K_1) \frac{\nu^{100}(\tau)}{\nu^{100}(\tau_0)}. 
\end{align}
In summary, by \eqref{L infty esti: near region I_1 esti}\eqref{L infty esti: far region I_1 est}, we have when $\zeta^*$ is sufficiently large:
\begin{equation}\label{L infty esti: total I_1 esti}
    \|I_1(\br,\bz,\tau)(1+\zeta(\tau))^\frac{3}{2}\|_{L^\infty(\zeta\geq\zeta^*)}\leq C(K_1)\frac{\nu^3(\tau)}{\nu^3(\tau_0)}\dr{\|\eps(\tau_0)(1+\zeta(\tau_0))^{\frac{3}{2}}\|_{L^\infty(\zeta\geq\zeta^*)}+\|\eps_*(\tau_0)\|_{L^\infty}},
\end{equation}
where $C$ is some universal constant. This completes the estimate for $I_1$.

Next, according to \eqref{L infty esti: def of S}, we can decompose the source term as
\[
    S = \eta\nabla\cdot S_1+\frac{1}{\br+R/\mu}S_2+S_3-2\nabla\cdot(\eps\nabla\eta),
\]
where (plugging in the definition of $\tilde E$)
\begin{align*}
    &S_1 = -\eps\nabla\Phi_W-W\nabla\Phi_\eps-\eps\nabla\Phi_\eps-W\nabla\Theta_W-P\nabla\Psi_P,\\
    &S_2 = -\eta(\eps\pa_{\br}\Phi_W+W\pa_{\br}\Phi_\eps+\eps\pa_{\br}\Phi_\eps)-\eta(1-\chi(\zeta\nu))(\pa_{\br}W+W\pa_{\br}\Phi_W)-\pa_{\br}\eta\eps,\\
    &S_3 = \eta\mod_0\vp_{0,\nu}+\eta\mod_1\vp_{1,\nu}-\eta a\nu_\tau\nu\pa_\nu(\vp_{1,\nu}-\vp_{0,\nu})+\eta\dr{\frac{\nu_\tau}{\nu}-\beta}(\Lambda U_\nu+16\nu^2\vp_{0,\nu})\\
    &\quad+\eta\frac{R_\tau}{\mu}\pa_{\br}W+\eta a(R_1-R_0)+\Delta\eta\eps+\frac{1}{2}((\br,\bz)\cdot\nabla\eta)\eps-\frac{R_\tau}{\mu}\pa_{\br}\eta\eps+\frac{\eta\chi(\zeta\nu)}{\br+R/\mu}(\pa_{\br}W+W\pa_{\br}\Phi_W).
\end{align*}
By the bootstrap assumption, the pointwise estimates in Proposition \ref{proposition 1}, the Poisson field estimates 
\eqref{appendix: 3&2D near field approx}\eqref{appendix: 3&2D far field control}\\\eqref{appendix: pointwise est of Poisson field radial and nonradial, parabolic}\eqref{appendix: pointwise est of the gradient of Poisson field based on L infty of u}, and the $L^\infty$ control on the boundary $\|\eps\|_{L^\infty(\zeta^*\leq\zeta\leq 2\zeta^*)}\lesssim \|\eps\|_{H^2(\frac{1}{2}\zeta^*\leq \zeta\leq4\zeta^*)}\lesssim C(\zeta^*)\frac{K_6\nu^2}{|\log\nu|}$, we obtain the pointwise estimates
\begin{align}\label{L infty esti: pointwise est of S}
\begin{split}
    \|S_1(\br,\bz,\tau)(1+\zeta)^{\frac{3}{2}}\|&_{L^\infty(\zeta\geq \zeta^*)}\leq \frac{C}{\zeta^*}\|\eps(\tau)(1+\zeta)^\frac{3}{2}\|_{L^\infty(\zeta\geq\zeta^*)}+C(K_i)\nu(\tau)^3,\\
    \|S_2(\br,\bz,\tau)(1+\zeta)^{\frac{3}{2}}\|&_{L^\infty}\leq \frac{C}{\zeta^*}\|\eps(\tau)(1+\zeta)^\frac{3}{2}\|_{L^\infty(\zeta\geq\zeta^*)}+C(\zeta^*)\frac{K_6\nu^2}{|\log\nu|}+C(K_i)\nu(\tau)^3,\\
    \|S_3(\br,\bz,\tau)(1+\zeta)^{\frac{3}{2}}\|&_{L^\infty}\leq \frac{C|\mod_0|}{\zeta^*}+C(\zeta^*)\frac{K_6\nu^2}{|\log\nu|}+\frac{C(K_i)\nu(\tau)^2}{|\log\nu(\tau)|^2}.
\end{split}
\end{align}
Through integration by parts,
\begin{align*}
    I_2(\br,\bz,\tau) &= \int_{\tau_0}^{\tau}\frac{\mu^\frac{1}{4}(\tau)}{\mu^\frac{1}{4}(\tilde\tau)}\cdot\frac{\mu^\frac{7}{4}(\tau)}{\mu^\frac{7}{4}(\tilde\tau)}\,d\tilde\tau\int_{\reall^2}\hH(S_3-\nabla\eta\cdot S_1)\,d\tilde p d\tilde q\\
    &\quad+\int_{\tau_0}^{\tau}\frac{\mu^\frac{1}{4}(\tau)}{\sqrt{\sigma(\tau,\tilde\tau)}\mu^\frac{1}{4}(\tilde\tau)}\cdot\frac{\mu^\frac{7}{4}(\tau)}{\mu^\frac{7}{4}(\tilde\tau)}\,d\tilde\tau\int_{\reall^2}\frac{\sqrt{\sigma(\tau,\tilde\tau)}\hH}{\tilde p+R(\tilde\tau)/\mu(\tilde\tau)} S_2-\sqrt{\sigma(\tau,\tilde\tau)}\nabla \hH\cdot (\eta S_1-2\nabla\eta\eps)\,d\tilde p d\tilde q\\
   & =:  \int_{\tau_0}^{\tau}\frac{\mu^\frac{1}{4}(\tau)}{\mu^\frac{1}{4}(\tilde\tau)} I_{2,a}(\br,\bz,\tau,\tilde\tau)\,d\tilde\tau+\int_{\tau_0}^{\tau}\frac{\mu^\frac{1}{4}(\tau)}{\sqrt{\sigma(\tau,\tilde\tau)}\mu^\frac{1}{4}(\tilde\tau)}I_{2,b}(\br,\bz,\tau,\tilde\tau)\,d\tilde \tau.
\end{align*}
Observe that kernels $\frac{\sqrt{\sigma(\tau,\tilde\tau)}\hH}{\tilde p+R(\tilde\tau)/\mu(\tilde\tau)}$ and $\sqrt{\sigma(\tau,\tilde\tau)}\nabla \hH$ share similar properties with $\hH$: bounded total mass and exponential decay, which are all we need in deriving the estimate for $I_1$. Thus, by \eqref{L infty esti: key scaling of mu} and \eqref{L infty esti: pointwise est of S}, with a similar argument we can show that
\begin{align*}
    &\quad\|I_{2,a}(\br,\bz,\tau,\tilde \tau)(1+\zeta(\tau))^\frac{3}{2}\|_{L^\infty(\zeta\geq \zeta^*)}\\
    &\leq Ce^{4\sqrt{\beta\tilde\tau+M_0}-4\sqrt{\beta\tau+M_0}} \dr{\|S_3(\tilde\tau)(1+\zeta(\tilde\tau))^\frac{3}{2}\|_{L^\infty}+\|S_1(\tilde\tau)(1+\zeta(\tilde\tau))^\frac{3}{2}\|_{L^\infty(\zeta\geq\zeta^*)}}\\
    &\leq \frac{CK_7}{\zeta^*}\frac{e^{-2\sqrt{\beta\tau+M_0}}}{\sqrt{\beta\tau+M_0}}+\frac{C(\zeta^*,K_1,K_4,K_5,K_6)e^{-2\sqrt{\beta\tau+M_0}}}{\sqrt{\beta\tau+M_0}}+\frac{C(\zeta^*,K_i)e^{-2\sqrt{\beta\tau+M_0}}}{\beta\tau+M_0},
\end{align*}
and
\begin{align*}
    &\quad\|I_{2,b}(\br,\bz,\tau,\tilde \tau)(1+\zeta(\tau))^\frac{3}{2}\|_{L^\infty(\zeta\geq \zeta^*)}\\
    &\leq Ce^{4\sqrt{\beta\tilde\tau+M_0}-4\sqrt{\beta\tau+M_0}} \dr{\|S_2(\tilde\tau)(1+\zeta(\tilde\tau))^\frac{3}{2}\|_{L^\infty}+\|(\eta S_1(\tilde\tau)-2\nabla\eta\eps(\tilde \tau))(1+\zeta(\tilde\tau))^\frac{3}{2}\|_{L^\infty}}\\
    &\leq  \frac{CK_7}{\zeta^*}\frac{e^{-2\sqrt{\beta\tau+M_0}}}{\sqrt{\beta\tau+M_0}}+\frac{C(\zeta^*,K_1,K_6)e^{-2\sqrt{\beta\tau+M_0}}}{\sqrt{\beta\tau+M_0}}+C(K_i)\nu^3.
\end{align*}
Finally, it remains to estimate the time integrals:
\begin{align*}
    \int_{\tau_0}^{\tau}\frac{\mu^{\frac{1}{4}}(\tau)}{\mu^{\frac{1}{4}}(\tilde\tau)}\,d\tilde\tau = \int_{\tau_0}^{\tau}e^{-\frac{1}{8}(\tau-\tilde\tau)}\,d\tilde\tau < 8,
\end{align*}
and
\begin{align*}
    \int_{\tau_0}^{\tau}\frac{\mu^{\frac{1}{4}}(\tau)}{\sqrt{\sigma(\tau,\tilde\tau)}\mu^{\frac{1}{4}}(\tilde\tau)}\,d\tilde\tau = \int_{\tau_0}^{\tau}\frac{e^{-\frac{1}{8}(\tau-\tilde\tau)}}{\sqrt{1-e^{\tilde\tau-\tau}}}\,d\tilde\tau &= \int_{t_0}^{t}\frac{(T-t)^\frac{1}{8}}{(T-\tilde t)^\frac{5}{8}}\frac{1}{\sqrt{t-\tilde t}}\,d\tilde t\\
    &\leq \int_{t-1}^{t}\frac{(T-t)^\frac{1}{8}}{(T-\tilde t)^\frac{5}{8}}\frac{1}{\sqrt{t-\tilde t}}\,d\tilde t \leq 10+\Oc\dr{(T-t)^\frac{1}{8}}.
\end{align*}
 This completes the proof of the Lemma.   
\end{proof}

\section{Existence of Blowup Solutions}\label{sec: existence of blowup}
\newcommand{\ba}{\bar a}
\newcommand{\ta}{\tilde a}
\newcommand{\bn}{\bar\nu}
\newcommand{\tn}{\tilde \nu}

Now with the energy estimates and modulation equations, we are ready to prove the existence of blowup solutions. It suffices to show that there exist certain initial data $(\eps_0,\nu_0,a_0,R_0)$ and parameters $\zeta^*,K_1,\dots,K_7$, such that the evolution \eqref{evolution of epsilon} will be trapped in some bootstrap regime $\BS(K_i:1\leq i\leq 7)$ for all $\tau\in [0,+\infty)$. Roughly speaking, $\zeta^*$ is chosen first which depends on some of the universal constants in the estimates, then $K_i$ (the order of dependence among which will be specified shortly), and finally $\nu_0$ (or equivalently, $M_0$ in the bootstrap statement), so that $C(\zeta^*,K_i)/|\log\nu|$ will have the smallness we want.

The following lemma will help us close the bootstrap for $\eps$.
\begin{lemma}\label{lemma diff ineq}
    Let $f(\tau)\geq 0$ be a differentiable function in $\tau$. Let $\nu(\tau)$ be the parameter in the Bootstrap regime $\BS(\tau_0,\tau_*, \zeta^*, M_0, \{K_i\}_{i=1}^7)$. Suppose we have the following differential inequality
    \begin{equation*}
        f'(\tau)\leq -\delta f(\tau)+\frac{K\nu^k(\tau)}{|\log\nu(\tau)|^2},
    \end{equation*}
    for some constants $\delta,K>0$ and $k>1$. Then, there exists constant $C(\delta,k),\bar M>0$, such that for any $\log(\nu(0))^2=M_0>\bar M$, we have
    \begin{equation}\label{lemma diff ineq: ineq}
        |f(\tau)|\leq \frac{K\nu^k(\tau)}{\delta|\log\nu|^2}+\frac{KC(\delta,k,K_1,K_2,K_4,K_5)\nu^k}{|\log\nu|^3}+f(0)e^{-\delta\tau},
    \end{equation}
    holds for any $\tau\in [0,\tau_*]$. 
\end{lemma}
\begin{proof}
    The proof applies Gronwall's inequality and integration by parts. First, we have by Gronwall's inequality,  
    \begin{align}\label{lemma diff ineq: 1}
        f(\tau)\leq e^{-\delta \tau}f(0)+e^{-\delta\tau}\int_{0}^{\tau}\frac{K\nu^k(s)}{|\log\nu(s)|^2}e^{\delta s}\;ds.
    \end{align}
    Through integration by parts, we have
    \begin{equation}\label{lemma diff ineq: 2}
        \int_{0}^{\tau}\frac{K\nu^k(s)}{|\log\nu(s)|^2}e^{\delta s}\;ds = \frac{e^{\delta\tau}}{\delta}\frac{K\nu^k(\tau)}{|\log\nu(\tau)|^2}-\frac{K\nu^k(0)}{|\log\nu(0)|^2}-\frac{K}{\delta}\int_{0}^{\tau}\frac{\nu'(s)}{\nu(s)}\dr{\frac{k\nu^k(s)}{|\log\nu(s)|^2}-\frac{2\nu^k(s)}{(\log\nu(s))^3}}e^{\delta s}\;ds.
    \end{equation}
    Then, by the estimate of $|\frac{\nu'}{\nu}|$ in \eqref{cor of mod est: est of nu_tau} and the bootstrap assumption for $\nu$, we have
    \begin{align}\label{lemma diff ineq: 3}
        \da{\frac{K}{\delta}\int_{0}^{\tau}\frac{\nu'(s)}{\nu(s)}\dr{\frac{k\nu^k(s)}{|\log\nu(s)|^2}-\frac{2\nu^k(s)}{(\log\nu(s))^3}}e^{\delta s}\;ds} \leq \frac{KkC(K_1,K_2,K_4,K_5)}{\delta}\int_{0}^{\tau} \frac{e^{\delta s - k\sqrt{\beta s+M_0}}}{(\beta s+M_0)^\frac{3}{2}}\;ds
    \end{align}
    By change of variables ($x:=\sqrt{\beta s +M_0}$) and integration by parts,
    \begin{align*}
        \int_{0}^{\tau} \frac{e^{\delta s - k\sqrt{\beta s+M_0}}}{(\beta s+M_0)^\frac{3}{2}}\;ds &= \int_{\sqrt{M_0}}^{\sqrt{\beta\tau+M_0}}\frac{2}{\beta x^2}e^{\frac{\delta}{\beta}x^2-kx-\frac{\delta}{\beta}M_0};dx\\
        &=\left.\frac{e^{\frac{\delta}{\beta}x^2-kx-\frac{\delta}{\beta}M_0}}{\delta x^2(x-\frac{\beta k}{2\delta})}\,\right|_{x=\sqrt{M_0}}^{\sqrt{\beta\tau+M_0}}-\int_{\sqrt{M_0}}^{\sqrt{\beta\tau+M_0}}\dr{\frac{1}{\delta x^2(x-\frac{\beta k}{2\delta})}}'e^{\frac{\delta}{\beta}x^2-kx-\frac{\delta}{\beta}M_0}\;dx\\
        &\leq \frac{Ce^{\delta \tau-k\sqrt{\beta\tau+M_0}}}{\delta(\beta\tau+M_0))^\frac{3}{2}}+\frac{C}{M_0\delta}\int_{\sqrt{M_0}}^{\sqrt{\beta\tau+M_0}}\frac{2}{\beta x^2}e^{\frac{\delta}{\beta}x^2-kx-\frac{\delta}{\beta}M_0};dx
    \end{align*}
    Therefore, when $M_0$ is sufficiently large, we have 
    \begin{equation}\label{lemma diff ineq: 4}
        \int_{0}^{\tau} \frac{e^{\delta s - k\sqrt{\beta s+M_0}}}{(\beta s+M_0)^\frac{3}{2}}\;ds\leq\frac{Ce^{\delta \tau-k\sqrt{\beta\tau+M_0}}}{\delta(\beta\tau+M_0))^\frac{3}{2}}.
    \end{equation}
    Finally, inserting \eqref{lemma diff ineq: 4}, \eqref{lemma diff ineq: 3} and \eqref{lemma diff ineq: 2} back into \eqref{lemma diff ineq: 1}, we obtain the result.
\end{proof}

Now we are ready to prove the main proposition, which will conclude the proof of Theorem \ref{main theorem}.
\begin{proposition}\label{proposition: main}
    There exist a choice of parameters $(\zeta^*,M_0,\{K_i\}_{i=1}^{7})$ and initial data for $w$, such that the solution $w$ of \eqref{Parabolic System} will be trapped in the bootstrap regime $\BS(0,+\infty, \zeta^*, M_0, \{K_i\}_{i=1}^7)$.
\end{proposition}
\begin{proof}
    The proof proceeds as follows.  First by specifying the dependence on the parameters $(\zeta^*,\{K_i\}_{i=1}^{\nu},M_0)$, we exploit the energy estimates to show that the remainder $\eps$ will always be trapped in the bootstrap regime, given sufficiently small initial data. Then, as for the modulation parameters, the main part is to apply a topological argument to show the existence of an initial $a(0)$ such that the parameter $a(\tau)$ will remain trapped in the bootstrap regime for all time. The rest part ($\nu$ and $R_\tau/\mu$) follows directly from the $|\mod_i|$ estimates and time integration.\\
    \underline{Trapping $\eps$:} Suppose $w$ is a solution trapped in some bootstrap regime $\BS(\tau_0,\tau_*, \zeta^*, M_0, \{K_i\}_{i=1}^7)$. Since the parameter $M_0$ is chosen at last to make $C(\zeta^*,K_i)/|\log\nu|$ arbitrarily small, it suffices to keep track of only the leading-order terms in the energy estimates. In the following, when we say ``$M_0$" is large enough, it means $M_0$ is chosen large depending on $\zeta^*$ and all $K_i$'s. First, choose $\zeta^*\gg C$ and $K_7\gg \zeta^*C(\zeta^*,K_1,K_4,K_5,K_6)$ for the constants $C(\zeta^*,K_1,K_4,K_5,K_6)$ in \eqref{L infty control: specific result in Lemma}. Then, for small enough initial $\eps(0)$ (e.g., $\|\eps(0)(1+\zeta^\frac{3}{2})\|_{L^\infty(\zeta\geq\zeta^*)}+\|\eps(0)\|_{L^\infty(\zeta\geq\frac{1}{2}\zeta^*)}\lesssim \nu^3(0)$ suffices) that is even in $z$-variable and satisfies the orthogonality conditions \eqref{local orthorgonality conditions} and large enough $M_0$, by \eqref{L infty control: specific result in Lemma}, we have
    \begin{equation}\label{main proposition: far field closed}
        \|\eps(\tau)(1+\zeta)^\frac{3}{2}\|_{L^\infty(\zeta\geq\zeta^*)}\leq \frac{2CK_7}{\zeta^*}\frac{e^{-2\sqrt{\beta\tau+M_0}}}{\sqrt{\beta\tau+M_0}} \leq \frac{K_7e^{-2\sqrt{\beta\tau+M_0}}}{2\sqrt{\beta\tau+M_0}},\quad \forall\,\tau\in[0,\tau_*].
    \end{equation}
    As for the middle range $H^2$-estimate, by \eqref{lemma diff ineq: ineq}, $|\mod_0|$ estimate \eqref{Lemma modulation eq: inequalities} and Lemma \ref{lemma: H^2 control in the middle} (taking $\zeta_1=\frac{1}{4}\zeta^*$ and $\zeta_2=2\zeta^*$), we have
    \begin{equation}\label{main proposition: middle H^2 closed}
        \|\eps\|_{H^2(\frac{1}{4}\zeta^*\leq\zeta\leq2\zeta^*)}\leq C(\zeta^*)\|\eps\|_{H^2_*(\frac{1}{4}\zeta^*,2\zeta^*)}\leq \frac{C(\zeta^*)(K_4+\sqrt{K_4K_5})\nu^2}{|\log\nu|}\leq \frac{K_6\nu^2}{2|\log\nu|}\quad \forall\,\tau\in[0,\tau_*],
    \end{equation}
    once we take $K_6^2\gg C(\zeta^*)(K^2_4+K_4K_5)$ and $\eps(0)$ to be small enough. For the $H^1$ estimate, define
    \[
        \|\eps\|^2_{H^1_{\nu}} := K^2\dr{\frac{1}{2}\int\eps\chi_\nu\sqrt{\vr_\nu}\Ms^\zeta_\nu(\eps\chi_\nu\sqrt{\vr_\nu})-\frac{d_0}{2}\int\eps\sqrt{\vr_\nu}\Ms^\zeta_\nu(\vp_{0,\nu}\sqrt{\vr_\nu}\chi_\nu)}+\frac{1}{2}\int U_\nu|\nabla\Ms^\zeta_\nu(\eps^*)|^2, 
    \]
    where $K\gg1$ is to be determined. Then, by Lemma \ref{lemma: L^2 energy estimate}, Lemma \ref{lemma: H^1 energy estimate} ,\eqref{main proposition: far field closed}, and the equivalence of norms, we have 
    \begin{align*}
        \frac{d}{d\tau}\|\eps\|^2_{H^1_\nu}&\leq -\frac{(\delta_0K^2-C)}{\nu^2}\dr{\|\eps\|^2_\inn+\|\nabla\eps\|^2_\inn}+\frac{CK^2\nu^2}{|\log\nu|^2}+\frac{C(K_1)K^2_7}{(\zeta^*)^2}\frac{\nu^2}{|\log\nu|^2}\\
        &\leq  -\frac{(\delta_0K^2-C)}{C(\zeta^*)}\|\eps\|^2_{H^1_\nu}+\frac{CK^2\nu^2}{|\log\nu|^2}+\frac{C(K_1)K^2_7}{(\zeta^*)^2}\frac{\nu^2}{|\log\nu|^2}.
    \end{align*}
    Thus, we can choose $K^2 = C(K_1)K^2_7/K^2_4\gg 1$, and by \eqref{lemma diff ineq: ineq} we obtain
    \[
       \frac{1}{\nu^2}\dr{K^2\|\eps\|^2_\inn+\|\nabla\eps^*\|^2_{L^2(U_\nu)}} \leq C\|\eps\|^2_{H^1_\nu}\leq \frac{C(\zeta^*)\nu^2}{\delta_0|\log\nu|^2}+\frac{C(\zeta^*)K^2_4\nu^2}{\delta_0|\log\nu|^2}\quad\forall\,\tau\in[0,\tau_*], 
    \]
    when $M_0$ is sufficiently large and $\eps(0)$ is sufficiently small. Thus, choosing $K^2_5\gg \frac{C(\zeta^*)K^2_4}{\delta_0}$ and $K^2_4\gg \frac{C(\zeta^*)}{\delta_0}$ above, we obtain
    \begin{align}\label{main prop: H^1 closed}
        \|\eps\|^2_\inn \leq\frac{K^2_4\nu^4}{2|\log\nu|^2},\quad \|\nabla\eps^*\|^2_{L^2(U_\nu)}\leq \frac{K^2_5\nu^4}{2|\log\nu|^2},\quad\forall\,\tau\in[0,\tau_*].
    \end{align}
    Collecting \eqref{main proposition: far field closed}, \eqref{main proposition: middle H^2 closed} and \eqref{main prop: H^1 closed}, we see that all the bootstrap constants are improved by a factor of $\frac{1}{2}$. As a summary, it is helpful to recap the dependence of the parameters:
    \begin{align}\label{dependence on parameters 1}
    \begin{split}
         K_7\gg C(\zeta^*,K_1,K_4,K_5,K_6)\Longrightarrow K_5\gg C(\zeta^*)K_6\Longrightarrow K_6\gg C(\zeta^*)\sqrt{K_4K_5}\Longrightarrow K_5\gg\frac{C(\zeta^*)}{\sqrt{\delta_0}}K_4\\
         \Longrightarrow K_4\gg \frac{C(\zeta^*)}{\sqrt{\delta_0}}\Longrightarrow \zeta^*\gg C.
    \end{split}
    \end{align}
    \underline{Trapping modulation parameters:} First, simply choosing $K_3\gg C(K_4,K_5)$ in \eqref{Lemma modulation eq: inequalities}, we have 
    \[\da{\frac{R_\tau}{\mu}}\leq \frac{C(K_4,K_5)\nu}{|\log\nu|}\leq \frac{K_3\nu}{2|\log\nu|}.
    \]
    Denote $\ta:= a-8\nu^2$, and the bootstrap assumption on $a$ is equivalent to $|\ta|\leq \frac{K_2\nu^2}{|\log\nu|}$. Then, inserting the decomposition into the $\mod_1$-equation gives
    \[
        \da{\frac{\nu_\tau}{\nu}-\frac{\beta}{2\log\nu}+\frac{\ta_\tau}{16\nu^2}}\leq \frac{C(K_2,K_4,K_5)}{|\log\nu|^2}.
    \]
    It follows that
    \[
        \frac{d}{d\tau}(\log^2\nu) = \beta - \frac{\ta_\tau\log\nu}{8\nu^2}+\Oc\dr{\frac{C(K_2,K_4,K_5)}{|\log\nu|}}.
    \]
    Integrating in $[0,\tau]$, we have
    \begin{equation}\label{main prop: compute log(nu)^2}
        (\log\nu(\tau))^2 = (\log\nu(0))^2+\beta\tau-\int_{0}^{\tau}\frac{\ta_\tau(s)\log\nu(s)}{8\nu^2(s)}\;ds+\Oc\dr{\int_{0}^{\tau}\frac{C(K_2,K_4,K_5)}{|\log\nu|}\;ds}.
    \end{equation}
    Note, through integration by parts, that
    \begin{align*}
      \int_{0}^{\tau}\frac{\ta_\tau(s)\log\nu(s)}{8\nu^2(s)}\;ds &= \frac{\ta(\tau)\log\nu(\tau)}{8\nu^2(\tau)} -\frac{\ta(0)\log\nu(0)}{8\nu^2(0)}-\int_{0}^{\tau}\ta(s)\dr{\frac{\log\nu(s)}{8\nu^2(s)}}'\;ds\\
      &= \Oc(K_2+C(K_2,K_4,K_5)\sqrt{\beta\tau}).
    \end{align*}
    where we use \eqref{cor of mod est: est of nu_tau} and
    \[
        \int_{0}^{\tau}\frac{1}{|\log\nu(s)|}\;ds\lesssim \sqrt{\beta\tau+M_0}-\sqrt{M_0}\lesssim \sqrt{\beta\tau},
    \]
    when $M_0$ is large enough. Therefore, once we choose $K_1\gg C(K_2,K_4,K_5)$, by \eqref{main prop: compute log(nu)^2} and the estimate above we obtain
    \begin{equation*}\label{main prop: nu closed}
        \frac{2}{K_1}e^{-\sqrt{\beta\tau+M_0}} \leq e^{-\sqrt{\beta\tau+M_0+\Oc(C(K_2,K_4,K_5)(1+\sqrt{\beta\tau}))}} \leq \frac{K_1}{2}e^{-\sqrt{\beta\tau+M_0}}.
    \end{equation*}
    Finally, we control $a(\tau)$. By the $|\mod_0|$-estimate in \eqref{Lemma modulation eq: inequalities} and the decomposition $a = 8\nu^2+\ta$, we have
    \begin{equation}\label{main prop: unstable ode for a}
        |\ta_\tau-2\beta\ta|\leq \frac{C(K_4,K_5)\nu^2}{|\log\nu|}.
    \end{equation}
    We choose $K_2\gg C(K_4,K_5)$, and claim that there exists an initial $\ta(0)\in [-\frac{K_2\nu^2(0)}{|\log\nu(0)|},\frac{K_2\nu^2(0)}{|\log\nu(0)|}]$, such that $a$ will be trapped for all time. Before proving this claim, we summarize the dependence of these parameters:
    \begin{equation}\label{dependence on parameters 2}
        K_1\gg C(K_2,K_4,K_5)\Longrightarrow K_2\gg C(K_4,K_5),\; K_3\gg C(K_4,K_5).
    \end{equation}
    Combining \eqref{dependence on parameters 1} and \eqref{dependence on parameters 2}, it is clear that there exist parameters $\zeta^*$ and $\{K_i\}_{i=1}^7$ that satisfy all these constraints.
    Now we prove the claim. First, we fix the parameters $(\zeta^*,\{K_i\}_{i=1}^{7})$ and other initial conditions according to the aforementioned discussion. Then, suppose, for contradiction, that for any $a_0:=\ta(0)\in [-\frac{K_2\nu^2(0)}{|\log\nu(0)|},\frac{K_2\nu^2(0)}{|\log\nu(0)|}]$, the corresponding solution $w_{\ta}$ will exit the bootstrap regime in finite time. We denote the supremum of time that $w_{a_0}$ stays in the bootstrap regime by $\tau_{a_0}$. Note that when $\tau=\tau_{a_0}$, since all other bootstrap constants are improved by $\frac{1}{2}$, we have $|\ta(\tau_{a_0})| = \frac{K_2\nu^2(\tau_{a_0})}{|\log\nu(\tau_{a_0})|}$. Denote $I:= [-K_2,K_2]$, then we obtain a map $\psi: I\to \{-K_2,K_2\}$, defined as 
    \[
    \psi: \frac{a_0|\log\nu(0)|}{\nu^2(0)}\mapsto \frac{\ta(\tau_{a_0})|\log\nu(\tau_{a_0})|}{\nu^2(\tau_{a_0})}.
    \] 
    First, this map is continuous due to the standard continuous dependence of the solutions on the initial data. Second, if $|\ta(\tau)| = \frac{K_2\nu^2(\tau)}{|\log\nu(\tau)|}$ for any time $\tau$, by \eqref{main prop: unstable ode for a}, we have
    \[
        \ta_\tau(\tau) = 2\beta\ta(\tau)+\Oc\dr{\frac{C(K_4,K_5)\nu^2(\tau)}{|\log\nu(\tau)|}}.
    \]
    Thus, $\ta(\tau)$ is nonzero and have the same sign as $\ta(\tau)$ (since we choose $K_2\gg C(K_4,K_5)$). Then, for any $\tau'>\tau$ (and $|\tau'-\tau|$ suitably small), we will have $|\ta(\tau')|>\frac{K_2\nu^2(\tau')}{|\log\nu(\tau')|}$, existing the bootstrap regime. As a consequence, $\psi$ is the identity map when restricting on $\pa I = \{-K_2,K_2\}$. However, now since $I = \psi^{-1}(\{K_2\})\cap \psi^{-1}(\{-K_2\})$ which is the union of two disjoint nonempty (as $\psi(\pm K_2)=\pm K_2$) open sets, this gives a contradiction as $I$ is topologically connected (this is in fact a special case of the Brouwer fixed point theorem). 
\end{proof}
Theorem \ref{main theorem} is directly implied by Proposition \ref{proposition: main}, except for the part $\|\tilde u(t)\|_\Ec\to 0$ as $t\to T$. This is because the inner region (which is controlled by an $H^1$ norm) and outer region (which is controlled by an $L^\infty$ norm) for the perturbation $\eps$ in the bootstrap assumptions is divided at the parabolic scale ($\zeta\sim 1$), while the definition of $\|\cdot\|_\Ec$ makes such division in the soliton scale ($\zeta\sim \nu$). However, the former can easily imply the latter. To see this, note that by Lemma \ref{lemma: H^2 control in the middle} (and the proof therein) and Lemma \ref{lemma diff ineq}, we have $\|\eps\|_{H^2(\nu\leq \zeta\leq \zeta^*)}=\Oc(|\log\nu|^{-1})$. Thus, by the Sobolev embedding $H^2\hookrightarrow L^\infty$ in $2$D, we have $\|\eps\|_{L^\infty(\nu\leq \zeta\leq \zeta^*)}=\Oc(|\log\nu|^{-1})$. Combining with the far field $L^\infty$ control of $\eps$, we obtain $\|\eps(1+\zeta)^\frac{3}{2}\|_{L^\infty(\zeta\geq \nu)}=\Oc(|\log\nu|^{-1})$. It follows that $\|q(1+\gamma)^\frac{3}{2}\|_{L^\infty(\gamma\geq 1)}=\Oc(\sqrt{\nu}/|\log\nu|)\to 0$ as $t\to T$ (we recall $q(\rho,\xi,t) = \nu^2\eps(\nu\rho,\nu\xi,t)$), so that we actually have the weighted $L^\infty$ control of $\tilde u$ (the notation in Theorem \ref{main theorem}) from the soliton scale, i.e., $\gamma\geq 1$. This completes the proof of Theorem \ref{main theorem}.

\appendix
\section{Appendix: Estimates and inequalities}

\subsection*{Estimates of the Poisson Fields}
On $\reall^3$, we denote the cylindrical coordinate by $(r,\theta,z)$, and denote $\tr := r-R/\mu$ where we assume $0<\mu\ll1$ to be a small parameter tending to zero. Let $u = u(\tr,z)$ be an axisymmetric function centered around the ring with radius $R/\mu$ on the plane $z=0$. Let $\Phi_u =\frac{1}{4\pi|\x|}*u$ be its $3$D Poisson field. Then, we are to define a $2$D Poisson field as an approximation to the $3$D one in a sense that will soon become clear. Define a function on $(\tr,z)\in\reall^2$ by
\[
    \bar u(\tr,z) :=
    \begin{cases}
        u(\tr,z),\quad &\tr\geq -R/\mu,\\
        u(-2R/\mu-\tr,z),\quad &\tr<-R/\mu.
    \end{cases}
\]
Let $\eta(x)$ be a smooth cutoff function in $1$D, such that
\[
   0\leq \eta(x)\leq1,\quad \eta(x)\equiv 1\quad\forall\;x\in[0,+\infty),\quad\text{and }\eta(x)\equiv 0\quad\forall\;x\in (-\infty,-1].
\]
Now, define $u_*(\tr,z) := \eta(r)\bar u(\tr,z)$, which is a smooth extension of $u$ to $\reall^2$, and
\begin{equation}\label{Appendix: def of modified 2D Poisson field}
    \Psi_u:= -\frac{1}{2\pi}\log|(\tr,z)|*u_*.
\end{equation}
One remark is that the choice of definition for the $2$D Poisson field of $u$ is not unique. In the following analysis, we will see that what matters is that $\Psi_u$ solves the Poisson equation on the half plane:
 \[
     -(\pa^2_{\tr}+\pa^2_z)\Psi(\tr,z) = u(\tr,z),\quad\forall\; (\tr,z)\in [-R/\mu,+\infty)\times\reall,
 \]
and $\nabla\Phi_u$ has certain decay property. In general, it works if we extend $u$ to a function on the whole $\reall^2$ with suitably small modification, and then consider its convolution with the $2$D fundamental solution to the Laplace equation. In our case, since $u_*$ and $u$ differ little in $L^p$ norms when viewed as $\reall^2$-functions (with an extension by zero for $u$) estimates of $\Psi_u$ follow from those of $u$. The following lemma illustrates that the difference between $\nabla\Psi_u(\tr,z)$ and $\nabla\Phi_u(\tr,z)$ can be controlled pointwisely when $u(\br,z)$ satisfies certain decay property.   

\begin{lemma}[Difference of 2D and 3D Poisson fields]\label{appendix: difference of 2d 3d Poisson field}
Assume $u(r,z)$ to be a function with suitable decay in  $\reall^2$ and denote its $2$D Poisson field by $\Psi_u:=-\frac{1}{2\pi}\log|(r,z)|*u$. We can also interpret $u(r,z),\;(r,z)\in \H:=\{(r,z):r>0\}$, to be an axisymmetric function in $\reall^3$, in which case we can define its $3$D Poisson field by $\Phi_u: = \frac{1}{4\pi|\x|}*u\;( \x\in\reall^3 )$. Denote by $B_{(x,y)}(l)\subset\reall^2$ the ball centered at $(x,y)$ with radius $l>0$. Assume the following decay property of $u(r,z)$:
\begin{align}\label{appendix: diff 23D Poisson, control of u}
\begin{split}
    &\|u(r,z)(1+(r-R/\mu)^2+z^2)^\frac{3}{4}\|_{L^\infty(\reall^2\backslash B_{(R/\mu,0))}(\zeta^*)}\leq L_\infty,\\
    &\|u(r,z)\|_{L^2(B_{(R/\mu,0)}(\zeta^*))} \leq L_2, \quad \|\nabla u(r,z)\|_{L^2(B_{(R/\mu,0)}(\zeta^*))} \leq L'_2.
\end{split}
\end{align}
where $\zeta^*>0$ is fixed, and $L_\infty,\;L_2,\;L_2'>0$ are some constants. Then, there exists $\mu^*>0$ such that for any $0<\mu<\mu^*$, it holds that given any $\frac{6}{7}<s<1$ we have the following estimates on the gradients of Poisson fields:\\
(\romannumeral 1) \textup{(Near field approximation)}
\begin{align}\label{appendix: 3&2D near field approx}
|\nabla\Psi_u(r,z) - \nabla\Phi_u(r,z)|\leq C_1(L_\infty+L_2+L'_2)\mu^\kappa,\quad\forall\; (r,z)\in B_{(R/\mu,0)}(\mu^{-s}),
\end{align}
(\romannumeral 2) \textup{(Far-field control)}
\begin{align}\label{appendix: 3&2D far field control}
    |\nabla\Phi_u(r,z)|\leq C_2(L_\infty+L_2)\mu^\kappa,\quad\forall\; (r,z)\in \H\backslash B_{(R/\mu,0)}(\mu^{-s}).
\end{align}
Here $\kappa>0$, and $ \kappa, C_1, C_2$ are all universal constants depending only on $s$.
\begin{proof}
The strategy of the proof is to exploit the explicit expressions of the Poisson fields, especially its decay behaviors in the far field. A key observation is that: away from the axis of symmetry the fundamental solution of the $3$D axisymmetric Poisson equation is asymptotically close to that of the $2$D Poisson equation. We also note that, by our construction, $u_*$ shares the same control (up to a universal constant) of $u$ as in \eqref{appendix: diff 23D Poisson, control of u}. \\
\noindent\underline{\textit{Proof of (\romannumeral1)}.}
First, the Poisson fields are expressed as
\begin{equation}
    \nabla\Psi_u(r,z) = \int_{\reall^2}\nabla E_2(r,z,\tilde r, \tilde z)u(\tr,\tz)\;d\tr d\tz,\quad  \nabla\Phi_u(r,z) = \int_{\H}\nabla E_3(r,z,\tilde r, \tilde z)u(\tr,\tz)\;d\tr d\tz,
\end{equation}
where
\begin{align}\label{appendix: explicit expression of Poisson fields}
\begin{split}
    &\nabla E_2(r,z,\tr,\tz) = -\frac{1}{2\pi}\frac{(r-\tr,z-\tz)}{(r-\tr)^2+(z-\tz)^2},\\
    &\nabla E_3(r,z,\tr,\tz) = -\frac{\tr}{2\pi}\int_{0}^{\pi}\frac{(r-\tilde r\cos(\theta),z-\tz)}{(r^2-2r\tr\cos(\theta)+\tr^2+(z-\tz)^2)^{\frac{3}{2}}}\;d\theta.
\end{split}
\end{align}
From \eqref{appendix: explicit expression of Poisson fields}, one can get the decay property of the Poisson fields:
\[
    \da{\nabla E_2(r,z,\tr,\tz)}\lesssim \frac{1}{((r-\tr)^2+(z-\tz)^2)^\frac{1}{2}},\quad \da{\nabla E_3(r,z,\tr,\tz)}\lesssim \frac{\tr}{(r-\tr)^2+(z-\tz)^2},
\]
 when $(r-\tr)^2+(z-\tz)^2\gg1$. In the following estimates of integrals, we denote $d: = \sqrt{(r-\tr)^2+(z-\tz^2)}$ and $\tilde d:= \sqrt{(\tr-R/\mu)^2+\tz^2}$ for brevity. We always bear in mind that $\mu\to 0$, which is much smaller than any fixed constant. Now, we assume $(r,z)\in B_{(R/\mu,0)}(\mu^{-s})$. By the expansion of elliptic integrals, one can show that (for example, see \cite{Chaabi2015} for a related discussion) for any $(\tr,\tz)\in B_{(r,z)}(2\mu^{-s})$,
\begin{equation}
    \nabla E_3(r,z,\tr,\tz) = \sqrt{\frac{\tr}{r}}\nabla E_2(r,z,\tr,\tz)+\Oc\dr{\sqrt{\frac{\tr}{r^3}}\log(d)}=\dr{1+\Oc(\mu^{1-s})}\nabla E_2(r,z,\tr,\tz)+\Oc\dr{\mu\log(d)}.
\end{equation}
Thus, by Hardy's inequality and Cauchy's inequality,
\begin{align}\label{appendix: inner ball}
    &\quad\da{\int_{B_{(r,z)}(2\mu^{-s})}\nabla E_3(r,z,\tr,\tz)u(\tr,\tz)-\nabla E_2(r,z,\tr,\tz)u(\tr,\tz)\;d\tr d\tz}\nonumber\\
    &\lesssim \da{\int_{B_{(r,z)}(2\mu^{-s})}\dr{1-\sqrt{\frac{\tr}{r}}}\nabla E_2(r,z,\tr,\tz) u(\tr,\tz)\;d\tr d\tz}+\mu\int_{B_{(r,z)}(2\mu^{-s})}|\log(d)u(\tr,\tz)|\;d\tr d\tz \nonumber\\
    &\lesssim \mu^{1-s}\int_{B_{(r,z)}(2\mu^{-s})}\frac{|u(\tr,\tz)|}{d}\;d\tr d\tz + \mu\|\log(d)\|_{L^2(B_{(r,z)}(2\mu^{-s}))}\cdot\|u\|_{L^2(B_{(r,z)}(2\mu^{-s}))}\nonumber\\
    &\lesssim \mu^{1-s}L'_2+\mu^{1-s}\int_{B_{(r,z)}(2\mu^{-s})\backslash B_{(r,z)}(\zeta^*)}\frac{|u(\tr,\tz)|}{d}\;d\tr d\tz + \mu^{1-s}|\log(\mu)|\cdot\|u\|_{L^2(B_{(r,z)}(2\mu^{-s}))}\nonumber\\
    &\lesssim \mu^{1-s}L'_2+\mu^{1-s}\|d^{-1}\|_{L^2(\zeta^*\leq d\leq 2\mu^{-s})}\cdot  \|u\|_{L^2(\zeta^*\leq d\leq 2\mu^{-s})}+ \mu^{1-s}|\log(\mu)|\cdot\|u\|_{L^2(B_{(r,z)}(2\mu^{-s}))}\nonumber\\
    &\lesssim \mu^{1-s}L'_2+\mu^{1-s}|\log(\mu)|\dr{L_2+L_\infty}.
\end{align}
It remains to estimate the tails of the integrals. Observe that $d/2\leq\tilde d\leq 2d$ for any $(\tr,\tz)\in \H\backslash B_{(r,z)}(2\mu^{-s})$. By the decay property of $u,u_*$ and the fundamental solutions, we have
\begin{align}\label{appendix: tail estimate 1}
    \da{\int_{\H\backslash B_{(r,z)}(2\mu^{-s})}\nabla E_3(r,z,\tr,\tz)u(\tr,\tz)\;d\tr d\tz}\lesssim L_\infty\int_{\H\backslash B_{(r,z)}(2\mu^{-s})}\frac{d+R/\mu}{d^{2+\frac{3}{2}}}\;d\tr d\tz \lesssim L_\infty(\mu^{\frac{s}{2}}+\mu^{\frac{3}{2}s-1}).
\end{align}
and
\begin{align}\label{appendix: tail estimate 2}
    \quad\da{\int_{\reall^2\backslash B_{(r,z)}(2\mu^{-s})}\nabla E_2(r,z,\tr,\tz)u_*(\tr,\tz)\;d\tr d\tz}
    \lesssim L_\infty\int_{\reall^2\backslash B_{(r,z)}(2\mu^{-s})}\frac{1}{d^{\frac{5}{2}}}\;d\tr d\tz\lesssim L_\infty\mu^{\frac{s}{2}}.
\end{align}
Combining \eqref{appendix: inner ball} \eqref{appendix: tail estimate 1} \eqref{appendix: tail estimate 2} gives (\romannumeral1).
\\[6pt]
\noindent\underline{\textit{Proof of (\romannumeral2)}.} Denoting $d_0 := \sqrt{(r-R/\mu)+z^2}-\frac{1}{2}\mu^{-s}\geq \frac{1}{2}\mu^{-s}$, we decompose the domain of integration into three parts:
\[
    D_1: = B_{(r,z)}(d_0) ,\quad D_2:=B_{(R/\mu,0)}(\mu^{-s}/2),\quad D_3 := \H\backslash(D_1\cup D_2).
\]
For the first part, by the decay rates of $u$ and $\nabla E_3$, we have the estimate
\begin{align*}
    &\da{\int_{D_1}\nabla E_3(r,z,\tr,\tz)u(\tr,\tz)\;d\tr d\tz}\leq\\ &\da{\int_{B_{(r,z)}(C)}\nabla E_3(r,z,\tr,\tz)u(\tr,\tz)\;d\tr d\tz}+\da{\int_{D_1\backslash B_{(r,z)}(C)}\nabla E_3(r,z,\tr,\tz)u(\tr,\tz)\;d\tr d\tz},
\end{align*}
where $C>0$ is some fixed constant. Since $|u(\tr,\tz)|\lesssim L_\infty d_0^{-\frac{3}{2}}$ on $B_{(r,z)(C)}$ and $d_0\geq \mu^{-s}/2$, we have
\begin{equation}
    \da{\int_{B_{(r,z)}(C)}\nabla E_3(r,z,\tr,\tz)u(\tr,\tz)\;d\tr d\tz}\lesssim L_\infty\frac{d_0+R/\mu}{d^{\frac{3}{2}}} \lesssim L_\infty\dr{\mu^{\frac{s}{2}}+\mu^{\frac{3}{2}s-1}}.
\end{equation}
As for the second integral, by the decay of $u$, we obtain
\begin{align*}
    \da{\int_{D_1\backslash B_{(r,z)}(C)}\nabla E_3(r,z,\tr,\tz)u(\tr,\tz)\;d\tr d\tz}&\lesssim \int_{D_1\backslash B_{(r,z)}(C)} \frac{\tilde r}{d^2}|u(\tr,\tz)|\;d\tr d\tz\\
    &\lesssim L_\infty\int_{D_1\backslash B_{(r,z)}(C)} \frac{\tilde r}{{\tilde d}^{\frac{3}{2}}}\cdot\frac{1}{d^2}\;d\tr d\tz\\
    &\lesssim L_\infty\int_{C}^{d_0}\int_{0}^{2\pi}\frac{d_0+R/\mu+y\cos(\theta)}{(d_0+y\cos(\theta)+\mu^{-s})^{\frac{3}{2}}}\cdot\frac{1}{y}\;d\theta dy.
\end{align*}
Fixing another constant $0<c<1$, when $\delta:=-\cos(\theta)>1-c$, we have the estimates:
\begin{align*}
    \int_{C}^{d_0}\frac{d_0-\delta y}{(d_0-\delta y+\mu^{-s})^{\frac{3}{2}}}\cdot\frac{1}{y}\;dy = \int_{C}^{\delta d_0}\frac{d_0-x}{(d_0-x+\mu^{-s})^{\frac{3}{2}}}\cdot\frac{1}{x}\;dx
    &\leq \int_{C}^{d_0}\frac{1}{(d_0-x+1)^{\frac{1}{2}}x}\;dx\\
    &\lesssim \frac{\log(d_0)}{\sqrt{d_0}}\lesssim \mu^{\frac{s}{2}}|\log(\mu)|,
\end{align*}
and by H\"older's inequality,
\begin{align*}
     \int_{C}^{d_0}\frac{R/\mu}{(d_0-\delta y+\mu^{-s})^{\frac{3}{2}}}\cdot\frac{1}{y}\;dy &= R\mu^{\frac{3}{2}s-1}\int_{C}^{\delta d_0}\frac{d_0-x}{(\mu^{s}(d_0-x)+1)^{\frac{3}{2}}}\cdot\frac{1}{x}\;dx\\
     &\leq R\mu^{\frac{3}{2}s-1}\dr{\int_{C}^{+\infty}\frac{1}{x^{\frac{3}{2}}}\;dx}^{\frac{2}{3}}\cdot\dr{\int_{0}^{+\infty}\frac{1}{(\mu^{s}x+1)^{\frac{9}{2}}}\;dx}^{\frac{1}{3}}
     \lesssim \mu^{\frac{7}{6}s-1}.
\end{align*}
When $\delta:=-\cos(\theta)\leq 1-c$, the estimate is more straightforward:
\begin{align*}
     \int_{C}^{d_0}\frac{d_0-\delta y+R/\mu}{(d_0-\delta y+\mu^{-s})^{\frac{3}{2}}}\cdot\frac{1}{y}\;dy \lesssim \frac{d_0+R/\mu}{(d_0+\mu^{-s})^{\frac{3}{2}}}\int_{C}^{d_0}\frac{1}{y}\;dy\leq \frac{d_0+R/\mu}{(d_0+\mu^{-s})^{\frac{3}{2}}}\log(d_0)\lesssim (\mu^{\frac{s}{2}}+\mu^{\frac{3}{2}s-1})|\log(\mu)|.
\end{align*}
In summary, we obtain
\begin{equation}\label{appendix: domain D_1}
    \da{\int_{D_1}\nabla E_3(r,z,\tr,\tz)u(\tr,\tz)\;d\tr d\tz}\lesssim L_\infty\dr{(\mu^{\frac{s}{2}}+\mu^{\frac{3}{2}s-1})|\log(\mu)|+\mu^{\frac{7}{6}s-1}}.
\end{equation}
For the second domain,
\begin{align}\label{appendix: domain D_2}
    \da{\int_{D_2}\nabla E_3(r,z,\tr,\tz)u(\tr,\tz)\;d\tr d\tz}\lesssim \int_{D_2}\frac{R}{\mu}u(\tr,\tz)\mu^{2s}\;d\tr d\tz\nonumber
    &\lesssim \mu^{2s-1}L_2 + L_\infty\mu^{2s-1}\int_{1}^{\mu^{-s}}\frac{1}{\sqrt{x}}\;dx\nonumber\\
    &\lesssim L_2\mu^{2s-1}+L_\infty\mu^{\frac{3}{2}s-1}.
\end{align}
Finally, for the third domain, observe that $\tilde d\leq 3d$ in this domain. Thus, we have the estimate
\begin{equation}\label{appendix: domain D_3}
     \da{\int_{D_3}\nabla E_3(r,z,\tr,\tz)u(\tr,\tz)\;d\tr d\tz}\lesssim L_\infty\int_{D_3} \frac{\tilde d+R/\mu}{\tilde d^{2+\frac{3}{2}}}\;d\tr d\tz\lesssim L_\infty\int_{\mu^{-s}}^{+\infty}\frac{x+R/\mu}{x^{\frac{5}{2}}}\;dx \lesssim L_\infty (\mu^{\frac{1}{2}s}+\mu^{\frac{3}{2}s-1}).
\end{equation}
Finally, collecting \eqref{appendix: domain D_1}\eqref{appendix: domain D_2}\eqref{appendix: domain D_3}, we obtain (\romannumeral2).

\end{proof}
\end{lemma}
Now we derive some useful estimates for the $2$D Poisson field. For convenience, in the following we will switch between the Cartesian coordinate $\y=(\rho,\xi)$ and the polar coordinate $(\gamma,\theta)$ from time to time, where $\rho = \gamma\sin\theta,\;\xi=\gamma\cos\theta$. The specific choice of coordinates will be clear from the context.  
\begin{lemma}[Pointwise estimates of the $2$D Poisson field]\label{appendix: pointwise est of 2D Poisson}
    The following pointwise estimates of $2$D Poisson fields hold:\\
    (\romannumeral1) For $u$ and its $2$D Poisson field $\Psi_u:= -\frac{1}{2\pi}\log|\y|*u$, we have for any $\alpha>0$,
    \begin{equation}\label{appendix: pointwise est of Poisson field by weighted L^2 norm}
        \|\Psi_u\|^2_{L^\infty(\gamma\leq 1)}+\dn{\frac{\Psi_u}{1+\log(\gamma)}}^2_{L^\infty(\gamma>1)} \leq C_\alpha\int_{\reall^2} u^2(\y)(1+\gamma)^{2+\alpha}\;d\y,
    \end{equation}
    with constant $C_\alpha>0$ depending on $\alpha$. Moreover, if $\int u = 0$, we have the improved estimate
    \begin{equation}\label{appendix: pointwise est of Poisson field by weighted L^2 norm, when int u = 0}
        \|\Psi_u\|^2_{L^\infty} \leq C_\alpha\int_{\reall^2} u^2(\y)(1+\gamma)^{2+\alpha}\;d\y.
    \end{equation}
    \\
    (\romannumeral2) If $u=u(\gamma)$ is a radial function on $\reall^2$, we have for $0\leq \alpha \leq 1$:
    \begin{equation}\label{appendix: pointwise est of Poisson field radial}
        |\gamma\pa_\gamma\Psi_u(\gamma)|^2 = \da{\int_{0}^{\gamma}ru(r)\;dr}^2\lesssim(1+\one_{\{\gamma\geq 1\}}\log(\gamma))\frac{\gamma^2}{(1+\gamma)^{2\alpha}}\int_{0}^{\gamma}ru^2(r)(1+r)^{2\alpha}\;dr.
    \end{equation}
    On the other hand, if $u$ is without radial component, then for any $0<\alpha<2$, we have (in the polar coordinates) 
    \begin{equation}\label{appendix: pointwise est of Poisson field nonradial}
    \int_{0}^{2\pi}|\Psi_u(\gamma,\theta)|^2+\gamma^2|\nabla\Psi_u(\gamma,\theta)|^2\;d\theta\lesssim \gamma^2(1+\gamma)^{-2\alpha}(1+\one_{\{\gamma\leq 1\}}|\log(\gamma)|)\int_{\reall^2} u^2(\y)(1+\gamma)^{2\alpha}\;d\y.
    \end{equation}
    It is also convenient to write equivalently in the parabolic variables ( $\zeta = \nu\gamma$ ) in our setting:
    \begin{align}\label{appendix: pointwise est of Poisson field radial and nonradial, parabolic}
    \begin{split}
        |\zeta\pa_\zeta\Psi_u(\zeta)|^2 \lesssim(1+\one_{\{\zeta\geq \nu\}}\log(\zeta/\nu))\frac{\zeta^2}{(\nu+\zeta)^{2\alpha}}\int_{\reall^2}u^2(\nu+\zeta)^{2\alpha}\;d\x\quad\text{ (radial) }&0\leq\alpha\leq 1,\\
       \int_{0}^{2\pi}|\Psi_u(\zeta,\theta)|^2+\zeta^2|\nabla\Psi_u(\zeta,\theta)|^2\;d\theta\lesssim \zeta^2(\nu+\zeta)^{-2\alpha}(1+\one_{\{\zeta\leq \nu\}}|\log(\zeta/\nu)|)\int_{\reall^2} &u^2(\x)(\nu+\zeta)^{2\alpha}\;d\x,\\
       \text{ (without radial component) } &0<\alpha<2. 
    \end{split}
    \end{align}
    (\romannumeral3) For any $1\leq p<2$, we have the following estimate based on the $L^\infty$-norm of $u$:
    \begin{equation}\label{appendix: pointwise est of the gradient of Poisson field based on L infty of u}
        \|\nabla\Psi_u\|_{L^\infty}\leq C(p) \dr{\|u\|_{L^\infty} + \|u\|_{L^p}}.
    \end{equation}
\end{lemma}
\begin{proof}
    See \textit{Lemma 7.2} in \cite{Raphael2014} for a proof of (\romannumeral1) and (\romannumeral3). As for (\romannumeral 2), \eqref{appendix: pointwise est of Poisson field radial}  follows directly from the 
    explicit expression of $\pa_\zeta\Psi_u$ when $u$ is radial and Cauchy's inequality. To prove \eqref{appendix: pointwise est of Poisson field nonradial}, we expand both $u$ and $\Psi_u$ into trigonometric series:
    \begin{align*}
        &u(\gamma,\theta) = \sum_{j=1}^{+\infty}u^{+,j}(\gamma)\sin(j\theta)+\sum_{j=1}^{+\infty}u^{-,j}(\gamma)\cos(j\theta),\\
        &\Psi_u(\gamma,\theta) = \sum_{j=1}^{+\infty}\Psi^{+,j}_u(\gamma)\sin(j\theta)+\sum_{j=1}^{+\infty}\Psi^{-,j}_u(\gamma)\cos(j\theta).
    \end{align*}
     By $-\Delta \Psi_u = u$, we have
     \[
        -\dr{\pa^2_\gamma+\frac{1}{\gamma}\pa_\gamma-\frac{j^2}{\gamma^2}}\Psi_u^{\pm,j}(\gamma) = u^{\pm,j}(\gamma)\quad j\geq 1,
     \]
     which admits explicit solutions
     \begin{align*}
         &\Psi_u^{\pm,j}(\gamma) = \frac{\gamma^j}{2j}\int_{\gamma}^{+\infty}u^{\pm,j}(y)y^{1-j}\;dy+\frac{\gamma^{-j}}{2j}\int^{\gamma}_{0}u^{\pm,j}(y)y^{1+j}\;dy,\\
         &\pa_\gamma\Psi_u^{\pm,j}(\gamma) = \frac{\gamma^{j-1}}{2}\int_{\gamma}^{+\infty}u^{\pm,j}(y)y^{1-j}\;dy-\frac{\gamma^{-j-1}}{2}\int^{\gamma}_{0}u^{\pm,j}(y)y^{1+j}\;dy.
     \end{align*}
     By Cauchy's inequality, we have
     \begin{align*}
         \da{\int_{\gamma}^{+\infty}u^{\pm,j}(y)y^{1-j}\;dy} &\lesssim \dr{\int_{\gamma}^{+\infty}y^{1-2j}(1+y)^{-2\alpha}\;dy}^\frac{1}{2}\dr{\int_\gamma^{+\infty}(u^{\pm,j}(y))^2(1+y)^{2\alpha}y\;dy}^\frac{1}{2}\\
         &\lesssim \gamma^{1-j}(1+\gamma)^{-\alpha}(1+\one_{\{\gamma\leq 1\}}|\log\gamma|)^\frac{1}{2}\dr{\int_\gamma^{+\infty}(u^{\pm,j}(y))^2(1+y)^{2\alpha}y\;dy}^\frac{1}{2},\\
       \end{align*} 
       and
       \begin{align*} \da{\int_{0}^{\gamma}u^{\pm,j}(y)y^{1+j}\;dy}& \lesssim \dr{\int_{0}^{\gamma}y^{1+2j}(1+y)^{-2\alpha}\;dy}^\frac{1}{2}\dr{\int^\gamma_{0}(u^{\pm,j}(y))^2(1+y)^{2\alpha}y\;dy}^\frac{1}{2}\\
       &\lesssim \gamma^{1+j}(1+\gamma)^{-\alpha}\dr{\int^\gamma_{0}(u^{\pm,j}(y))^2(1+y)^{2\alpha}y\;dy}^\frac{1}{2},
     \end{align*}
     where all constants of the inequalities above are independent of $j$. It follows that
     \begin{align*}
        & |\Psi_u^{\pm,j}(\gamma)|^2\lesssim \frac{1}{j^2}\gamma^2(1+\one_{\{\gamma\leq 1\}}|\log\gamma|)\int_{0}^{+\infty}(u^{\pm,j}(y))^2(1+y)^{2\alpha}y\;dy,\\
        & |\pa_\gamma\Psi_u^{\pm,j}(\gamma)|^2\lesssim(1+\one_{\{\gamma\leq 1\}}|\log\gamma|)\int_{0}^{+\infty}(u^{\pm,j}(y))^2(1+y)^{2\alpha}y\;dy.
     \end{align*}
     Finally, by Parseval's identity, we obtain
     \begin{align*}
         \int_{0}^{2\pi}|\Psi_u(\gamma,\theta)|^2+\gamma^2|\nabla\Psi_u(\gamma,\theta)|^2\;d\theta &\lesssim \sum_{\pm}\sum_{j\geq 1}j^2|\Psi_u^{\pm,j}(\gamma)|^2+\sum_{\pm}\sum_{j\geq 1}\gamma^2|\pa_\gamma\Psi_u^{\pm,j}(\gamma)|^2\\
         &\lesssim \gamma^2(1+\one_{\{\gamma\leq 1\}}|\log\gamma|)\int_0^{+\infty}\sum_{\pm}\sum_{j\geq 1}(u^{\pm,j}(y))^2(1+y)^{2\alpha}y\;dy\\
         &\lesssim \gamma^2(1+\one_{\{\gamma\leq 1\}}|\log\gamma|)\int_{\reall^2}|u(\y)|^2(1+|\y|)^{2\alpha}\;d\y.
     \end{align*}
     which is \eqref{appendix: pointwise est of Poisson field nonradial}.
\end{proof}

\subsection*{Inequalities}
Here is one Hardy-Poincar\'e type inequality:
\begin{lemma}
    Let $b>0$ be a small parameter. Then, there exists a universal constant $C>0$, such that for any function $\eps$ on $\reall^2$, it holds that ($|\y|:=\gamma,\;\y\in\reall^2$)
    \begin{equation}\label{appendix: one Hardy-Poincare type ineq in soliton}
        \int_{\reall^2} \dr{(1+\gamma)^2\eps^2+b^2\gamma^2(1+\gamma)^4\eps^2}e^{-\frac{b\gamma^2}{2}}\leq C\int_{\reall^2}\dr{b\eps^2+|\nabla\eps|^2}(1+\gamma)^4e^{-\frac{b\gamma^2}{2}}. 
    \end{equation}
    If we change of variable $\gamma:=\zeta/\sqrt{b}$, an equivalent form of this inequality is written as
    \begin{equation}\label{appendix: one Hardy-Poincare type ineq in parabolic}
        \int_{\reall^2} \dr{(\sqrt{b}+\zeta)^2\eps^2+\zeta^2(\sqrt{b}+\zeta)^4\eps^2}e^{-\frac{\zeta^2}{2}}\leq C\int_{\reall^2}\dr{\eps^2+|\nabla\eps|^2}(\sqrt{b}+\zeta)^4e^{-\frac{\zeta^2}{2}}. 
    \end{equation}
\end{lemma}
\begin{proof}
    It suffices to prove \eqref{appendix: one Hardy-Poincare type ineq in soliton}. Denote $\y = (y_1,y_2)$. Integrating by parts, we have for $i=1,2$,
    \[
       2\int(y_i+y^3_i)\eps\pa_i \eps e^{-\frac{b\gamma^2}{2}} = -\int(1+3y^2_i)\eps^2e^{-\frac{b\gamma^2}{2}}+\int b(y^2_i+3y^4_i)\eps^2e^{-\frac{b\gamma^2}{2}}.
    \]
    By Cauchy's inequality,
    \[
       \da{\int(y_i+y^3_i)\eps\pa_i \eps e^{-\frac{b\gamma^2}{2}}}\leq C\int (1+\gamma)^4|\nabla\eps|^2e^{-\frac{b\gamma^2}{2}} + \frac{1}{10}  \int(1+3y^2_i)\eps^2e^{-\frac{b\gamma^2}{2}},
    \]
    for some constant $C>0$. Thus, we have
    \[
         \int_{\reall^2} (1+\gamma^2)\eps^2e^{-\frac{b\gamma^2}{2}}\leq C\int_{\reall^2}\dr{b\eps^2+|\nabla\eps|^2}(1+\gamma)^4e^{-\frac{b\gamma^2}{2}}
    \]
    The other part is similar. Integrate by parts:
    \[
        2\int b(y_i+y_i^5)\eps\pa_i \eps e^{-\frac{b\gamma^2}{2}} = \int b^2(y_i^2+y_i^6)\eps^2e^{-\frac{b\gamma^2}{2}}-\int b(1+5y^4_i)\eps^2e^{-\frac{b\gamma^2}{2}}.
    \]
    Then, apply the Cauchy's inequality
    \[
       \da{\int b(y_i+y_i^5)\eps\pa_i \eps e^{-\frac{b\gamma^2}{2}}} \leq C\int (1+\gamma)^4|\nabla\eps|^2e^{-\frac{b\gamma^2}{2}}+\frac{1}{10} \int b^2(y_i^2+y_i^6)\eps^2e^{-\frac{b\gamma^2}{2}},
    \]
    and it follows that
    \[
       \int b^2\gamma^2(1+\gamma)^4\eps^2e^{-\frac{b\gamma^2}{2}} \leq C\int_{\reall^2}\dr{b\eps^2+|\nabla\eps|^2}(1+\gamma)^4e^{-\frac{b\gamma^2}{2}}.
    \]
    The proof is thus complete.
\end{proof}

We recall the Hardy-Littlewood-Sobolev (HLS) inequality in $\reall^n$: for $0<s<n$, $1<p<q<\infty$ with $\frac{1}{q}=\frac{1}{p}-\frac{s}{n}$, we have
\begin{equation*}
    \dn{\frac{1}{|\x|^{n-s}}*f}_{L^q}\leq C\|f\|_{L^p},
\end{equation*}
where $C=C(p)$.
Now combining the $2$D HLS ($p>2$) and H\"older inequality, we obtain:
\begin{align}\label{appendix: HLS ineq}
    \|\nabla \Psi_u\|_{L^p}\lesssim \dn{\frac{1}{|\x|}*u}_{L^p}\lesssim \|u\|_{L^{\frac{2p}{2+p}}}\lesssim \|u/\sqrt{W}\|_{L^2}\|W\|^\frac{1}{2}_{L^\frac{p}{2}},
\end{align}
for any weight function $W$.
A useful corollary in our setting is the following:
\begin{equation}\label{appendix: HLS ineq in our setting}
    \|\nabla\Psi_{\eps}\|_{L^p}\lesssim \|\eps\|_{L^{\frac{2p}{2+p}}}\lesssim \|\eps\nu/\sqrt{U_\nu}\|_{L^2}\|U_\nu/\nu^2\|^{\frac{1}{2}}_{L^{\frac{p}{2}}}\lesssim \nu^{-\frac{2p-2}{p}}\|\eps\nu/\sqrt{U_\nu}\|_{L^2}. 
\end{equation}

\vspace{0.2in}

\paragraph{Acknowledgments:} The research of P. Song and T. Y. Hou was in part supported by NSF Grant DMS-2205590 and the Choi Family Gift Fund. V. T. Nguyen is supported by the National Science and Technology Council of Taiwan.

\bibliographystyle{plain} 
\bibliography{ref} 

\end{document}